%% file: stacksfinal.tex
\documentclass[10pt]{amsart}

\pdfoutput=1

\usepackage{mathrsfs, amsmath, amsthm, amssymb}%
\usepackage[linktocpage]{hyperref}%
\usepackage[all]{xy}%
\xyoption{2cell}%
\UseAllTwocells%

\hypersetup{
 pdftitle={Characterizing categories of coherent sheaves},
 pdfauthor={Daniel Schaeppi},
 pdfkeywords={Tannaka duality} {algebraic stacks} {weakly Tannakian categories}
}

\DeclareMathOperator{\id}{id}
\DeclareMathOperator{\el}{el}
\DeclareMathOperator{\op}{op}
\DeclareMathOperator{\Lan}{Lan}
\DeclareMathOperator{\Mod}{\mathbf{Mod}}

\DeclareMathOperator{\fp}{fp}
\DeclareMathOperator{\Hom}{Hom}

\DeclareMathOperator{\Spec}{Spec}

\DeclareMathOperator{\Ind}{Ind}
\DeclareMathOperator{\Span}{\mathbf{Span}}
\DeclareMathOperator{\Cospan}{\mathbf{Cospan}}
\DeclareMathOperator{\EM}{EM}

\DeclareMathOperator{\Aff}{\mathbf{Aff}}

\DeclareMathOperator{\Coh}{\mathbf{Coh}}
\DeclareMathOperator{\QCoh}{\mathbf{QCoh}}

\DeclareMathOperator{\Cov}{Cov}
\DeclareMathOperator{\fpqc}{\mathit{fpqc}}

\DeclareMathOperator{\pr}{pr}

\DeclareMathOperator{\Fil}{Fil}
\DeclareMathOperator{\MF}{MF}
\DeclareMathOperator{\colim}{colim}
\DeclareMathOperator{\Er}{Er}

\DeclareMathOperator{\Ps}{\mathbf{Ps}}

\DeclareMathOperator{\Cocts}{\mathbf{Cocts}}
\DeclareMathOperator{\Lex}{\mathbf{Lex}}

\DeclareMathOperator{\Cat}{\mathbf{Cat}}
\DeclareMathOperator{\Gpd}{\mathbf{Gpd}}

\DeclareMathOperator{\Set}{\mathbf{Set}}
\DeclareMathOperator{\Ab}{\mathbf{Ab}}

\DeclareMathOperator{\Comod}{\mathbf{Comod}}

\newcommand{\ca}[1]{\mathscr{#1}}

\newcommand{\Prs}[1]{\mathcal{P}\ca{#1}}

\newcommand{\ten}[1]{\mathop{{\otimes}_{#1}}}

\newcommand{\tenlr}[2]{\mathop{{}_{#1}{\otimes}_{#2}}}

\newcommand{\pb}[1]{\mathop{{\times}_{#1}}}

\newcommand{\defl}{\mathrel{\mathop:}=}

\theoremstyle{plain}
\newtheorem{thm}{Theorem}[subsection]
\newtheorem{prop}[thm]{Proposition}
\newtheorem{lemma}[thm]{Lemma}
\newtheorem{cor}[thm]{Corollary}

\theoremstyle{definition}
\newtheorem{example}[thm]{Example}
\newtheorem{rmk}[thm]{Remark}
\newtheorem{dfn}[thm]{Definition}

\newtheoremstyle{citing}{}{}{\itshape}{}{\bfseries}{.}{ }{\thmnote{#3}}
\theoremstyle{citing}
\newtheorem{cit}{}

\newtheoremstyle{citingdfn}{}{}{}{}{\bfseries}{.}{ }{\thmnote{#3}}
\theoremstyle{citingdfn}


\author{Daniel Sch\"appi}
\title[Characterizing categories of coherent sheaves]{A characterization of categories of coherent sheaves of certain algebraic stacks}

\thanks{The author gratefully acknowledges support from a 2012 Endeavour Research Fellowship from the Australian Department of Education, Employment and Workplace Relations.}

\begin{document}

\begin{abstract}
 Under certain conditions, a scheme can be reconstructed from its category of quasi-coherent sheaves. The Tannakian reconstruction theorem provides another example where a geometric object can be reconstructed from an associated category, in this case the category of its finite dimensional representations. Lurie's result that the pseudofunctor which sends a geometric stack to its category of quasi-coherent sheaves is fully faithful provides a conceptual explanation for why this works.

 In this paper we prove a generalized Tannakian recognition theorem, in order to characterize a part of the image of the extension of the above pseudofunctor to algebraic stacks in the sense of Naumann. This allows us to further investigate a conjecture by Richard Pink about categories of filtered modules, which were defined by Fontaine and Laffaille to construct $p$-adic Galois representations.

 In order to do this we give a new characterization of Adams Hopf algebroids, which also allows us to answer a question posed by Mark Hovey.
\end{abstract}

\maketitle

\tableofcontents

\input{introduction}
\input{weaklytannakian}
\input{enriched}
\input{localization}
\input{stackslocalization}
\input{adams}
\input{filteredmodules}

\appendix
\input{superextensive}

\bibliographystyle{amsalpha}
\bibliography{stacksii}

\end{document}

%% file: introduction.tex
 \section{Introduction}

 The relationship between schemes and their categories of quasi-coherent sheaves has been studied extensively. Gabriel proved that a Noetherian scheme can be recontructed from its category of quasi-coherent sheaves \cite{GABRIEL}, a result which has been generalized by Rosenberg, Garkusha \cite{ROSENBERG, GARKUSHA}. Balmer \cite{BALMER} has shown that Noetherian schemes can be reconstructed from their derived category of perfect complexes. This was generalized to quasi-compact quasi-separated schemes by Buan-Krause-Solberg \cite{BKS}.

 Tannaka duality in algebraic geometry, developed by Saavedra, Deligne and Milne \cite{SAAVEDRA,DELIGNE_MILNE,DELIGNE}, gives another example where a geometric object (an affine group scheme, or a gerbe with affine band) can be reconstructed from an associated category, in this case the category of its finite dimensional representations.

 In \cite{LURIE}, Lurie has proved a theorem that suggests a common explanation for these phenomena. For geometric stacks he has identified those functors between categories of quasi-coherent sheaves which are induced by morphisms of stacks. A consequence of this is that a geometric stack is determined up to equivalence by its category of quasi-coherent sheaves. This explains why it is possible to reconstruct such stacks from their categories of sheaves.

 Our goal is to complement Lurie's result with a characterization of categories of quasi-coherent sheaves in terms of fiber functors. To do this we introduce the new notion of weakly Tannakian category, and prove that every such category is equivalent to the category of coherent sheaves on a stack.

\subsection{Weakly Tannakian categories and the recognition theorem}\label{section:introduction_weakly_tannakian}
 All the stacks we consider are stacks on the $\fpqc$-site of affine schemes. Following Naumann \cite{NAUMANN}, we call a stack \emph{algebraic} if it is associated to a flat affine groupoid $(X_0,X_1)$. Under the equivalence between affine schemes and commutative rings, such groupoids correspond to flat Hopf algebroids $(A,\Gamma)$. The category of quasi-coherent sheaves on the stack associated to $(X_0,X_1)$ is equivalent to the category of comodules of $(A,\Gamma)$ (see \cite[\S3.4]{NAUMANN} and \cite[Remark~2.39]{GOERSS}). This allows us to state all the results in the language of algebraic stacks even though the proofs are mostly written in the language of Hopf algebroids.

\begin{dfn}\label{dfn:weakly_tannakian}
 Let $R$ be a commutative ring, $B$ a commutative $R$-algebra, and let $\ca{A}$ be a symmetric monoidal abelian $R$-linear category. A strong symmetric monoidal $R$-linear functor
\[
 w \colon \ca{A} \rightarrow \Mod_B 
\]
 is called a \emph{fiber functor} if it is faithful and exact. A fiber functor is called \emph{neutral} if $B=R$.

 If $\ca{A}$ satisfies the conditions:
\begin{enumerate}
 \item[i)]
 There exists a fiber functor $w \colon \ca{A} \rightarrow \Mod_B$ for some commutative $R$-algebra $B$;
 \item[ii)]
 For all objects $A \in \ca{A}$ there exists an epimorphism $A^{\prime} \rightarrow A$ such that $A^{\prime}$ has a dual;

\end{enumerate}
 it is called \emph{weakly Tannakian}.
\end{dfn}

 We will show that a weakly Tannakian category $\ca{A}$ is equivalent (as symmetric monoidal $R$-linear category) to the category of coherent sheaves on an algebraic stack over $R$. The converse is not true in general. The following definitions allow us to describe those stacks $X$ for which $\Coh(X)$ is weakly Tannakian.
 
 A stack is called \emph{coherent} if it satisfies a mild finiteness condition (see Definition~\ref{dfn:coherent} for details). If $(A,\Gamma)$ is a commutative Hopf algebroid where $A$ is a coherent ring, then the associated stack $X$ is coherent.

 An algebraic stack $X$ has the \emph{resolution property} if every coherent sheaf is a quotient of a dualizable sheaf. It is hard to describe the precise class of those algebraic stacks which have the resolution property, but they are abundant. Thomason has studied conditions for a stack to have the resolution property, see \cite{THOMASON}. The proof of \cite[Lemma~2.4]{THOMASON} can be used to show that all algebraic stacks associated to flat Hopf algebroids $(A,\Gamma)$ where $A$ is a Dedekind ring have the resolution property. 

 A lot of algebraic stacks that arise in algebraic topology have the resolution property, for example, the moduli stack of formal groups, and all the stacks associated to the groupoids arising from a long list of cohomology theories (see \cite[\S1.4]{HOVEY}, \cite[Propositions~6.8 and 6.9]{GOERSS}).  The main theorem of our paper is the following.

\begin{thm}\label{thm:recognition}
 Let $R$ be a commutative ring. An $R$-linear category $\ca{A}$ is weakly Tannakian if and only if there exists a coherent algebraic stack $X$ over $R$ with the resolution property, and an equivalence
\[
 \ca{A} \simeq \Coh(X)
\]
 of symmetric monoidal $R$-linear categories.
\end{thm}

 Specializing to neutral fiber functors we obtain the following result, which can be phrased entirely in the language of flat affine group schemes and their categories of representations.

\begin{thm}\label{thm:affine_group_schemes_recognition}
 If a weakly Tannakian $R$-linear category $\ca{A}$ admits a neutral fiber functor, then it is equivalent to the category of finitely presentable representations of a flat affine group scheme over $R$.
\end{thm}

 Note that a neutral fiber functor can only exist if the ring $R$ is coherent. Otherwise the category of finitely presentable objects in the category of representations of an affine group scheme cannot be abelian (it contains the category of finitely presentable $R$-modules as the full subcategory of trivial representations).

 We prove these results in \S \ref{section:weaklytannakian}. As already shown in \cite{DAY_TANNAKA} and \cite{SCHAEPPI}, the Tannakian formalism works reasonably well for categories enriched in a \emph{cosmos}, that is, a complete and cocomplete symmetric monoidal closed category $\ca{V}$. In \S \ref{section:enriched} we introduce the notion of an \emph{enriched weakly Tannakian category} for a certain class of abelian cosmoi $\ca{V}$ (see Definition~\ref{dfn:enriched_weakly_tannakian_category}), and prove a recognition result for them (see Theorem~\ref{thm:enriched_recognition}). Examples of cosmoi for which this can be done are the cosmoi of graded or differential graded $R$-modules, and the cosmos of Mackey functors. In \S \ref{section:enriched} we assume that the reader has some familiarity with the basic theory of enriched categories, see \cite{KELLY_BASIC} and \cite{KELLY_FINLIM}. The rest of the paper does not depend on the contents of \S \ref{section:enriched}.

\subsection{A generalization of Lurie's embedding theorem for geometric stacks}

 In \cite{LURIE}, Lurie showed that for geometric stacks in the sense of loc.\ cit., the assignment which sends $X$ to $\QCoh(X)$ gives an embedding of the 2-category of geometric stacks into the 2-category $\ca{T}$ of symmetric monoidal abelian categories and tame functors (see Definition~\ref{dfn:tame}). In \S\ref{section:localization} we will show that the following slight generalization of Lurie's result also holds in the  context of algebraic stacks in the sense of Naumann. 

\begin{thm}\label{thm:stacks_embedding}
 Let $\ca{AS}$ be the 2-category of algebraic stacks, and let $\ca{T}$ be the 2-category of symmetric monoidal abelian categories, tame functors, and symmetric monoidal natural transformations. The pseudofunctor
\[
\QCoh(-) \colon \ca{AS}^{\op} \rightarrow \ca{T} 
\]
 which sends an algebraic stack to its category of quasi-coherent sheaves is an equivalence on hom-categories.
\end{thm}

 While we do not prove a version of this result for the enriched setting, we do try to reduce as much of the proof of Theorem~\ref{thm:stacks_embedding} as possible to categorical results which hold in this setting. We turn to a brief overview of the proof strategy.

 Instead of working with the category of algebraic stacks directly, we will show that both 2-categories satisfy the same (bicategorical) universal property with respect to the 2-category $\ca{H}$ of flat Hopf algebroids. More precisely, we show that both the pseudofunctor
\[
 L \colon \ca{H}^{\op} \rightarrow \ca{AS}
\]
 which sends a flat Hopf algebroid to its associated stack and the pseudofunctor
\[
 \Comod \colon \ca{H} \rightarrow \mathrm{ess.im}(\Comod) \subseteq \ca{T}
\]
 which sends a flat Hopf algebroid to its category of comodules form a bicategorical localization (in the sense of Pronk \cite{PRONK}) at the same class of morphisms. Theorem~\ref{thm:stacks_embedding} follows from the fact that bicategorical localizations are unique up to biequivalence.

 The fact that stacks associated to groupoids form a bicategorical localization at internal weak equivalences is of independent interest. We prove the following theorem in \S \ref{section:stackslocalization}.

\begin{thm}\label{thm:stacks_localization}
 The pseudofunctor $L$ which sends a flat affine groupoid $X$ to its associated $\fpqc$-stack exhibits the 2-category of algebraic stacks as a bicategorical localization of the 2-category of flat affine groupoids at the surjective weak equivalences (see Definition~\ref{dfn:surjective_weak_equivalence}).
\end{thm}

 This has been proved for several sites in \cite{PRONK}. Our proof suggests that this might be true for superextensive sites (see Definition~\ref{dfn:superextensive}) as long as the singleton coverings are effective descent morphisms. For topoi, results along those lines have been proved by Bunge \cite{BUNGE}. Joyal-Tierney \cite{JOYAL_TIERNEY} discuss this using the language of model categories.

 Roberts has shown that the localization at internal weak equivalences can also be described in terms of anafunctors, see \cite{ROBERTS} and \cite{ROBERTS_ANAFUNCTORS}.

\subsection{Adams Hopf algebroids and Adams stacks}

 In general it is not easy to determine whether or not a given functor is tame in the sense of \cite{LURIE}. Therefore it is convenient to know examples of stacks $X$ for which \emph{all} strong symmetric monoidal left adjoints with domain $\QCoh(X)$ are tame. We will show that this is the case for Adams stacks.

 A Hopf algebroid $(A,\Gamma)$ is called an \emph{Adams Hopf algebroid} if $\Gamma$, considered as an $(A,\Gamma)$-comodule, is a filtered colimit of dualizable comodules. An algebraic stack associated to an Adams Hopf algebroid is called an \emph{Adams stack} (see \cite[Definition~6.5]{GOERSS}). The following theorem gives a new characterization of Adams Hopf algebroids, hence of Adams stacks. Its proof is independent of the rest of the paper. We say that an algebraic stack has the \emph{strong resolution property} if the dualizable quasi-coherent sheaves form a generator of the category of quasi-coherent sheaves. Every algebraic stack with the strong resolution property has the resolution property, and the two notions coincide for coherent algebraic stacks (see Remark~\ref{rmk:strong_resolution}). In \S \ref{section:adams}, we will prove the following theorems.

\begin{thm}\label{thm:adams_iff_resolution_property}
 A flat Hopf algebroid $(A,\Gamma)$ is an Adams Hopf algebroid if and only if the dualizable comodules form a generator of $\Comod(A,\Gamma)$. Thus an algebraic stack is an Adams stack if and only if it has the strong resolution property.
\end{thm}

 Since weakly equivalent Hopf algebroids have equivalent categories of comodules, the above result in particular implies that Adams Hopf algebroids are stable under weak equivalences. This gives a positive answer to a question posed by Hovey \cite[Question~1.4.12]{HOVEY}.

\begin{thm}\label{thm:adams_implies_tame}
 Let $(A,\Gamma)$ be an Adams Hopf algebroid, and let $(B,\Sigma)$ be a flat Hopf algebroid. Then all strong symmetric monoidal left adjoint functors
\[
F \colon \Comod(A,\Gamma) \rightarrow \Comod(B,\Sigma)
\]
 are tame.
\end{thm}

 Translated into the language of algebraic stacks, this shows that all strong symmetric monoidal left adjoints
\[
 \QCoh(X) \rightarrow \QCoh(Y)
\]
 between categories of quasi-coherent sheaves of algebraic stacks $X$, $Y$ are tame if $X$ is an Adams stack.

 Brandenburg-Chirvasitu \cite{BRANDENBURG_CHIRVASITU} have studied this question for schemes. They show that any symmetric monoidal left adjoint whose domain is the category of quasi-coherent sheaves on a quasi-compact and quasi-separated scheme is tame. Note that their result is not a consequence of our result: Adams stacks have the resolution property, and not all such schemes are known to have the resolution property. This suggests that there might be a larger class of algebraic stacks for which all symmetric monoidal left adjoints are tame.

 Using Theorem~\ref{thm:adams_implies_tame} we can prove the following two theorems.

\begin{thm}\label{thm:adams_embedding}
 The assignment which sends an Adams stack $X$ to the category $\QCoh_{\fp}(X)$ of finitely presentable quasi-coherent sheaves on $X$ gives a pseudofunctor from the 2-category of Adams stacks to the 2-category $\ca{RM}$ of right exact symmetric monoidal additive categories (see Definition~\ref{dfn:right_exact_symmetric_monoidal}). This pseudofunctor is an equivalence on hom-categories.
\end{thm}

\begin{thm}\label{thm:coherent_adams_weakly_tannakian_equivalence}
 The pseudofunctor which sends an algebraic stack $X$ to the category $\Coh(X)$ of coherent modules of $X$ restricts to a contravariant biequivalence between the  2-category of coherent algebraic stacks with the resolution property and the 2-category of weakly Tannakian categories, right exact strong symmetric monoidal functors, and symmetric monoidal natural transformations.
\end{thm}

\subsection{An application to Fontaine-Laffaille filtered modules}

 Fix a perfect field $k$ of characteristic $p > 0$, let $W$ be its ring of Witt vectors, and let $K$ be the field of fractions of $W$. In \cite{FONTAINE_LAFFAILLE}, Fontaine and Laffaille define a Tannakian category $\MF^{\Phi,f}_K$ of weakly admissible filtered $\Phi$-modules, and use it to give an explicit construction of associated $p$-adic Galois representations. In order to do this they consider a category $\MF_{W,\fp}$ of modules over $W$ endowed with additional structures, and they show that every weakly admissible $\Phi$-module can be obtained by base change of a torsion free such module to the field of fractions $K$ of $W$. 

 Richard Pink conjectured that the Tannakian formalism should be applicable to these categories of modules over $W$. The results of \cite{SCHAEPPI} can be used to show that there is indeed an affine groupoid whose category of dualizable representations is the full subcategory of $\MF_{W,\fp}$ of objects whose underlying $W$-module is torsion free. Richard Pink then conjectured that this groupoid has the following properties. Firstly, it should be a $\mathbb{Z}_p$-model of the affine groupoid associated to the $\mathbb{Q}_p$-linear Tannakian category $\MF^{\Phi,f}_K$. Secondly, the representations of its reduction modulo $p^n$ should be the subcategory of $\MF_{W,\fp}$ consisting of modules which are annihilated by $p^n$, and the corresponding commutative Hopf algebroid should be the limit of its reductions modulo $p^n$. In discussions it became quickly clear that this last point is problematic.

 Nevertheless, using our recognition theorem, the generalization of Lurie's result, and our results about Adams stacks, we can show that the conjecture holds in the 2-category of Adams stacks.

\begin{thm}\label{thm:conjecture_for_adams_stacks}
 There is an algebraic stack $\mathfrak{P}$ on the $\fpqc$-site $\Aff$ and a symmetric monoidal equivalence $\MF_{W,\fp} \simeq \Coh(\mathfrak{P})$. Let 
\[
\vcenter{ \xymatrix{ \mathfrak{P}_{{\mathbb{Q}}_p} \ar[d] \ar[r] & \mathfrak{P} \ar[d] \\ 
\Spec(\mathbb{Q}_p) \ar[r] & \Spec(\mathbb{Z}_p)} }
\quad\mbox{and}\quad
\vcenter{ \xymatrix{ \mathfrak{P}_{n} \ar[d] \ar[r] & \mathfrak{P} \ar[d] \\ 
\Spec(\mathbb{Z}\slash p^n \mathbb{Z}) \ar[r] & \Spec(\mathbb{Z}_p)}}
\]
 be bipullback squares in the category of \emph{all} stacks on the $\fpqc$-site $\Aff$. Then $\Coh(\mathfrak{P}_{{\mathbb{Q}}_p})$ is equivalent (as a symmetric monoidal $\mathbb{Q}_p$-linear category) to the Tannakian category $\MF^{\Phi,f}_K$ of weakly admissible $\Phi$-modules, and $\Coh(\mathfrak{P}_n)$ is equivalent to the full subcategory of $\MF_{W,\fp}$ of objects whose underlying $W$-modules are annihilated by $p^n$. Moreover, the canonical morphisms
\[
 \mathfrak{P}_n \rightarrow \mathfrak{P}
\]
 exhibit $\mathfrak{P}$ as bicolimit of the sequence $\mathfrak{P}_n \rightarrow \mathfrak{P}_{n+1}$ in the full sub-2-category of Adams stacks.
\end{thm}

 We will prove this theorem in \S \ref{section:filteredmodules}. The major difficulty lies in identifying the categories of quasi-coherent sheaves of $\mathfrak{P}_n$ and $\mathfrak{P}_{\mathbb{Q}_p}$.

\begin{rmk}
 In the proof we will see that the stacks $\mathfrak{P}$, $\mathfrak{P}_n$ and $\mathfrak{P}_{\mathbb{Q}_p}$ are coherent algebraic stacks with the resolution property, so they are in particular Adams stacks.
\end{rmk}

 Note that $\mathfrak{P}_{\mathbb{Q}_p}$ is a stack on the $\fpqc$-site $\Aff$ of affine schemes over $\mathbb{Z}$, so it is not equal to the gerbe whose representations are equivalent to the Tannakian category $\MF^{\Phi,f}_K$ constructed in \cite{DELIGNE}. The latter is a stack on the $\fpqc$-site of schemes over $\mathbb{Q}_p$. We cannot avoid this, because we clearly need a 2-category containing both $\mathfrak{P}$ and $\Spec(\mathbb{Z}_p)$ in order to talk about the desired bipullbacks.

\section*{Acknowledgments}
 This project forms a part of my graduate studies at the University of Chicago under the supervision of Prof.\ Peter May. I want to thank Peter May for providing a wonderful learning environment, both in personal discussions and in the seminars he organises at Chicago, and for helpful and detailed suggestions on how to improve this paper. I thank Mike Shulman for pointing out a few corrections.
 
 The paper was written at Macquarie University during a half-year stay made possible by a 2012 Endeavour Research Fellowhship from the Australian government.
 I am very grateful to Richard Garner, Steve Lack, and Ross Street for many long discussions on various aspects of this work. Ross Street pointed out a significant simplification of the proof of the recognition theorem, which made it possible to extend the results of \S \ref{section:enriched} to a larger class of enriched categories than I originally considered. Last, but certainly not least, I want to thank Richard Pink for suggesting and explaining such an interesting problem.

%% file: weaklytannakian.tex
 \section{Weakly Tannakian categories}\label{section:weaklytannakian}

\subsection{The recognition theorem in the language of Hopf algebroids}

 In this section we will show how Theorem~\ref{thm:recognition} can be deduced from the following result, which is stated entirely in the language of flat Hopf algebroids.

\begin{thm}\label{thm:hopf_recognition}
 Let $R$ be a commutative ring, and let $\ca{A}$ be a weakly Tannakian category with $R$-linear fiber functor $w \colon \ca{A} \rightarrow \Mod_B$. Then there exists a coherent commutative Hopf algebroid $(B,\Gamma)$ in $\Mod_R$ with the resolution property, together with a symmetric monoidal $R$-linear equivalence $\ca{A} \simeq \Comod_{\fp}(B,\Gamma)$ such that the triangle
\[
 \xymatrix{\ca{A} \ar[r]^-{\simeq} \ar[rd]_w & \Comod_{\fp}(B,\Gamma) \ar[d]^V\\ & \Mod_B} 
\]
 commutes, where $V$ denotes the forgetful functor.
\end{thm}

 This result has the following immediate consequence, which is not entirely obvious from the definition of weakly Tannakian categories.

\begin{cor}
 An object $A \in \ca{A}$ of a weakly Tannakian category has a dual if and only if $w(A)$ does for some fiber functor $w \colon \ca{A} \rightarrow \Mod_B$.
\end{cor}

\begin{proof}
 It is a well-known fact that a comodule of a flat Hopf algebroid has a dual if and only if its underlying module does (see for example \cite[Proposition~1.3.4]{HOVEY}).
\end{proof}

 It also allows us to prove that weakly Tannakian $R$-linear categories which admit a neutral fiber functor are equivalent to categories of finitely presentable representations of flat affine group schemes over $R$.

\begin{proof}[Proof of Theorem~\ref{thm:affine_group_schemes_recognition}]
 Recall that a fiber functor is neutral if $B=R$. In that case, Theorem~\ref{thm:hopf_recognition} shows that $\ca{A}$ is equivalent to the category of comodules of a flat Hopf algebroid $(R,\Gamma)$. The pair $(R,\Gamma)$ is a commutative Hopf algebroid in $\Mod_R$ if and only if $\Gamma$ is a commutative Hopf algebra in $\Mod_R$. Under the equivalence between $R$-algebras and affine schemes over $R$, flat Hopf algebras correspond to flat affine group schemes.
\end{proof}

 In order to relate the two recognition theorems, Theorems~\ref{thm:recognition} and \ref{thm:hopf_recognition} we first need to define the terms in the former. 

\begin{dfn}
 The $\fpqc$-topology on the category $\Aff$ of affine schemes is generated by finite families consisting of flat morphisms $U_i \rightarrow U$ such that $\coprod U_i \rightarrow U$ is faithfully flat, meaning that the corresponding morphism $A \rightarrow \prod A_i$ of commutative rings makes $\prod A_i$ a faithfully flat $A$-algebra. 
\end{dfn}

 Following Naumann \cite{NAUMANN}, we call a stack on the $\fpqc$-site $\Aff$ \emph{algebraic} if it has an affine diagonal and admits a faithfully flat morphism from an affine scheme. We will call such a faithfully flat morphism a \emph{presentation} of the algebraic stack. Note that the usual definition of an algebraic stack uses a different Grothendieck topology, imposes certain finiteness conditions which would exclude the gerbes considered in \cite{SAAVEDRA, DELIGNE}, and at the same time relaxes the requirement that there exists a presentation by an affine scheme. Algebraic stacks in the sense of Naumann generalize Lurie's geometric stacks, so from that point of view it might be preferable to call them geometric. However, they do not have to satisfy any smoothness conditions, so the author prefers to follow Naumann's terminology.

 All algebraic stacks arise as stacks associated to flat affine groupoid schemes, see \cite[\S 3]{NAUMANN}, so this definition coincides with the one given in \S \ref{section:introduction_weakly_tannakian}. Note that all the stacks we consider are stacks on the site of affine schemes over $\mathbb{Z}$. The following proposition shows that we can still talk about algebraic stacks over an arbitrary commutative ring $R$ in this setting, literally as an algebraic stack over $\Spec(R)$.

\begin{prop}\label{prop:general_base}
 Let $X$ be an algebraic stack, associated to the flat affine groupoid scheme $\bigl( \Spec(A),\Spec(\Gamma) \bigr)$, and let $R$ be a commutative ring. Then there is an equivalence between morphisms $X \rightarrow \Spec(R)$ of stacks and $R$-algebra structures on $A$ which endow $(A,\Gamma)$ with the structure of a Hopf algebroid in $\Mod_R$.
\end{prop}
 
\begin{proof}
 Since $\Spec(R)$ is a stack and $X$ is the stack associated to the presheaf of groupoids represented by $\bigl(\Spec(A),\Spec(\Gamma)\bigr)$, giving a morphism $X \rightarrow \Spec(R)$ of stacks amounts to giving a morphism
\[
 (R,R) \rightarrow (A,\Gamma)
\]
 of Hopf algebroids. Such morphisms are uniquely determined by the ring homomorphism $R \rightarrow A$, and an arbitrary such ring homomorphism gives rise to a morphism of Hopf algebroids as above if and only if the two resulting $R$-actions on $\Gamma$ coincide.
\end{proof}

 Recall that an object $C$ of a category $\ca{C}$ is called \emph{finitely presentable} if $\ca{C}(C,-)$ preserves filtered colimits. The category $\ca{C}$ \emph{locally finitely presentable} if it is cocomplete and it has a strong generator consisting of finitely presentable objects (see \cite{GABRIEL_ULMER}). Here the term `local' is used in the category theorists' sense, not in the geometric sense.

 \begin{dfn}\label{dfn:coherent}
 An algebraic stack $X$ is called \emph{coherent} if the category $\QCoh(X)$ is locally coherent, that is, if it is locally finitely presentable, the category $\Coh(X)$ coincides with the category of finitely presentable objects, and it is closed under finite limits in $\QCoh(X)$ (see \cite[\S 2]{ROOS} and \cite[Theorem~1.6]{HERZOG}).
 \end{dfn}

 \begin{dfn}\label{dfn:resolution_property}
  An algebraic stack $X$ has the \emph{resolution property} if every coherent sheaf is a quotient of a dualizable coherent sheaf. 
 \end{dfn}

 We are now ready to deduce Theorem~\ref{thm:recognition} from Theorem~\ref{thm:hopf_recognition}.

\begin{proof}[Proof of Theorem~\ref{thm:recognition}]
 Recall from \cite[\S3.4]{NAUMANN} and \cite[Remark~2.39]{GOERSS} that the category of quasi-coherent sheaves on a stack $X$ associated to a flat Hopf algebroid $(B,\Gamma)$ is equivalent to the category of comodules of $(B,\Gamma)$.

 First let $X$ be a coherent algebraic stack over $R$ with the resolution property. Since $X$ is algebraic, it is associated to a flat Hopf algebroid $(B,\Gamma)$. The morphism $X \rightarrow \Spec(R)$ corresponds to a morphism of Hopf algberoids $(R,R) \rightarrow (B,\Gamma)$, which implies that $(B,\Gamma)$ is a Hopf algbroid in $\Mod_R$ (see Proposition~\ref{prop:general_base}). The forgetful functor from the category of finitely presentable comodules to the category of $B$-modules gives the desired $R$-linear fiber functor. The category is abelian since $X$ is coherent (see Definition~\ref{dfn:coherent}), and it satisfies condition ii) of Definition~\ref{dfn:weakly_tannakian} because $X$ has the resolution property (see Definition~\ref{dfn:resolution_property}).

 Conversely, suppose that $\ca{A}$ is an $R$-linear weakly Tannakian category. By Theorem~\ref{thm:hopf_recognition} there exists a flat Hopf algebroid $(B,\Gamma)$ in $\Mod_R$ and a symmetric monoidal $R$-linear equivalence $\ca{A} \simeq \Comod(A,\Gamma)$. Since the two $R$-actions of on $\Gamma$ coincide, we get a morphism $(R,R) \rightarrow (B,\Gamma)$. Passing to associated stacks we get the desired algebraic stack $X$ over $R$.
\end{proof}

 The proof of Theorem~\ref{thm:hopf_recognition}, the recognition theorem for flat Hopf algebroids, has four key ingredients: left Kan extensions, Beck's comonadicity theorem, dense subcategories, and a characterization of such subcategories due to Day and Street. We recall these categorical concepts and results as they are needed.

\subsection{Left Kan extensions and Beck's comonadicity theorem}

 Let $R$ be a commutative ring. A natural transformation
\[
 \xymatrix{\ca{A} \ar[rd]_{F} \ar[r]^K & \ca{C} \ar[d]^{L} \dtwocell\omit{^<2>\alpha} \\ & \ca{B} }
\]
 between $R$-linear functors is said to exhibit $L$ as \emph{left Kan extension} of $F$ along $K$ if the function
\[
 [\ca{C},\ca{B}](L,G) \rightarrow  [\ca{A},\ca{B}](F,KG)
\]
 which sends a natural transformation $\varphi \colon L \Rightarrow G$ to $\varphi K \cdot \alpha$ is a bijection\footnote{This is the right definition for $R$-linear categories. For general enriched categories, the unviersal property has to be strengthened (see \cite[\S4]{KELLY_BASIC})}. If this is the case we write $\Lan_K F$ for $L$.

 The natural transformation $\alpha$ is called the \emph{unit} of the Kan extension. If $K$ is fully faithful (which will always be the case for the left Kan extensions we consider), then the unit $\alpha$ is invertible (see \cite[Proposition~4.23]{KELLY_BASIC}). If $\ca{B}$ is cocomplete and $\ca{A}$ is essentially small, the left Kan extension of any $R$-linear functor $F$ exists.

 We will frequently use the special case where $K$ is the Yoneda embedding $Y \colon \ca{A} \rightarrow \Prs{A}$. In that case, the $R$-linear functor $L$ has a right adjoint, henceforth always denoted by
\[
\widetilde{F} \colon \ca{B} \rightarrow \Prs{A} \smash{\rlap{,}}
\]
 which sends $B\in \ca{B}$ to $\ca{B}(F-,B) \in \Prs{A}$. If $\ca{B}$ is cocomplete, then the functor which sends $F$ to $\Lan_Y F$ gives an equivalence
\[
 [\ca{A},\ca{B}] \rightarrow \Cocts[\Prs{A},\ca{B}] \smash{\rlap{,}}
\]
 with inverse given by whiskering with the Yoneda embedding.

 The first theorem from category theory that we need is Beck's comonadicity theorem, which gives a characterization of categories of comodules\footnote{Comodules are often called coalgebras. We do not use this terminology because we frequently consider comonads induced by $R$-coalgebras.} of comonads. Let
\[
 F \colon \ca{C} \rightleftarrows \ca{D} \colon G
\]
 be an $R$-linear adjunction with unit $\eta$ and counit $\varepsilon$. Then $FG$ is a comonad with comultiplication given by $F\eta G \colon FGFG \Rightarrow FG$ and counit given by $\varepsilon \colon FG \Rightarrow \id$. The assignment which sends $D \in \ca{D}$ to the comodule $(FD,F\eta_D)$ gives an $R$-linear functor
\[
 \ca{C} \rightarrow \ca{D}_{FG}
\]
 from $\ca{C}$ to the category of comodules of $FG$. The left adjoint $F$ is called \emph{comonadic} if the comparison functor is an equivalence. Beck's comonadicity theorem gives necessary and sufficient conditions for when this is the case. The following special case of the enriched version of Beck's result suffices for our purposes.

\begin{prop}
 Let $F \colon \ca{C} \rightleftarrows \ca{D} \colon G$ be an $R$-linear adjunction. If $\ca{C}$ has equalizers, $F$ preserves them, and $F$ reflects isomorphisms, then $F$ is comonadic. In particular, if $\ca{C}$ is abelian and $F$ is faithful and exact, then it is comonadic.
\end{prop}

\begin{proof}
 The first claim follows from the enriched comonadicity theorem in \cite[Theorem~2.II.1]{DUBUC}, and the second follows from the fact that any faithful exact $R$-linear functor reflects isomorphisms.
\end{proof}

\subsection{The proof strategy}
 Throughout the remainder of this section we fix a commutative ring $R$, a commutative $R$-algebra $B$, and an $R$-linear weakly Tannakian category $\ca{A}$ with a fiber functor $w \colon \ca{A} \rightarrow \Mod_B$. We will write $\ca{A}^d$ for the full subcategory of dualizable objects.

\begin{prop}\label{prop:L_comonadic}
 The left Kan extension $L \colon \Ind(\ca{A}) \rightarrow \Mod_B$ of $w$ along the inclusion $\ca{A} \rightarrow \Ind(\ca{A})$ is left exact, comonadic, and strong symmetric monoidal.
\end{prop}

\begin{proof}
 The category $\Ind(\ca{A})$ can be identified with the full subcategory of $\Prs{A}$ consisting of left exact presheaves. By \cite[Theorem~4.47]{KELLY_BASIC}, $L$ is given by the restriction of $\Lan_Y w$ to $\Ind(\ca{A})$.  Right exactness of $w$ implies that
\[
\widetilde{w} \colon \Mod_B \rightarrow \Prs{A} 
\]
 factors through $\Ind(\ca{A})$. This shows that $L$ is a left adjoint.

 Since $w$ is left exact and faithful, so is $L$. Thus $L$ reflects isomorphisms and preserves equalizers, hence it is comonadic. That it is strong symmetric monoidal follows directly from the definition of the tensor product of ind-objects of $\ca{A}$ as filtered colimit of tensor products of objects in $\ca{A}$.
\end{proof}

 The above result already allows us to give a proof strategy for the recognition theorem, which we summarize in the following corollary.

\begin{cor}\label{cor:proof_strategy}
 If the comonad induced by the functor $L$ from Proposition~\ref{prop:L_comonadic} is a cocontinuous symmetric Hopf monoidal comonad, then $\Ind(\ca{A})$ is equivalent (as a symmetric monoidal category) to the category of comodules of a flat Hopf algebroid $(B,\Gamma)$.
\end{cor}

\begin{proof}
 Since the induced comonad is cocontinuous, its underlying functor is given by
\[
 \Gamma \mathop{\otimes_B} - \colon \Mod_B \rightarrow \Mod_B
\]
 for some $B$-$B$ bimodule $\Gamma$. The comonad structure endows the pair $(B,\Gamma)$ with the structure of a coalgebroid. The symmetric monoidal structure corresponds to a bialgebroid structure on $(B,\Gamma)$, and the comonad is Hopf monoidal if and only if $(B,\Gamma)$ is a Hopf algebroid (cf.\ \cite[Example~11.6.2]{SCHAEPPI}). Comodules of $(B,\Gamma)$ correspond to comodules of the comonad induced by $L$. 
\end{proof}

 To show that the induced comonad is Hopf monoidal and that the right adjoint of $L$ is cocontinuous we will need a different way to compute the comonad induced by $L$. This is the key step in our argument, and requires further concepts and results from category theory.

\subsection{Density and left Kan extensions}
 An $R$-linear functor $K \colon \ca{A} \rightarrow \ca{C}$ is called \emph{dense} if $\widetilde{K} \colon \ca{C} \rightarrow \Prs{A}$ is fully faithful. A full subcategory is called dense if the induced inclusion functor is dense. See \cite[Theorem~5.1]{KELLY_BASIC} for various equivalent conditions. From the Yoneda lemma we get a natural isomorphism $\widetilde{Y}F \cong F$, which shows that the Yoneda embedding is always dense.

\begin{dfn}
 Let $K \colon \ca{A} \rightarrow \ca{C}$ be a fully faithful functor. A colimit in $\ca{C}$ is called \emph{$K$-absolute} if it is preserved by $\widetilde{K}$. If $\ca{A}$ is essentially small, this is equivalent to being preserved by $\ca{C}(KA,-)$ for all $A \in \ca{A}$.
\end{dfn}

 This concept allows us to identify certain functors as left Kan extensions.

\begin{thm}\label{thm:left_kan_characterization}
 Let $K \colon \ca{A} \rightarrow \ca{C}$ be a fully faithful dense functor. Then a functor $S \colon \ca{C} \rightarrow \ca{D}$ is isomorphic to $\Lan_{K} SK$ if and only if it preserves all $K$-absolute colimits.
\end{thm}

\begin{proof}
 This follows from \cite[Theorem~5.29]{KELLY_BASIC}.
\end{proof}

 The key ingredient in our proof is the following characterization of dense subcategories in Grothendieck abelian categories, which is a special case of a theorem of Day and Street. Recall that a set $\ca{G}$ of objects in a category $\ca{C}$ is called a \emph{strong generator} if the representable functors $\ca{C}(G,-)$, $G \in \ca{G}$, jointly reflect isomorphisms.

\begin{thm}[Day-Street]\label{thm:density_characterization}
 Let $\ca{C}$ be a Grothendieck abelian category. Then a subcategory $\ca{A}$ of $\ca{C}$ is dense if and only if it is a strong generator. 
\end{thm}

\begin{proof}
 This follows from \cite[Theorem~2]{DAY_STREET_GENERATORS} and Example~(3) in loc.\ cit. 
\end{proof}

\subsection{Proof of the recognition theorem for Hopf algebroids}\label{section:recognition_proof}

 The following proposition can be used to give an alternative description of the comonad induced by the functor $L$ from Proposition~\ref{prop:L_comonadic}.

\begin{prop}\label{prop:duals_dense}
 Let $L \colon \Ind(\ca{A}) \rightarrow \Mod_B$ be the left Kan extension of $w$ along the inclusion $\ca{A} \rightarrow \Ind(\ca{A})$. Write $K \colon \ca{A}^d \rightarrow \Ind(\ca{A})$ for the evident inclusion. The functor $K$ is dense, and $L$ is a left Kan extension of the restriction of $w$ to $\ca{A}^d$ along $K$.
\end{prop}

\begin{proof}
 To show that $K$ is dense we only have to show that $\ca{A}^d$ forms a strong generator of $\Ind(\ca{A})$ (see Theorem~\ref{thm:density_characterization}). To do this it suffices to show that $\Ind(\ca{A})$ is the closure of $\ca{A}^d$ under colimits (see \cite[Proposition~3.40]{KELLY_BASIC}). The inclusion of $\ca{A}$ in $\Ind(\ca{A})$ is exact, and every ind-object is a filtered colimit of objects in $\ca{A}$. Therefore it suffices to show that every object of $\ca{A}$ is a finite colimit of objects of $\ca{A}^d$.

 Let $A \in \ca{A}$, and choose an epimorphism $D \rightarrow A$ with $D \in \ca{A}^d$ (such an epimorphism exists by part ii) of Definition~\ref{dfn:weakly_tannakian}). Let $B$ be its kernel, and choose an epimorphism $D^{\prime} \rightarrow B$ with $D^{\prime} \in \ca{A}^d$. We obtain an exact sequence
\[
 \xymatrix{D^{\prime} \ar[r] & D \ar[r] & A \ar[r] & 0} \smash{\rlap{,}}
\]
 which shows that $A$ is the cokernel of a morphism in $\ca{A}^d$. Thus $\ca{A}^d$ is indeed a strong, and therefore dense, generator of $\Ind(\ca{A})$.

 In Proposition~\ref{prop:L_comonadic} we have already seen that $L$ is comonadic. Therefore it preserves all colimits, and it follows from Theorem~\ref{thm:left_kan_characterization} that $L$ is isomorphic to $\Lan_K LK$. Using the fact that the restriction of $L$ to $\ca{A}$ is isomorphic to $w$ it follows that $LK \cong w K$.
\end{proof}

 The notion of density heavily depends on the enriching category. In order to extend the proof of the above result to more general enriched categories, we will have to deal with this problem. On the other hand, in the remainder of this section we can give arguments that work equally well for $R$-linear categories and the enriched categories considered in \S\ref{section:enriched}.

 Recall that the category of $R$-linear presheaves of a symmetric monoidal category has a symmetric monoidal structure given by Day convolution (see \cite{DAY_CLOSED}). The Yoneda embedding is strong monoidal, and the left Kan extension of a strong monoidal functor along the Yoneda embedding is again strong monoidal (see \cite[Theorem~5.1]{IM_KELLY}).

\begin{cor}\label{cor:restricted_comonad_isomorphic}
 Let $w_d$ be the restriction of $w$ to $\ca{A}^d$. The comonad induced by the functor $L$ is naturally isomorphic (as symmetric monoidal $R$-linear comonad) to the comonad induced by $\Lan_Y w_d  \dashv \widetilde{w_d}$.
\end{cor}

\begin{proof}
 To establish the isomorphism of symmetric monoidal comonads we will argue entirely in the 2-category of symmetric (lax) monoidal $R$-linear functors and symmetric monoidal natural transformations. Giving an adjunction in this 2-category is equivalent to endowing the left adjoint with the structure of a \emph{strong} symmetric monoidal structure (see \cite{KELLY_DOCTRINAL}). 

 Let $\ca{C}$ be a cocomplete symmetric monoidal $R$-linear category for which tensoring with a fixed object gives a cocontinuous functor. If we endow the category of presheaves of a small symmetric monoidal $R$-linear category $\ca{B}$ with its Day convolution structure, then left Kan extensions along the Yoneda embedding induces an equivalence between strong symmetric monoidal functors $\ca{B} \rightarrow \ca{C}$ and strong monoidal left adjoints $\Prs{B} \rightarrow \ca{C}$. The inverse of this equivalence is given by whiskering with the Yoneda embedding, and one of the required natural isomorphisms is given by the units of the left Kan extensions (this is implicit in the the proof of \cite[Theorem~5.1]{IM_KELLY}). Applying this to the case $\ca{B}=\ca{A}^d$ and $\ca{C}=\Ind(\ca{A})$, we find that
\[
\Lan_Y K \colon \mathcal{P} \ca{A}^d \rightarrow \Ind(\ca{A}) 
\]
 has a unique strong symmetric monoidal structure such that the unit $\Lan_Y K \cdot Y \cong K$ is symmetric monoidal. Similarly we get a unique symmetric monoidal structure on $\Lan_Y w^d$ such that the unit $\Lan_Y w^d \cdot Y \cong w^d$ is symmetric monoidal.

 If we whisker the symmetric monoidal isomorphism $L \cdot Y_{\ca{A}} \cong w$ with the inclusion $\ca{A}^d \rightarrow \ca{A}$ we get a symmetric monoidal isomorphism $L \cdot K \cong w^d$. Together these three isomorphisms give a symmetric monoidal isomorphism
\[
 L \cdot \Lan_Y K \cdot Y \cong L \cdot K \cong w_d \cong \Lan_Y w^d \cdot Y \smash{\rlap{,}}
\]
 and \cite[Theorem~5.1]{IM_KELLY} implies that $L \cdot \Lan_Y K \cong \Lan_Y w^d$ as symmetric monoidal $R$-linear functors. From this it follows that $L \cdot \Lan_Y K$ and $\Lan_Y w_d$ induce isomorphic symmetric monoidal comonads.

 It remains to show that $L$ and $L \cdot \Lan_Y K$ induce isomorphic symmetric monoidal comonads. Let $W$ be the right adjoint of $L$. The comonad induced by $L$ is thus given by $L \cdot W$. It follows that
\[
 L\varepsilon W \colon L \cdot \Lan_Y K \cdot \widetilde{K} \cdot W \rightarrow L \cdot W
\]
 defines a morphism of symmetric monoidal comonads (the analogous fact is true for any pair of adjunctions in a 2-category). Thus it suffices to show that the counit $\varepsilon$ of the adjunction $\Lan_Y K \dashv \widetilde{K}$ is an isomorphism. From Proposition~\ref{prop:duals_dense} we know that $K$ is dense, hence that $\widetilde{K}$ is fully faithful. A right adjoint is fully faithful if and only if the counit of the adjunction is an isomorphism. 
\end{proof}

\begin{cor}\label{cor:L_hopf_monoidal}
 The symmetric monoidal comonad induced by the functor $L$ from Proposition~\ref{prop:L_comonadic} is cocontinuous and Hopf monoidal.
\end{cor}

\begin{proof}
 By Corollary~\ref{cor:restricted_comonad_isomorphic} it suffices to check this for the comonad $\Lan_Y w_d \cdot \widetilde{w_d}$. The functor $\widetilde{w_d}$ sends a $B$-module $M$ to the presheaf $\Mod_B\bigl(w_d(-),M\bigr)$. Since colimits in a presheaf category are computed pointwise, it suffices to check that for every $A \in \ca{A}^d$, the functor
\[
 \Mod_B\bigl(w(A),-\bigr) \colon \Mod_B \rightarrow \Mod_R
\]
 is cocontinuous. This follows immediately from the fact that $w(A)$ is dualizable.

 There are three ways to see that the resulting comonad is Hopf monoidal (respectively that the corresponding bialgebroid is a Hopf algebroid). We sketch the first two, and give a completely rigorous proof using some of the machinery from \cite{SCHAEPPI}.

  Using the coend formula for left Kan extensions, we find that the bialgebroid $(B,\Gamma)$ is given by
\[
 \Gamma=\int^{A \in \ca{A}^d} w_d(A)\mathop{\otimes_B} w_d(A)^{\vee} \smash{\rlap{,}}
\]
 and one can check that the usual argument for the existence of an antipode (for example in \cite{DELIGNE, DAY_TANNAKA, STREET_QUANTUM}) does not depend on the existence of any limits or colimits in the domain category.

 Alternatively, we can use the fact that any symmetric monoidal $R$-linear adjunction
\[
 F \colon \ca{C} \rightleftarrows \ca{D} \colon G
\]
 where $\ca{C}$ has a dense set of dualizable objects and $G$ is cocontinuous induces a Hopf monoidal comonad. Indeed, for a dualizable object $X \in \ca{C}$ and $C \in \ca{C}$, $D \in \ca{D}$ arbitrary we have the sequence of natural isomorphisms
\begin{align*}
 \ca{C}(C, X\otimes GD) & \cong \ca{C}(C \otimes X^{\vee}, GD) \\
& \cong \ca{D}\bigl(F(C \otimes X^{\vee}),D\bigr)\\
& \cong \ca{D}\bigl(FC \otimes (FX)^{\vee},D \bigr)\\
& \cong \ca{D}(FC, FX\otimes D)\\
& \cong \ca{C}\bigl(C,G(FX\otimes D)\bigr) \smash{\rlap{,}}
\end{align*}
 and density of the dualizable objects and the fact that $G$ is cocontinuous allow us to conclude that there is an isomorphism $C\otimes GD \rightarrow G(FC\otimes D)$ for all $C \in \ca{C}$. This makes it at least plausible that the canonical such morphism is invertible as well. There is something to be checked, though: in the above sequence of isomorphisms we have definitely used the fact that $X$ is dualizable, so one needs to show that the morphisms only available for such $X$ cancel in the end. Checking this would imply that the adjunction is strong right coclosed in the sense of \cite{CHIKHLADZE_LACK_STREET}. In \cite{FAUSK_HU_MAY}, such an adjunction is called an adjunction which \emph{satisfies the projection formula}. These induce (right) Hopf monoidal comonads by \cite[Proposition~4.4]{CHIKHLADZE_LACK_STREET}. Showing that the adjunction is strong \emph{left} coclosed is analogous.

 Finally, the following gives a completely rigorous proof that the comonad induced by the adjunction $\Lan_Y w_d \dashv \widetilde{w_d}$ is Hopf monoidal. Note that we can think of this adjunction as an internal symmetric monoidal adjunction between $\ca{A}^d$ and $B$ in the symmetric monoidal bicategory of profunctors. The conclusion follows from the general fact that a left adjoint monoidal 1-cell between autonomous map pseudomonoids in a monoidal bicategory induces a Hopf monoidal comonad (see \cite[Theorem~9.5.1]{SCHAEPPI}). 
\end{proof}

 We are now ready to prove the recognition theorem for Hopf algebroids.

\begin{proof}[Proof of Theorem~\ref{thm:hopf_recognition}]

 By Corollary~\ref{cor:L_hopf_monoidal}, the comonad induced by $L$ is cocontinuous and Hopf monoidal, hence by Corollary~\ref{cor:proof_strategy}, there exists a flat Hopf algebroid $(B,\Gamma)$ and a symmetric monoidal equivalence $\Ind(\ca{A}) \simeq \Comod(B,\Gamma)$ such that the triangle
\[
 \xymatrix{\Ind(\ca{A}) \ar[r]^-{\simeq} \ar[rd]_L & \Comod(B,\Gamma) \ar[d]^V \\ &\Mod_B }
\]
 commutes up to isomorphism. Since a comodule structure can be transferred uniquely along an isomorphism (in categorical language: the forgetful functor $V$ is an isofibration), we can replace this equivalence by one that makes the triangle strictly commutative.

 By definition of $L$ (see Proposition~\ref{prop:L_comonadic}), the triangle
\[
 \xymatrix{\ca{A} \ar[r] \ar[rd]_w & \Ind(\ca{A}) \ar[d]^L\\ & \Mod_B}
\]
 commutes up to symmetric monoidal isomorphism. Combining these two triangles, we get the desired symmetric monoidal equivalence. It only remains to show that $(B,\Gamma)$ is coherent and that it has the resolution property.

 It is coherent because the finitely presentable objects in $\Comod(B,\Gamma)\simeq \Ind(\ca{A})$ form an abelian category (see \cite[\S 2]{ROOS} and \cite[Theorem~1.6]{HERZOG}). By Proposition~\ref{prop:duals_dense}, the dualizable objects form a dense generator of $\Ind(\ca{A})$, so the Hopf algebroid $(B,\Gamma)$ has the resolution property by \cite[Proposition~1.4.1]{HOVEY}.
\end{proof}

%% file: enriched.tex
\section{Enriched weakly Tannakian categories}\label{section:enriched}
 
\subsection{The enriched recognition theorem}
 In this section we will state the definition of weakly Tannakian categories enriched in a cosmos $\ca{V}$, and the corresponding recognition theorem. To do that we have to assume that the cosmos $\ca{V}$ satisfies some technical conditions.

\begin{dfn}\label{dfn:DAG}
 A small full subcategory $\ca{X} \subseteq \ca{V}$ is called a \emph{dense autonomous generator} if it is $\Set$-dense, it consists of dualizable objects, and it is closed under duals and tensor products. 
\end{dfn}

 The category of ind-objects of a finitely complete additive category $\ca{A}$ can be identified with the category of \emph{left exact} presheaves, that is, functors
\[
 \ca{A}^{\op} \rightarrow \Ab
\]
 which preserve finite limits. Kelly has shown that the same is true for categories enriched in cosmos which is finitely presentable as a closed category (see \cite[\S5.5]{KELLY_FINLIM}). Therefore, throughout this section, the category
\[
 \Lex[\ca{A}^{\op},\ca{V}]
\]
 of left exact $\ca{V}$-presheaves (in the sense of \cite[\S4.5]{KELLY_FINLIM}) will be used in much the same way $\Ind(\ca{A})$ was used in \S\ref{section:weaklytannakian}.

 Recall from \cite[Proposition~7.3.2]{SCHAEPPI} that a cosmos with dense autonomous generator is finitely presentable as a closed category if and only if the unit object $I \in \ca{V}$ is finitely presentable, that is, the functor
\[
 \ca{V}(I,-) \colon \ca{V} \rightarrow \Set
\]
 preserves filtered colimits.

 Throughout this section we will therefore fix an abelian cosmos $\ca{V}$ with a dense autonomous generator $\ca{X}$ and a finitely presentable unit object $I$. We are particularly interested in the following examples.

\begin{example}
 The category of graded or differential graded $R$-modules is an abelian cosmos with a dense autonomous generator given by bounded finitely generated free (differential) graded modules. If $(A,\Gamma)$ is an Adams Hopf algebroid, then $\ca{V}=\Comod(A,\Gamma)$ is an abelian cosmos with a dense autonomous generator given by the dualizable comodules.

 Finally, the category of Mackey functors for a finite group $G$ is an abelian cosmos with dense autonomous generator. Recall that this is the category of $k$-linear presheaves of the autonomous category of spans of finite $G$-sets (see \cite[\S13]{PANCHADCHARAM_STREET} for details). The representable functors give the desired dense autonomous generator.
\end{example}
 
 If $B \in \ca{V}$ is a commutative monoid, then the category $\Mod_B$ is a $\ca{V}$-category in a natural way. One way to see this is to note that a module is really a presheaf on the $\ca{V}$-category $\ca{B}$ with one object $\ast$ and $\ca{B}(\ast,\ast)=B$. Thus $\Mod_B=\Prs{B}$ as a $\ca{V}$-category.

\begin{dfn}\label{dfn:enriched_weakly_tannakian_category}
 Let $\ca{V}$ be an abelian cosmos with dense autonomous generator $\ca{X}$ and finitely presentable unit object. Let $B \in \ca{V}$ be a commutative monoid, and let $\ca{A}$ be an $\ca{X}$-tensored abelian symmetric monoidal $\ca{V}$-category. A strong symmetric monoidal $\ca{V}$-functor
\[
 w\colon \ca{A} \rightarrow \Mod_B
\]
 is called a \emph{fiber functor} if its underlying unenriched functor $w_0$ is faithful and exact. A fiber functor is called \emph{neutral} if $B=I$ (and therefore $\Mod_B \simeq \ca{V}$).

 If $\ca{A}$ satisfies the conditions:
\begin{enumerate}
 \item[i)]
 There exists a fiber functor $w \colon \ca{A} \rightarrow \Mod_B$ for some commutative monoid $B \in \ca{V}$;
 \item[ii)]
 For all objects $A \in \ca{A}$ there exists an epimorphism $A^{\prime} \rightarrow A$ such that $A^{\prime}$ has a dual;
\end{enumerate}
 it is called a \emph{weakly Tannakian $\ca{V}$-category}.
\end{dfn}

 The goal of this section is to prove the following enriched version of our main theorem.

\begin{thm}\label{thm:enriched_recognition}
 Let $\ca{V}$ be an abelian cosmos with dense autonomous generator $\ca{X}$ and finitely presentable unit object. Let $\ca{A}$ be a weakly Tannakian $\ca{V}$-category, and let $w \colon \ca{A} \rightarrow \Mod_B$ be a fiber functor. Then there exists a Hopf algebroid $(B,\Gamma)\in \ca{V}$ and an equivalence $\ca{A} \simeq \Comod_{\fp}(B,\Gamma)$ of symmetric monoidal $\ca{V}$-categories such that the triangle
\[
 \xymatrix{\ca{A} \ar[r]^-{\simeq} \ar[rd]_w & \Comod_{\fp}(B,\Gamma) \ar[d]^V\\ & \Mod_B} 
\]
 commutes, where $V$ denotes the forgetful functor. The Hopf algebroid $(B,\Gamma)$ is flat in the sense that $\Gamma \ten{B}-$ preserves finite weighted limits.
\end{thm}

 The proof of this result closely follows the strategy used in the proof of Theorem~\ref{thm:recognition}.

\subsection{Proof of the enriched recognition theorem}\label{section:enriched_recognition_proof}
 In this section we assume that the reader has some familiarity with the basic concepts of enriched category theory \cite{KELLY_BASIC} and with the theory of weighted finite limits for categories enriched in a cosmos $\ca{V}$ which is locally finitely presentable as a closed category (see \cite{KELLY_FINLIM}).

\begin{lemma}\label{lemma:finitely_complete}
 Any weakly Tannakian $\ca{V}$-category $\ca{A}$ is finitely complete and finitely cocomplete in the enriched sense, and the fiber functor preserves finite weighted limits and finite weighted colimits.
\end{lemma}

\begin{proof}
 The dense autonomous generator $\ca{X}$ is by definition closed under duals. The fact that $\ca{A}$ is $\ca{X}$-tensored implies that it is also $\ca{X}$-cotensored. Indeed, the cotensor of $X\in \ca{X}$ and $A \in \ca{A}$ is given by the tensor $X^{\vee}\odot A$. Since $\ca{X}$ is a $\Set$-dense generator of $\ca{V}$ it follows that the notions of limits and colimits in the underlying category $\ca{A}_0$ coincide with the notion of conical limits and colimits in $\ca{A}$ (this follows from the $\Set$-density of $\ca{X}$ and the discussion in \cite[\S3.8]{KELLY_BASIC}).

 The closure of $\ca{X}$ under finite colimits is the full subcategory $\ca{V}_f$ of $\ca{V}$ consisting of finitely presentable objects. It follows that $\ca{A}$  has tensors and cotensors with objects in $\ca{V}_f$, given by the finite colimit (respectively finite limit) in $\ca{A}$ of the tensors (cotensors) with objects in $\ca{X}$. Thus $\ca{A}$ is finitely complete and finitely cocomplete in the enriched sense \cite[\S4.5]{KELLY_FINLIM}.

 Any $\ca{V}$-functor preserves tensors and cotensors with objects in $\ca{X}$ because they are absolute colimits respectively absolute limits (cf.\ \cite{STREET_ABSOLUTE}). Since $w$ preserves conical limits and colimits it follows from the above construction of tensors and cotensors with objects in $\ca{V}_f$ that they are preserved by $w$. 
\end{proof}

 Fix a weakly Tannakian $\ca{V}$-category $\ca{A}$ with fiber functor $w \colon \ca{A} \rightarrow \Mod_B$, and let $\ca{A}^d \subseteq \ca{A}$ be the full subcategory of dualizable objects. As already mentioned, some of the proofs of \S\ref{section:recognition_proof} generalize easily to the enriched setting. The general strategy is the same. The role of $\Ind(\ca{A})$ is played by the category $\Lex[\ca{A}^{\op},\ca{V}]$ of left exact functors, which makes sense by the above lemma. We start by showing that the left Kan extension
\[
 L \colon \Lex[\ca{A}^{\op},\ca{V}] \rightarrow \Mod_B
\]
 of $w$ along the corestricted Yoneda embedding
\[
 Y^{\prime} \colon \ca{A} \rightarrow \Lex[\ca{A}^{\op},\ca{V}]
\]
 is comonadic. Then we show that the comonad induced by $L$ is isomorphic (as a symmetric monoidal comonad) to the comonad induced by the left Kan extension $\Lan_Y w^d$ of the restriction
\[
 w^d \colon \ca{A}^d \rightarrow \Mod_B
\]
 of $w$ to $\ca{A}^d$ along the Yoneda embedding $Y \colon \ca{A}^d \rightarrow \mathcal{P}\ca{A}^d$. The arguments from \S\ref{section:recognition_proof} show that the latter is a cocontinuous symmetric Hopf monoidal comonad.

 \begin{lemma}\label{lemma:lex_abelian}
  The category $\Lex[\ca{A}^{\op},\ca{V}]_0$ is a Grothendieck abelian category.
 \end{lemma}

\begin{proof}
 Since $\ca{V}$ is abelian, the underlying unenriched category $\ca{B}_0$ of every $\ca{V}$-category $\ca{B}$ is $\Ab$-enriched. The category $\Lex[\ca{A}^{\op},\ca{V}]$ is locally finitely presentable in the enriched sense, so its underlying unenriched category is locally finitely presentable as well (see \cite[Proposition~7.5]{KELLY_FINLIM}). Thus it is in particular complete and cocomplete, and filtered colimits commute with finite limits. It only remains to check that the canonical morphism
\[
 \mathrm{coker}\bigl(\mathrm{ker}(f)\bigr) \rightarrow  \mathrm{ker}\bigl(\mathrm{coker}(f)\bigr)
\]
 is an isomorphism for every morphism $f \colon A \rightarrow B$ in $\Lex[\ca{A}^{\op},\ca{V}]_0$. From the uniformity lemma proved in \cite[\S 1]{DAY_STREET_LOCALISATIONS} we know that $f$ can be written as filtered colimit of morphisms $f_i \colon A_i \rightarrow B_i$ in the arrow category, where the $f_i$ lie in the image of the Yoneda embedding of $\ca{A}_0$. The comparison morphism induced by the $f_i$ is an isomorphism since $\ca{A}_0$ is abelian and the Yoneda embedding is exact. The claim follows since filtered colimits commute with kernels and cokernels.
\end{proof}

 \begin{prop}\label{prop:lex_symmetric_monoidal}
 The reflection of the Day convolution symmetric monoidal structure on $\Prs{A}$ defines a symmetric monoidal structure on $\Lex[\ca{A}^{\op},\ca{V}]$. 
 \end{prop}

\begin{proof}
 The category $\Lex[\ca{A}^{\op},\ca{V}]$ is a reflective subcategory of $\Prs{A}$. By Day's reflection theorem \cite{DAY_REFLECTION} it suffices to check that for all representable presheaves $F=\ca{A}(-,A) \in \Prs{A}$ and all $G \in \Lex[\ca{A}^{\op},\ca{V}]$, the internal hom $[F,G]$ lies in $\Lex[\ca{A}^{\op},\ca{V}]$. The internal hom for the Day convolution structure is given by the end
\[
[F,G]=[\ca{A}(-,A),G]=\int_{B\in \ca{A}} [\ca{A}(B,A),G(-\otimes B)] \smash{\rlap{,}}
\]
 which is isomorphic to $G(-\otimes A)$ by Yoneda. The existence of a fiber functor implies that $-\otimes A$ preserves finite weighted colimits (cf.\ Lemma~\ref{lemma:finitely_complete}), hence that 
\[
 [F,G] \cong G(-\otimes A)\colon \ca{A}^{\op} \rightarrow \ca{V}
\]
 preserves finite weighted limits.
\end{proof}

 The proof of the following proposition roughly follows the proof of \cite[Proposition~1]{DAY_TANNAKA}, but there are some complications since we do not assume that every object in $\ca{A}$ has a dual.

 \begin{prop}\label{prop:enriched_L_comonadic}
  The left Kan extension
\[
 L \colon \Lex[\ca{A}^{\op},\ca{V}] \rightarrow \Mod_B
\]
 is left exact, comonadic, and strong symmetric monoidal.
 \end{prop}

\begin{proof}
 We first show that $L$ is a left exact and strong symmetric monoidal left adjoint. To keep the notation simple we will write $\ca{C}=\Mod_B$ throughout this proof. Since the inclusion $\Lex[\ca{A}^{\op},\ca{V}] \rightarrow \Prs{A}$ is fully faithful, $L$ is isomorphic to the restriction of 
\[
 \Lan_Y w \colon \Prs{A} \rightarrow \ca{C}
\]
 to $\Lex[\ca{A}^{\op},\ca{V}] \subseteq \Prs{A}$ (see \cite[Theorem~4.47]{KELLY_BASIC}). The right adjoint $\widetilde{w}$ of $\Lan_Y w$ sends an object $C \in \ca{C}$ to $\ca{C}(w-,C)$. Since $w$ preserves finite weighted colimits (see Lemma~\ref{lemma:finitely_complete}) we know that $\ca{C}(w-,C)$ sends finite weighted colimits in $\ca{A}$ to finite weighted limits in $\ca{V}$, that is, it is a left exact presheaf. Therefore $\widetilde{w}$ also gives a right adjoint to the restriction of $\Lan_Y w$ to $\Lex[\ca{A}^{\op},\ca{V}]$.

 We claim that $\Lan_Y w$ is left exact. To see this it suffices to check that its composite with the forgetful functor $\ca{C}\rightarrow \ca{V}$ is left exact, since the latter creates limits. This forgetful functor also preserves (and creates) colimits, so it preserves left Kan extensions. Thus it suffices to show that the left Kan extension of a left exact functor is again left exact. This is the content of \cite[Theorem~6.11]{KELLY_FINLIM}. Since (finite) weighted limits in the reflective subcategory $\Lex[\ca{A}^{\op},\ca{V}]$ are computed as in $\Prs{A}$ it follows that $L$ is left exact as well.

 Let $F \colon \Prs{A} \rightarrow \Lex[\ca{A}^{\op},\ca{V}]$ be a reflection. Note that $\Lan_Y w \cong LF$, because both functors have isomorphic restrictions along the Yoneda embedding. By the proof of \cite[Theorem~5.1]{IM_KELLY}, $\Lan_Y w$ has a unique strong symmetric monoidal structure such that the unit of the Kan extension is symmetric monoidal. From the construction of the monoidal structure on $\Lex[\ca{A}^{\op},\ca{V}]$ as a reflection of the Day convolution symmetric monoidal structure it follows that $L$ is strong symmetric monoidal.

 It remains to check that the left adjoint $L$ is comonadic. Left exactness of $L$ implies that it preserves the necessary equalizers, so we only need to show that it reflects isomorphisms. To do this it suffices to show that the unit morphisms
\[
 \eta_A \colon A \rightarrow \widetilde{w} L(A)
\]
 are regular monomorphisms. Indeed, this is a statement about the underlying unenriched adjunction, and the implication follows easily from the dual of \cite[Theorem~3.13]{TTT}.
 
 First note that condition ii) of Definition~\ref{dfn:enriched_weakly_tannakian_category} implies that $w(A)$ is finitely presentable for all $A \in \ca{A}$. Thus $\widetilde{w}$ preserves filtered colimits. Since $\Lex[\ca{A}^{\op},\ca{V}]$ is finitely presentable, regular monomorphisms in it are closed under filtered colimits. Thus it suffices to check that the unit at a finitely presentable object is a regular monomorphism. The finitely presentable objects in $\Lex[\ca{A}^{\op},\ca{V}]$ are precisely the representable morphisms. This reduces the problem to showing that
\[
 w_{-,B} \colon \ca{A}(-,B) \rightarrow \ca{C}(w-,wB)
\]
 is a regular monomorphism for every $B \in \ca{A}$. From Lemma~\ref{lemma:lex_abelian} we know that $\Lex[\ca{A}^{\op},\ca{V}]_0$ is abelian, hence that every monomorphism in it is regular. Thus it suffices to check that for all $A,B \in \ca{A}$, the morphism
\[
 w_{A,B} \colon \ca{A}(A,B) \rightarrow \ca{C}(wA,wB) 
\]
 is a monomorphism in $\ca{V}$.  Since all $\ca{V}$-functors preserve tensors with $X \in \ca{X}$ (see \cite{STREET_ABSOLUTE}), the vertical arrows in the commutative diagram
\[
 \xymatrix{\ca{V}_0\bigr(X,\ca{A}(A,B)\bigl) \ar[d]_{\cong} \ar[rr]^-{\ca{V}_0(X,w_{A,B})} && \ca{V}_0\bigr(X,\ca{C}(wA,wB)\bigl) \ar[d]^{\cong} \\
\ca{V}_0\bigr(I,\ca{A}(X\odot A,B)\bigl) \ar[d]_{\cong} \ar[rr]^-{\ca{V}_0(I,w_{X\odot A,B})} && \ca{V}_0\Bigr(I,\ca{C}\bigl(w(X\odot A),wB\bigl)\Bigl) \ar[d]^{\cong} \\
\ca{A}_0(X\odot A,B) \ar[rr]^-{(w_0)_{X\odot A,B}} && \ca{C}_0\bigr(w(X \odot A),wB\bigl) }
\]
 are isomorphisms. The $\Set$-density of $\ca{X} \subseteq \ca{V}$ and the fact that $w_0 \colon \ca{A}_0 \rightarrow \ca{C}_0$ is faithful imply that $w_{A,B}$ is a monomorphism. This concludes the proof that $L$ reflects isomorphisms.
\end{proof}

 It remains to check that the proof of Proposition~\ref{prop:duals_dense} can be adapted to the enriched setting as well. In order to do this we need to recall the following theorem from \cite{SCHAEPPI}.

\begin{thm}\label{thm:density_with_DAG}
 Let $\ca{V}$ be a cosmos which has a dense autonomous generator $\ca{X}$. Let $\ca{A}$ be an $\ca{X}$-tensored $\ca{V}$-category, and let $\ca{C}$ be a $\ca{V}$-category which is cotensored. A $\ca{V}$-functor $K \colon \ca{A} \rightarrow \ca{C}$ is $\ca{V}$-dense if and only if $K_0 \colon \ca{A}_0 \rightarrow \ca{C}_0$ is $\Set$-dense.
\end{thm}

\begin{proof}
 This is \cite[Theorem~A.1.1]{SCHAEPPI}.
\end{proof}

 \begin{prop}\label{prop:duals_V_dense}
 Let $L \colon \Lex[\ca{A}^{\op},\ca{V}] \rightarrow \Mod_B$ be the left Kan extension of $w$ along the inclusion $\ca{A} \rightarrow \Lex[\ca{A}^{\op},\ca{V}]$. Write $K \colon \ca{A}^d \rightarrow \Lex[\ca{A}^{\op},\ca{V}]$ for the evident inclusion. The functor $K$ is $\ca{V}$-dense, and $L$ is a left Kan extension of the restriction of $w$ to $\ca{A}^d$ along $K$.
 \end{prop}

\begin{proof}
 We first check that $K$ is $\ca{V}$-dense. By Theorem~\ref{thm:density_with_DAG} it suffices to show that $K_0 \colon \ca{A}^d_0 \rightarrow \Lex[\ca{A}^{\op},\ca{V}]_0$ is $\Set$-dense. Since $\ca{A}^d_0$ has finite direct sums, the same theorem implies that this is the case if and only if  $K_0$ is $\Ab$-dense. We have already seen that $\Lex[\ca{A}^{\op},\ca{V}]_0$ is a Grothendieck abelian category (see Lemma~\ref{lemma:lex_abelian}). Thus we can apply the characterization of dense subcategories of such by Day and Street (see Theorem~\ref{thm:density_characterization}), which reduces the problem to showing that $\ca{A}^d$ is a strong generator. This can be proved exactly as in the proof of Proposition~\ref{prop:duals_dense}.

 From \cite[Theorem~5.29]{KELLY_BASIC} it follows that $L$ is isomorphic to $\Lan_K LK$ (as a $\ca{V}$-functor). The composite of $L$ with the corestricted Yoneda embedding
\[
 \ca{A} \rightarrow \Lex[\ca{A}^{\op},\ca{V}]
\]
 is isomorphic to $w$. Thus $L \cong \Lan_K wK$, as claimed.
\end{proof}

\begin{proof}[Proof of Theorem~\ref{thm:enriched_recognition}]
 As already mentioned in \S\ref{section:recognition_proof}, Corollaries~\ref{cor:restricted_comonad_isomorphic} and \ref{cor:L_hopf_monoidal} hold in the enriched context. Their proofs can be adapted simply by replacing `$R$-linear' with the prefix `$\ca{V}$-' where necessary. The $\ca{V}$-functor $L$ has the required properties by Propositions~\ref{prop:enriched_L_comonadic} and \ref{prop:duals_V_dense}.

 Corollary~\ref{cor:restricted_comonad_isomorphic} shows that the comonads induced by $L$ and by $\Lan_Y w^d$ are isomorphic as symmetric monoidal comonads, and Corollary~\ref{cor:L_hopf_monoidal} shows that these comonads are $\ca{V}$-cocontinuous symmetric Hopf monoidal comonads on $\Mod_B$. They must therefore be isomorphic to the comonad induced from a Hopf algebroid $(B,\Gamma)$ in $\ca{V}$. The finitely presentable objects in $\Lex[\ca{A}^{\op},\ca{V}]$ are precisely the representable functors (this follows from \cite[Theorem~7.2.(i)]{KELLY_FINLIM}, applied to the set $\ca{G}$ of representable functors). Composing the corestricted Yoneda embedding with the symmetric monoidal equivalence
\[
 \Lex[\ca{A}^{\op},\ca{V}] \rightarrow \Comod(B,\Gamma)
\]
 gives the desired commutative triangle
\[
 \xymatrix{\ca{A} \ar[r]^-{\simeq} \ar[rd]_w & \Comod_{\fp}(B,\Gamma) \ar[d]^V\\ & \Mod_B}
\]
 of symmetric monoidal functors. It can be chosen to be strictly commutative because $V$ is an isofibration (cf.\ the proof of Theorem~\ref{thm:hopf_recognition} at the end of \S \ref{section:recognition_proof}).
\end{proof}

%% file: localization.tex
\section{Embedding algebraic stacks in the 2-category of symmetric monoidal abelian categories}\label{section:localization}

\subsection{The embedding theorem as a consequence of a bicategorical universal property}

 In \cite[Remark~5.12]{LURIE}, Lurie has shown that the pseudofunctor which sends a geometric stack to its category of quasi-coherent sheaves gives an embedding of the 2-category of geometric stacks into the 2-category of symmetric monoidal categories, tame functors and symmetric monoidal natural isomorphisms.

 We first recall the definition of tame functors from \cite[Definition~5.9]{LURIE}.

\begin{dfn}[Lurie]\label{dfn:tame}
 An object $M \in \ca{A}$ of a symmetric monoidal abelian category is called \emph{flat} if the functor $M\otimes -$ is exact.

 A functor $F \colon \ca{A} \rightarrow \ca{B}$ is called \emph{tame} if it is a strong symmetric monoidal left adjoint satisfying the following two conditions:
\begin{enumerate}
 \item[i)]
 It sends flat objects to flat objects;
\item[ii)]
 It preserves exact sequences
\[
 \xymatrix{0 \ar[r] & M \ar[r] & M^{\prime} \ar[r] & M^{\prime\prime} \ar[r] & 0}
\]
 if $M^{\prime\prime}$ is flat.
\end{enumerate}
\end{dfn}

 As we will see, the key fact about tame functors is that they preserve faithfully flat algebras (see \cite[Remark~5.10]{LURIE}).

\begin{dfn}
 We write $\ca{H}$ for the 2-category of flat Hopf algebroids, $\ca{AS}$ for the 2-category of algebraic stacks, and $\ca{T}$ for the 2-category of symmetric monoidal closed abelian categories, tame functors, and symmetric monoidal natural transformations.
\end{dfn}

 The goal of this section is to prove Theorem~\ref{thm:stacks_embedding}, which states that the pseudofunctor
\[
\QCoh(-) \colon \ca{AS}^{\op} \rightarrow \ca{T} 
\]
 which sends an algebraic stack to its category of quasi-coherent modules is an equivalence on hom-categories.

 As already mentioned in the introduction, we prove Theorem~\ref{thm:stacks_embedding} by showing that both the associated stack pseudofunctor and the pseudofunctor which sends a flat Hopf algebroid to its category of comodules have the same universal property. This proof is quite different from the one given in \cite{LURIE}, where the analogous result is deduced from a more general fact about tame functors landing in the category of modules of some ringed topos. Our proof allows us to slightly generalize Lurie's result: we show that \emph{all} symmetric monoidal natural transformations between tame functors are invertible. Thus the 2-functor $\QCoh(-)$ from algebraic stacks to symmetric monoidal categories is fully faithful on 2-cells.

 To state the bicategorical universal property in question we need the following definition.

 \begin{dfn}\label{dfn:surjective_weak_equivalence}
  An internal functor
\[
(f_0,f_1) \colon (X_0,X_1) \rightarrow (Y_0,Y_1)
\]
 between affine groupoids is called a \emph{surjective weak equivalence} if $f_0 \colon X_0 \rightarrow Y_0$ is faithfully flat (the functor is surjective on objects), and the diagram
\[
 \xymatrix{X_1 \ar[d]_{(s,t)} \ar[r]^-{f_1} & Y_1 \ar[d]^{(s,t)} \\ 
X_0 \times X_0 \ar[r]^-{f_0 \times f_0} & Y_0 \times Y_0} 
\]
 is a pullback diagram (the functor is fully faithful).
 \end{dfn}

\begin{rmk}\label{rmk:weak_equivalences_hopf_algebroids}
 Under the equivalence between the 2-category of affine groupoids and the 2-category of flat Hopf algebroids, the surjective weak equivalences correspond to morphisms of Hopf algebroids $(A,\Gamma) \rightarrow (A^{\prime},\Gamma^{\prime})$ such that $A\rightarrow A^{\prime}$ is faithfully flat and the diagram
\[
 \xymatrix{\Gamma \ar[r] & \Gamma^{\prime} \\ 
A\otimes A \ar[r] \ar[u] & A^{\prime} \otimes A^{\prime} \ar[u]}
\]
 is a pushout in the category of commutative rings. We will call such morphisms surjective weak equivalences of Hopf algebroids.  
\end{rmk}

 Let $W$ be a class of 1-cells in a bicategory $\ca{C}$. For a bicategory $\ca{A}$ we write $\Hom_W(\ca{C},\ca{A})$ for the full sub-bicategory of those pseudofunctors which send 1-cells in $W$ to equivalences. A pseudofunctor $F \colon \ca{C} \rightarrow \ca{D}$ with this property is a \emph{bicategorical localization} of $\ca{C}$ at $W$ if the pseudofunctor
\[
 \Hom(\ca{D},\ca{A})  \rightarrow  \Hom_W(\ca{C},\ca{A})
\]
 given by precomposition with $F$ is an equivalence of bicategories.

 In \S \S \ref{section:comod_as_localization} and \ref{section:ftm}, we will prove the following result.

\begin{thm}\label{thm:comod_localization}
 The corestriction of the pseudofunctor
\[
 \Comod \colon \ca{H} \rightarrow \ca{T}
\]
 to its essential image is a bicategorical localization of $\ca{H}$ at the surjective weak equivalences.
\end{thm}

 A large part of the proof of the above theorem also works in the enriched context. In \S \ref{section:stackslocalization} we will prove Theorem~\ref{thm:stacks_localization}, which states that the category of algebraic stacks is also a bicategorical localization at the surjective weak equivalences. These two facts allow us to prove the embedding theorem.

\begin{proof}[Proof of Theorem~\ref{thm:stacks_embedding}]
 It follows directly from the definition that bicategorical localizations are unique up to essentially unique biequivalence. Using this, Theorem~\ref{thm:comod_localization}, and Theorem~\ref{thm:stacks_localization}, we find that there exists an essentially unique biequivalence
\[
 \ca{AS}^{\op} \rightarrow \mathrm{ess.im}(\Comod) \subseteq \ca{T}
\]
 which commutes with $L$ and $\Comod(-)$ up to pseudonatural equivalence. It only remains to show that the triangle
\[
 \xymatrix{& \ca{H} \ar[ld]_{L^{\op}} \ar[rd]^{\Comod(-)}\ar@{}[d]|{\simeq}  \\ \ca{AS}^{\op} \ar[rr]_{\QCoh(-)} && \ca{T} }
\]
 commutes up to pseudonatural equivalence. This follows from the equivalence
\[
 \QCoh\bigl(L(\Spec A, \Spec \Gamma)\bigr) \simeq \Comod(A,\Gamma)
\]
 of symmetric monoidal categories (see \cite[\S3.4]{NAUMANN} and \cite[Remark~2.39]{GOERSS}).
\end{proof}

 In order prove Theorems~\ref{thm:comod_localization} and \ref{thm:stacks_localization}, we will use the following characterization of bicategorical localizations due to Pronk \cite{PRONK}.

\subsection{A characterization of bicategorical localizations}
 In \cite{PRONK} a sufficient set of conditions on a class $W$ of morphisms in a bicategory $\ca{C}$ is given such that the localization of $\ca{C}$ at $W$ exists. These conditions generalize the concept of a calculus of right fractions for a class of morphisms in a category, which is due to Gabriel and Zisman.

\begin{dfn}[Pronk]
 A class of 1-cells $W$ in a bicategory $\ca{C}$ \emph{admits a calculus of right fractions} if the following conditions hold:
\begin{enumerate}
 \item[(BF1)] It contains the equivalences; 
 \item[(BF2)] It is closed under composition; 
 \item[(BF3)] For all $w\colon A \rightarrow B$ in $W$ and all 1-cells $f \colon C \rightarrow B$ there exist 1-cells $v$, $g$ and an invertible 2-cell
\[
 \xymatrix{D \ar[r]^-v \ar[d]_{g} \ar@{}[rd]|{\cong} & C \ar[d]^f \\ A \ar[r]_-{w} & B}
\]
 with $v \in W$; 
 \item[(BF4$'$)] Whiskering with a 1-cell in $W$ is fully faithful; 
 \item[(BF5)] The class $W$ is closed under 2-isomorphisms. 
\end{enumerate}
\end{dfn}

 Note that the condition (BF4$'$) above is stronger than the condition (BF4) from \cite[\S 2.1]{PRONK}, but it is always satisfied in the examples we consider. Pronk has shown that the bicategorical localization at a class which admits a right calculus of fractions exists. More important for us is the following characterization of localizations.

\begin{prop}[Pronk]\label{prop:localization_pronk}
 Let $W$ be a class of 1-cells of a bicategory $\ca{C}$ which admits a calculus of right fractions. A pseudofunctor $F \colon \ca{C} \rightarrow \ca{D}$ is a localization of $\ca{C}$ at $W$ if it sends 1-cells in $W$ to weak equivalences, and it satisfies:
 \begin{enumerate}
  \item[(EF1)] F is essentially surjective;
 \item[(EF2)] For every 1-cell $f \colon FA \rightarrow FB$ there exists a 1-cell $w \in W$ and a 1-cell $g$ in $\ca{C}$ such that $Fg\cong f \cdot Fw$;
 \item[(EF3)] F is fully faithful on 2-cells.
 \end{enumerate}
\end{prop}

\begin{proof}
 This is \cite[Proposition~24]{PRONK}.
\end{proof}

 The existence of a pseudofunctor $F$ with the properties (EF1)-(EF3) simplifies the criteria (BF1) through (BF5).

\begin{lemma}\label{lemma:localization_lemma}
 Let $W$ be a class of 1-cells of $\ca{C}$ which satisfies (BF3). Let $F$ be a pseudofunctor which sends the elements of $W$ to equivalences and which satisfies conditions (EF1)-(EF3) of Proposition~\ref{prop:localization_pronk}. Write $\overline{W}$ for the class of all 1-cells which $F$ sends to equivalences. Then a pseudofunctor $G \colon \ca{C} \rightarrow \ca{A}$ sends all the 1-cells in $W$ to equivalences if and only if it sends all the 1-cells of $\overline{W}$ to equivalences.
\end{lemma}

\begin{proof}
 Since $W$ is contained in $\overline{W}$, one direction is clear. Assume that $G$ sends the 1-cells in $W$ to equivalences. Let $g$ be a 1-cell such that $Fg$ is an equivalence, with inverse equivalence $h$, say. By condition (EF2) there exists a 1-cell $w \in W$ and a 1-cell $h^{\prime} \in \ca{C}$ such that $Fh^{\prime}\cong h \cdot Fw$. From this it follows that $F(gh^{\prime})$ is isomorphic to $Fw$, so we get an isomorphism $g\cdot h^{\prime}\cong w$ by (EF3). It follows that $G$ sends $g \cdot h^{\prime}$ to an equivalence.

 From (BF3) we conclude that there exist 1-cells $g^{\prime}$ in $\ca{C}$ and $v \in W$ such that $g\cdot v\cong w \cdot g^{\prime}$. From this we get
\[
 Fh^{\prime} \cdot Fg^{\prime} \cong h \cdot Fw\cdot Fg^{\prime}\cong h \cdot F(w\cdot g^{\prime}) \cong h \cdot F(g\cdot v) \cong h \cdot F(g) \cdot F(v) \cong F(v) \smash{\rlap{,}}
\]
 and from (EF3) we deduce that $h^{\prime} \cdot g^{\prime}\cong w$. Taken together we have shown that $G$ sends both $g\cdot h^{\prime}$ and $h^{\prime} \cdot g^{\prime}$ to equivalences. From this it follows that $Gh^{\prime}$ is an equivalence. Using the 2-out-of-3 property we find that $Gg$ is an equivalence.
\end{proof}

\begin{prop}\label{prop:localization}
 Let $W$ be a class of 1-cells of $\ca{C}$ and let $F \colon \ca{C} \rightarrow \ca{D}$ be a pseudofunctor which sends 1-cells in $W$ to equivalences. If $W$ satisfies (BF3), and $F$ satisfies (EF1)-(EF3), then $F$ is a bicategorical localization of $\ca{C}$ at $W$.
\end{prop}

\begin{proof}
 Let $\overline{W}$ be the class of 1-cells in $\ca{C}$ which are sent to equivalences by $F$. Lemma~\ref{lemma:localization_lemma} shows that the bicategories $\Hom_{W}(\ca{C},\ca{A})$ and $\Hom_{\overline{W}}(\ca{C},\ca{A})$ are equal for all bicategories $\ca{A}$. Thus $F$ is a localization of $\ca{C}$ at $W$ if and only if it is a localization of $\ca{C}$ at $\overline{W}$. The class $\overline{W}$ obviously contains the equivalences, is closed under composition, and under 2-isomorphic 1-cells, so it satisfies (BF1), (BF2), and (BF5).

 To see that it satisfies (BF3), let $g$ be a 1-cell such that $Fg$ is an equivalence, and let $f$ be an arbitrary 1-cell with the same codomain as $g$. By (EF2) there exists a 1-cell $w \in W \subseteq \overline{W}$ and a 1-cell $h$ such that $Fh \cong (Fg^{-1} \cdot Ff) \cdot Fw$. Composing with $Fg$ we get an isomorphism $F(g\cdot h)\cong F(f \cdot w)$, and from (EF3) we deduce the existence of the desired isomorphism $g\cdot h \cong f \cdot w$.

 To see that (BF4$'$) holds it suffices to observe that (EF3) implies that whiskering with a 1-cell in $\overline{W}$ gives a fully faithful functor on hom-categories in $\ca{C}$.
\end{proof}

\begin{prop}\label{prop:bf3}
 The class of surjective weak equivalences in the 2-category of affine groupoids with faithfully flat source and target maps satisfies (BF3).
\end{prop}

\begin{proof}
 Let $X$, $Y$ and $Z$ be affine groupoids whose source and target morphisms are faithfully flat. Let $w \colon X \rightarrow Z$ be a surjective weak equivalence and let $f \colon Y \rightarrow Z$ be an arbitrary morphism. Let $P$ be the comma object of $w$ and $f$ in the 2-category of affine groupoids. The comma object of a pair of internal functors is defined representably. In $\Cat$, the set of objects of a comma object consists of triples $(x,\varphi,y)$ where $\varphi \colon wx \rightarrow fy$, and the set of arrows consists of quadruples $(\alpha, \beta, \varphi, \psi) \in X_1 \times Y_1 \times Z_1 \times Z_1$ such that the diagram
\[
 \xymatrix{ws(\alpha) \ar[d]_{w\alpha} \ar[r]^-{\varphi} & fs(\beta) \ar[d]^{f\beta} \\
wt(\alpha) \ar[r]^-{\psi} & ft(\beta)}
\]
 is commutative. Since we are working with groupoids, the arrow $\psi$ is uniquely determined by the other three, so an arrow in the comma object is a triple $(\alpha,\beta, \varphi)$ subject to the matching conditions expressing that $\varphi$ is an arrow with domain $ws(\alpha)$ and codomain $fs(\beta)$. To internalize this we have to express these matching conditions using pullback diagrams. By doing this we find that the object of objects $P_0$ of $P$ and the object of arrows $P_1$ of $P$ fit in the commutative diagram
\[
 \xymatrix{P_1 \ar[r]^-{s} \ar[d] & P_0 \ar[r] \ar[d] & Z_1 \ar[d]^{(s,t)} \\
X_1\times Y_1 \ar[r]^-{s\times s} & X_0 \times Y_0 \ar[r]^{w_0\times f_0} & Z_0\times Z_0}
\]
 where both squares are pullbacks. This shows in particular that the source map of $P$ is faithfully flat. Since the source and target maps are isomorphic, it follows that $P$ lives in the desired subcategory of flat affine groupoids.

 It remains to show that the induced internal functor $w^{\prime} \colon P \rightarrow Y$ is a surjective weak equivalence. The object $P_0$ also fits in the following diagram
\[
 \xymatrix{P_0 \ar[r] \ar[d] & Z_1 \mathop{{\times}_{Z_0}} Y_0 \ar[r] \ar[d] & Y_0 \ar[d]^{f_0} \\
 X_0 \mathop{{\times}_{Z_0}} Z_1 \ar[r] \ar[d] & Z_1 \ar[d] \ar[r]^{t} & Z_0\\
X_0 \ar[r]^-{w_0} & Z_0
}
\]
 where all squares are pullbacks. This shows that $w^{\prime}_0$ is faithfully flat. We can check that $w^{\prime}$ is fully faithful representably, where it follows from the fact that for any diagram of solid arrows
\[
 \xymatrix{wx \ar@{..>}[d] \ar[r]^-{\varphi} & fy \ar[d]^{f\beta} \\
wx^{\prime} \ar[r]^-{\varphi^{\prime}} & fy^{\prime}}
\]
 there exists a unique arrow $\alpha \colon x \rightarrow x^{\prime}$ in $X$ such that $w\alpha$ in place of the dotted arrow makes the above square commutative.
\end{proof}

 Note that in this case, the class $W$ of surjective weak equivalences itself does not satisfy (BF1) and (BF5). A counterexample for (BF1) is given by a functor from the terminal groupoid to a chaotic groupoid $(X,X\times X)$ for some nontrivial affine scheme $X$ which admits a map $\ast \rightarrow X$. 

 The identity on $(X,X\times X)$ is then also naturally isomorphic to a functor which factors through the terminal groupoid. The former is a surjective weak equivalence while the latter is not, so (BF5) does not hold either.

\subsection{The comodule pseudofunctor is a bicategorical localization}\label{section:comod_as_localization}
 In this section and the next, we will prove Theorem~\ref{thm:comod_localization}. Since the category of Hopf algebroids is dual to the category of affine groupoids, we have to apply the dual version of Proposition~\ref{prop:localization}. Recall that $\ca{H}$ is the 2-category of flat Hopf algebroids, and that $\ca{T}$ is the 2-category of symmetric monoidal abelian categories, tame functors, and natural transformations between them. We have to show that the pseudofunctor
\[
 \Comod \colon \ca{H} \rightarrow \ca{T}
\]
 sends surjective weak equivalences to equivalences and satisfies the duals of the conditions (EF2) and (EF3). If we corestrict this pseudofunctor to its essential image, we get the desired localization by Proposition~\ref{prop:localization}. 

 We check the first two facts in the remainder of this section (see Proposition~\ref{prop:comod_weak_equivalences} and Corollary~\ref{cor:comod_ef2}). We will also prove some facts about tame functors which will be used in \S \ref{section:adams}. We check that $\Comod$ satisfies condition (EF3) in \S \ref{section:ftm} (see Proposition~\ref{prop:comod_ef3}).

 Throughout the remainder of this section we write
\[
 V \colon \Comod(A,\Gamma) \rightleftarrows \Mod_A \colon W
\]
 for the comonadic adjunction induced by $(A,\Gamma)$. Here $V$ is the forgetful functor, and $W$ is the functor which sends an $A$-module $M$ to $\Gamma \ten{A} M$.

 We first need to see that $\Comod$ sends surjective weak equivalences to equivalences.

\begin{prop}\label{prop:comod_weak_equivalences}
 The pseudofunctor $\Comod \colon \ca{H} \rightarrow \ca{T}$ sends surjective weak equivalences of flat Hopf algebroids to equivalences.
\end{prop}

\begin{proof}
 This follows from \cite[Theorem~5.5]{HOVEY_MORITA}, applied to the case $g=\id$, and \cite[Theorem~4.5]{HOVEY_MORITA}.
\end{proof}

 To stick with our goal to indicate that these conditions are true for arbitrary enrichment, we give a rough sketch which shows that the above proposition essentially follows from Beck's comonadicity theorem. The reader might want to skip straight to Proposition~\ref{prop:almost_ef2}, which is the key result for showing that (EF2) holds. It also allows us to find some properties of tame functors.

\begin{rmk}
 If $w \colon (A,\Gamma) \rightarrow (B,\Sigma)$ is a surjective weak equivalence, then $A \rightarrow B$ is faithfully flat. Thus the functor $B\ten{A}-$ is faithful and exact. From Beck's comonadicity theorem it follows that the composite
\[
 \xymatrix{\Comod(A,\Gamma) \ar[r]^-{V} & \Mod_A \ar[r]^{B\ten{A}-} & \Mod_B}
\]
 is comonadic. The right adjoints of both of these functors are cocontinuous, so the induced comonad is uniquely determined by where it sends $B$. The sequence
\[
 B \ten{A} VW(B)=B \ten{A} V(\Gamma \tenlr{\sigma}{A} B)=B \tenlr{A}{\tau} \Gamma \tenlr{\sigma}{A} B
\]
 shows that the comonad induced by $B \ten{A} V(-)$ is given by the Hopf algebroid $(B,B \tenlr{A}{\tau} \Gamma \tenlr{\sigma}{A} B)$. This Hopf algebroid is precisely the pushout in the second condition of being a surjective weak equivalence of Hopf algebroids (see Remark~\ref{rmk:weak_equivalences_hopf_algebroids}). This shows that
\[
 w_\ast \colon \Comod(A,\Gamma) \rightarrow \Comod(B,\Sigma)
\]
 is an equivalence of categories.
\end{rmk}

 \begin{prop}\label{prop:almost_ef2}
 Let $(A,\Gamma)$ and $(B,\Sigma)$ be flat Hopf algebroids, and let 
\[
F \colon \Comod(A,\Gamma) \rightarrow \Comod(B,\Sigma) 
\]
 be a strong symmetric monoidal left adjoint. Write $B^{\prime} \in \Mod_B$ for the image of the commutative monoid $\Gamma \in \Comod(A,\Gamma)$ under the functor $VF$. Then there exists a ring homomorphism $A \rightarrow B^{\prime}$ such that the diagram
\[
 \xymatrix{\Comod(A,\Gamma) \ar@{}[rrd]|{\cong} \ar[r]^-{F} \ar[d]_{V} & \Comod(B,\Sigma) \ar[r]^-{V} & \Mod_B \ar[d]^{B^{\prime} \ten{B} -} \\
 \Mod_A \ar[rr]_-{B^{\prime} \ten{A} -} & & \Mod_{B^{\prime}}}
\]
 commutes up to symmetric monoidal natural isomorphism.
 \end{prop}

 To prove this we need the following lemma.

\begin{lemma}\label{lemma:modules_monadic_over_comodules}
 For each flat Hopf algebroid $(A,\Gamma)$, the right adjoint
\[
 W \colon \Mod_A \rightarrow \Comod(A,\Gamma)
\]
 is monadic, and the induced symmetric monoidal monad on $\Comod(A,\Gamma)$ is isomorphic to the symmetric monoidal monad induced by the commutative monoid $\Gamma$ in the symmetric monoidal category $\Comod(A,\Gamma)$.
\end{lemma}

\begin{proof}
 Since the source morphism $A \rightarrow \Gamma$ is faithfully flat we know that $W$ is a faithful functor, so it reflects isomorphisms. It also preserves all colimits, so in particular coequalizers. Therefore it is monadic by Beck's monadicity theorem.

 The symmetric monoidal monad induced by the adjunction $V \dashv W$ is given by $WV \colon \Comod(A,\Gamma) \rightarrow \Comod(A,\Gamma)$. Since the adjunction is induced by a Hopf monoidal comonad, it is left and right coclosed, that is, there are natural isomorphisms
\[
 W(VM \otimes N) \cong M \otimes WN
\]
 for all $M \in \Comod(A,\Gamma)$ and all $N \in \Mod_A$ (see \cite[Theorem~4.3]{CHIKHLADZE_LACK_STREET}). In the terminology of \cite{FAUSK_HU_MAY}: the projection formula holds for this adjunction.

 Applying this to the case where $N$ is the unit object $A$ of $\Mod_A$, we get an isomorphism $WV(M) \cong M \otimes WA$, and we have $\Gamma=WA$ by definition of $W$.

 The (symmetric) monoidal structure and the monad structure on $WV$ both induce compatible multiplications on $\Gamma$, which must be commutative and equal to each other by the Eckmann-Hilton argument. The symmetric monoidal structure of $WV$ induces the usual commutative monoid structure on $\Gamma$. This shows that the above isomorphism is an isomorphism of symmetric monoidal monads.
\end{proof}

\begin{proof}[Proof of Proposition~\ref{prop:almost_ef2}]
 Let $\ca{M}=\Comod(A,\Gamma)$. Write $\ca{M}_{\Gamma}$ for the category of modules of the commutative monoid $\Gamma \in \ca{M}$. From Lemma~\ref{lemma:modules_monadic_over_comodules} we know that the adjunction $V \colon \ca{M} \rightleftarrows \Mod_A \colon W$ is monadic. Therefore there exists an equivalence of categories such that the diagram
\[
 \xymatrix{ \Mod_A \ar[rd] \ar[rr]^{\simeq} && \ar[ld] \ca{M}_{\Gamma} \\ & \ca{M}}
\]
 is commutative. We first claim that this equivalence is symmetric monoidal. To see this it suffices to check that the adjoint equivalence $\ca{M}_{\Gamma} \rightarrow \Mod_A$ is symmetric monoidal, and this can be further reduced to showing that it is so on the full subcategory of free $\Gamma$-modules. The latter is equivalent to the Kleisli-category of the symmetric monoidal monad $\Gamma \otimes -$, so the claim follows from the fact that the Kleisli category of a symmetric monoidal monad is a Kleisli object in the 2-category of strong symmetric monoidal functors.

 A basic observation about adjunctions valid in all 2-categories shows that the diagram
\[
 \xymatrix{ \Mod_A \ar[rr]^{\simeq} & \ar@{}[d]|{\cong} & \ca{M}_{\Gamma} \\ & \ar[lu]^{V} \ca{M} \ar[ru]_{\Gamma \otimes-}} 
\]
 commutes up to symmetric monoidal isomorphism.

 Now let $\ca{N}$ be any cocomplete symmetric monoidal closed category, and let $G \colon \ca{M} \rightarrow \ca{N}$ be a cocontinuous symmetric monoidal functor. Define a functor $\overline{G} \colon \ca{M}_{\Gamma} \rightarrow  \ca{N}_{G\Gamma}$ on objects by $\overline{G}(M)=GM$, with evident $G\Gamma$-module structure induced by the strong monoidal structure of $G$. The symmetric monoidal structure on the category of modules of a commutative monoid is defined as a coequalizer of the monoidal structure of the underlying symmetric monoidal category. Using the fact that $G$ preserves coequalizers we find that $\overline{G}$ is strong symmetric monoidal. The diagram
\[
 \xymatrix{\ca{M}_{\Gamma} \ar@{}[rd]|{\cong} \ar[r]^-{\overline{G}} & \ca{N}_{G\Gamma} \\
 \ca{M} \ar[u]^{\Gamma \otimes -} \ar[r]_-{G} & \ca{N} \ar[u]_{G\Gamma \otimes -}}
\]
 commutes up to symmetric monoidal equivalence given by the strong symmetric monoidal structure $\psi_{\Gamma,-} \colon G\Gamma \otimes G- \Rightarrow G(\Gamma \otimes-)$ of $G$. The functor $\overline{G}$ is cocontinuous because colimits in a category of modules are computed as in the underlying symmetric monoidal closed category. Applying these two facts to the case $\ca{N}=\Mod_B$ and $G=VF$ we find that there exists a cocontinuous strong symmetric monoidal functor $\overline{VF}$ such that the diagram
\[
 \xymatrix{\Comod(A,\Gamma) \ar@{}[rrd]|{\cong} \ar[r]^-{F} \ar[d]_{V} & \Comod(B,\Sigma) \ar[r]^-{V} & \Mod_B \ar[d]^{B^{\prime} \ten{B} -} \\
 \Mod_A \ar[rr]_-{\overline{VF}} & & \Mod_{B^{\prime}}} 
\]
 commutes up to symmetric monoidal isomorphism. The conclusion follows from the fact that, up to symmetric monoidal isomorphism, all cocontinuous strong symmetric monoidal functors between module categories are induced by scalar extension along some ring homomorphism.
\end{proof}

\begin{cor}\label{cor:image_of_gamma_ffl_implies_ef2}
 Let $(A,\Gamma)$ and $(B,\Sigma)$ be flat Hopf algebroids, and let
\[
 F \colon \Comod(A,\Gamma) \rightarrow \Comod(B,\Sigma)
\]
 be a cocontinuous strong symmetric monoidal functor. If $VF$ sends the commutative monoid $\Gamma \in \Comod(A,\Gamma)$ to a faithfully flat $B$-algebra, then there exists a Hopf algebroid $(B^{\prime},\Sigma^{\prime})$, a surjective weak equivalence $w \colon (B,\Sigma) \rightarrow (B^{\prime},\Sigma^{\prime})$ and a morphism of Hopf algebroids $f \colon (A,\Gamma) \rightarrow (B^{\prime},\Sigma^{\prime})$ such that the diagram
\[
 \xymatrix{\Comod(A,\Gamma) \ar[rd]_{f_\ast} \ar[rr]^-{F} & \ar@{}[d]|{\cong} & \Comod(B, \Sigma) \ar[ld]^{w_\ast} \\
& \Comod(B^{\prime},\Sigma^{\prime})}
\]
 commutes up to symmetric monoidal equivalence.
\end{cor}

\begin{proof}
 Let $B^{\prime}=VF(\Gamma)$. By assumption, it is a faithfully flat $B$-algebra. Let
\[
\xymatrix{\Sigma \ar[r]^{w} & \Sigma^{\prime} \\ B\times B \ar[u]^{\sigma \otimes \tau} \ar[r] & B^{\prime} \otimes B^{\prime} \ar[u]_{\sigma \otimes \tau}}
\]
 be a pushout diagram in the category of commutative rings. This defines a surjective weak equivalence $w \colon (B,\Sigma) \rightarrow (B^{\prime}, \Sigma^{\prime})$.

 Together with Proposition~\ref{prop:almost_ef2} we get a diagram
\[
 \xymatrix{\Comod(A,\Gamma) \ar@{}[rd]_{\cong} \ar[r]^-{F} \ar[d]_{V} & \Comod(B,\Sigma) \ar[d]^-{V} \ar[r]^-{w_{\ast}} & \Comod(B^{\prime},\Sigma^{\prime}) \ar[ldd]^{V} \\
 \Mod_A \ar[rd]_-{B^{\prime} \ten{A} -} & \Mod_{B} \ar[d]_{B^{\prime} \ten{B} -}\\ & \Mod_{B^{\prime}} }
\]
 which commutes up to symmetric monoidal isomorphism.

 Since comodule structures can be uniquely transferred along isomorphisms, the composite $w_{\ast} F$ is isomorphic to a symmetric monoidal functor $G$ which makes the diagram
\[
 \xymatrix{\Comod(A,\Gamma) \ar[d]_{V} \ar[r]^-G & \Comod(B^{\prime},\Sigma^{\prime}) \ar[d]^{V} \\
 \Mod_A \ar[r]^-{B^{\prime}\ten{A} -} & \Mod_{B^{\prime}} }
\]
 strictly commutative. Therefore $G$ is induced by a morphism of symmetric monoidal comonads (this is a general fact about monads and their categories of Eilenberg-Moore algebras, see \cite[\S\S2.1-2.2]{LACK_STREET_FTMII}), hence it is induced by a morphism of Hopf algebroids.
\end{proof}

 The following corollary gives a characterization of tame functors, which will also be used in \S \ref{section:adams} to show that all strong symmetric monoidal left adjoints between categories of comodules of Adams Hopf algebroids are tame.

\begin{cor}\label{cor:tame_characterization}
 Let $(A,\Gamma)$ and $(B,\Sigma)$ be flat Hopf algebroids. A functor
\[
 F \colon \Comod(A,\Gamma) \rightarrow \Comod(B,\Sigma)
\]
 is tame if and only if $VF$ sends the commutative algebra $\Gamma \in \Comod(A,\Gamma)$ to a faithfully flat $B$-algebra.
\end{cor}

 \begin{proof}
 Note that a comodule of a flat Hopf algebroid is flat (in the sense that tensoring with it is exact) if and only if its underlying module is flat. Indeed, since the right adjoint is faithful and exact it suffices to check that $W(VM\otimes -)$ is exact whenever $M$ is a flat comodule. This is an immediate consequence of the projection formula $W(VM\otimes -) \cong M\otimes W(-)$.

 Assume first that $F$ is tame. Note that $\Gamma \in \Comod(A,\Gamma)$ is a faithfully flat algebra in the sense of \cite[Definition~5.6]{LURIE}. The forgetful functor $V$ is exact and preserves flat objects by the above argument. Thus $V$, and therefore $VF$, are both tame. The conclusion follows since tame functors send faithfully flat algebras to faithfully flat algebras (see \cite[Remark~5.10]{LURIE}).

 Conversely, assume that $VF$ sends $\Gamma$ to a faithfully flat algebra. From Corollary~\ref{cor:image_of_gamma_ffl_implies_ef2} we know that, up to equivalence, $F$ is given by the functor induced by a morphism of Hopf algebroids. All such functors are tame because a comodule is flat if and only if its underlying module is.
 \end{proof}

\begin{cor}\label{cor:comod_ef2}
 The pseudofunctor
\[
 \Comod \colon \ca{H} \rightarrow \ca{T}
\]
  satisfies the dual of condition (EF2) of Proposition~\ref{prop:localization}.
\end{cor}

\begin{proof}
 This follows from Corollary~\ref{cor:image_of_gamma_ffl_implies_ef2} and Corollary~\ref{cor:tame_characterization}.
\end{proof}

 The 2-category $\ca{T}$ is large, because the category of tame functors between large categories need not be small. Nevertheless, the following corollaries show that there are no size issues in the essential image of the pseudofunctor $\Comod$.

\begin{cor}\label{cor:tame_preserves_presentability}
 Tame functors send $\lambda$-presentable comodules to $\lambda$-presentable comodules for all regular cardinals $\lambda$.
\end{cor}

\begin{proof}
 Since equivalences preserve $\lambda$-presentable comodules, it suffices to show this for functors induced by morphisms of Hopf algebroids. From \cite[Proposition~1.3.3]{HOVEY} we know that a comodule is $\lambda$-presentable if and only if its underlying module is. The claim follows from the fact that extension of scalars sends $\lambda$-presentable modules to $\lambda$-presentable modules.
\end{proof}

\begin{cor}\label{cor:tame_essentially_small}
 Let $(A,\Gamma)$ and $(B,\Sigma)$ be flat Hopf algebroids. The category of tame functors
\[
 \Comod(A,\Gamma) \rightarrow \Comod(B,\Sigma)
\]
 and symmetric monoidal natural transformations between them is essentially small. 
\end{cor}

\begin{proof}
 The category $\Comod(A,\Gamma)$ is locally $\lambda$-presentable for some regular cardinal $\lambda$ (see \cite{PORST}). Every tame functor preserves $\lambda$-filtered colimits, so up to isomorphism it is determined by its restriction to $\lambda$-presentable $(A,\Gamma)$-comodules. From Corollary~\ref{cor:tame_preserves_presentability} we know that this restriction factors through the essentially small category of $\lambda$-presentable $(B,\Sigma)$-comodules. This shows that the category of tame functors can be embedded in the category of functors between two essentially small categories.
\end{proof}

\subsection{Interlude on the formal theory of monads}\label{section:ftm}
 We have seen that the pseudofunctor $\Comod$ sends surjective weak equivalences to equivalences, and that it satisfies conditions (EF1) and (EF2) of Proposition~\ref{prop:localization_pronk}. It remains to show that it also satisfies (EF3), meaning that it is fully faithful on 2-cells. The functor $\Comod$ can be broken up into several pieces, and we will show that each of them is fully faithful on 2-cells. To do this it is convenient to take a slightly more abstract point of view.

 B\'enaubou \cite{BENABOU} observed that a small category can be thought of as a monad in the bicategory of spans of sets. The span is given by the source and target maps, the multiplication of the monad correponds to the composition function, and the unit corresponds to the function which sends an object to its identity morphism. More generally, a category $C$ internal to a finitely complete category $\ca{E}$ gives rise to a monad in the bicategory $\Span(\ca{E})$. The `formal theory of monads' developed by Street \cite{STREET_FTM} is the systematic study of monads in arbitrary 2-categories $\ca{K}$. In \cite{STREET_FTM}, Street defines a 2-category of monads in $\ca{K}$, and he defines \emph{Eilenberg-Moore objects}, which generalize the category of Eilenberg-Moore algebras of a monad in the case $\ca{K}=\Cat$. For $\ca{K}=\Span(\ca{E})$, a 1-cell of monads whose underlying span is of the form
\[
 \xymatrix@!=5pt{& D_0 \ar@{=}[ld] \ar[rd]^{f_0} \\ C_0 && D_0 } 
\]
 is precisely an internal functor from $C$ to $D$. On the other hand, natural transformations do not correspond to 2-cells of monads in the sense of \cite{STREET_FTM}. In the sequel \cite{LACK_STREET_FTMII}, Lack and Street proved that natural transformations of internal functors correspond to a more general kind of 2-cell between monad morphisms. Since we are interested in the dual situation of comonads, we recall the duals of the definitions from \cite{LACK_STREET_FTMII}.

\begin{dfn}[Lack-Street]
 Let $\ca{K}$ be a 2-category. A comonad on an object $A$ of $\ca{K}$ consists of a 1-cell $c \colon A \rightarrow A$, together with two 2-cells $\varepsilon \colon c \Rightarrow 1_A$ and $\delta \colon c \Rightarrow cc$, subject to the usual axioms. A 1-cell from a comonad $c$ on $A$ to a comonad $d$ on $B$ consists of a pair $(f,\varphi)$ of a 1-cell $f \colon A \rightarrow B$ and a 2-cell $\varphi \colon fc \Rightarrow df$ such that the two diagrams
\[
\vcenter{ \xymatrix{
 fc \ar[d]_{f\delta} \ar[rr]^-{\varphi} && df \ar[d]^{\delta f}\\
fcc \ar[r]^-{ \varphi c} & dfc \ar[r]^-{d\varphi } & ddf 
}} \quad \mbox{and} \quad
\vcenter{\xymatrix{
 fc \ar[rd]_{f\varepsilon} \ar[rr]^-{\varphi} && df \ar[ld]^{\varepsilon f} \\ 
& f 
}}
\]
 are commutative (these two notions already appear in \cite{STREET_FTM}). A 2-cell between two 1-cells $(f,\varphi)$ and $(g,\psi)$ of comonads consists of a 2-cell $\rho \colon fc \Rightarrow g$ in $\ca{K}$ such that the diagram
\[
 \xymatrix{fc \ar[r]^-{f\delta} \ar[d]_{f\delta} & fcc \ar[r]^-{\varphi c} & dfc \ar[d]^{d\rho} \\
 fcc \ar[r]^-{\rho c} & gc \ar[r]^-{\psi} & dg }
\]
 is commutative. The resulting 2-category is denoted by $\EM^c(\ca{K})$.

 If $\ca{K}$ is a bicategory, one has to include coherence isomorphisms in the evident places. This way one obtains a bicategory which we will again denote by $\EM^c(\ca{K})$
\end{dfn}

 In \cite{LACK_STREET_FTMII}, Lack and Street show that the category $\EM^c(\ca{K})$ is the free completion under Eilenberg-Moore objects for comonads. For us, the following facts are more important.

\begin{prop}\label{prop:cocategories_cospans_ff_2cells}
 Let $\ca{C}$ be a finitely cocomplete category, and let $\Cat^c(\ca{C})$ be the 2-category of cocategory objects in $\ca{C}$. The pseudofunctor
\[
 \Cat^c(\ca{C}) \rightarrow \EM^c\bigr(\Cospan(\ca{C})\bigl)
\]
 which sends a cocategory object
\[
 \xymatrix{ A \ar@<4pt>[r]^{\sigma} \ar@<-4pt>[r]_{\tau} & \ar[l]|{\varepsilon} \Gamma & \ar[l]|-{\delta} \ar@<5pt>[l] \ar@<-5pt>[l] \Gamma \mathop{{\cup}_A} \Gamma}
\]
 to the comonad
\[
 \xymatrix@!=5pt{& \Gamma \\ A \ar[ru]^{\tau} && \ar[lu]_{\sigma} A}  
\]
 in cospans is fully faithful on 2-cells.
\end{prop}

\begin{proof}
 This is dual to the extension of B\'enaubou's observation to the generalized 2-cells of monads introduced in \cite{LACK_STREET_FTMII} mentioned above. This dual situation of category objects internal to a finitely complete category and spans in that category is discussed in \cite[\S2.3]{LACK_STREET_FTMII}.
\end{proof}

\begin{lemma}\label{lemma:emc_ff_2cells}
 Let $F \colon \ca{K} \rightarrow \ca{L}$ be a pseudofunctor which is fully faithful on 2-cells. Then the induced pseudofunctor
\[
 \EM^c(F) \colon \EM^c(\ca{K}) \rightarrow \EM^c(\ca{L})
\]
 is fully faithful on 2-cells.
\end{lemma}

\begin{proof}
 The 2-cells in $\EM^c(\ca{K})$ are 2-cells $\rho$ in $\ca{K}$ such that one equation between two pasting diagrams hold. Since $F$ is fully faithful on 2-cells, this equation holds for $\rho$ in $\ca{K}$ if and only if it holds for $F\rho$ in $\ca{L}$. 
\end{proof}

 Recall from \cite{STREET_FTM} that an Eilenberg-Moore object for a comonad is universal among coactions of that comonad. In $\Cat$, the usual category of comodules of a comonad is an Eilenberg-Moore object of that comonad. If $c$ is a comonad on $A$ which has an Eilenberg-Moore object, we will denote that object by $v\colon A_c \rightarrow A$. The universal property of $A_c$ implies that $v$ has a right adjoint $w$, and that $c$ is equal to the comonad induced by this adjunction, that is, $c=v\cdot w$, with counit and comultiplication given by the counit and the unit of the adjunction.

 Now, given a 1-cell $(f,\varphi)$ between two comonads $c$ and $d$ which have Eilenberg-Moore objects, we obtain an induced coaction
\[
 \xymatrix{ fv \ar[r]^-{fv\eta} & fvwv \ar[r]^-{\varphi v} & dfv }
\]
 which, by the universal property of $B_d$, induces a 1-cell $\bar{f}$ making the diagram
\[
 \xymatrix{A_c \ar[d]_v \ar[r]^{\bar{f}} & B_d  \ar[d]^v \\
A \ar[r]^{f} & B } 
\]
 commutative. A 2-cell $\rho \colon (f,\varphi) \Rightarrow (g,\psi)$ of comonads induces a morphism
\[
 \xymatrix{fv \ar[r]^-{fv\eta} & fvwv \ar[r]^-{\rho v} & gv}
\]
 of the corresponding coactions, so by the universal property of Eilenberg-Moore objects it induces a 2-cell $\bar{\rho} \colon \bar{f} \Rightarrow \bar{g}$.

\begin{prop}\label{prop:emc_description_if_em_objects_exist}
 Let $\ca{K}$ be a 2-category with Eilenberg-Moore objects for comonads. Write $\ca{C}_{\ca{K}}$ for the following 2-category. An object of $\ca{C}_{\ca{K}}$ is a comonad in $\ca{K}$. A 1-cell $c \rightarrow d$ in $\ca{C}_{\ca{K}}$ is a pair $(f,\bar{f})$ making the square
\[
 \xymatrix{A_c \ar[d]_v \ar[r]^{\bar{f}} & B_d  \ar[d]^v \\
A \ar[r]^{f} & B}
\]
 commutative. A 2-cell $(f,\bar{f}) \Rightarrow (g,\bar{g})$ is a 2-cell $\bar{f} \Rightarrow \bar{g}$ in $\ca{K}$, subject to no compatibility conditions.

 The 2-functor
\[
 \EM^c(\ca{K}) \rightarrow \ca{C}_{\ca{K}}
\]
 which sends a comonad to itself, a 1-cell $(f,\varphi)$ to the pair $(f,\bar{f})$ defined above, and a 2-cell $\rho$ to $\bar{\rho}$ is an equivalence of 2-categories.
\end{prop}

\begin{proof}
 This is dual to \cite[\S2.2]{LACK_STREET_FTMII}.
\end{proof}

 We are now ready to show that the pseudofunctor $\Comod$ is fully faithful on 2-cells.

\begin{prop}\label{prop:comod_ef3}
 The 2-functor $\Comod$ is fully faithful on 2-cells.
\end{prop}

\begin{proof}
 We will show that this is true for all commutative bialgebroids, that is, the antipode is not needed for this result. Note that a commutative bialgebroid is precisely a cocategory object in the category of commutative rings. 

 Let $\ca{K}$ be the 2-category whose objects are rings, with 1-cells $A \rightarrow B$ the cocontinuous symmetric monoidal functors $\Mod_A \rightarrow \Mod_B$, and 2-cells the symmetric monoidal natural transformations. Let $\ca{A}$ be the category of commutative rings.

 Recall from \cite[Proposition~10.1.1]{SCHAEPPI} that the pseudofunctor
\[
F \colon \Cospan(\ca{A}) \rightarrow \ca{K} 
\]
 which sends a cospan
\[
 \xymatrix@!=5pt{& B \\ A \ar[ru] && \ar[lu] C }  
\]
 of commutative rings to the symmetric monoidal functor
\[
 B \ten{A} - \colon \Mod_A \rightarrow \Mod_C
\]
 gives a biequivalence between the bicategory $\Cospan(\ca{A})$ and the 2-category $\ca{K}$. It is in particular fully faithful on 2-cells, which is all we need here.

 Let $\ca{C}_{\ca{K}}$ be as in Proposition~\ref{prop:emc_description_if_em_objects_exist}. There is a 2-functor
\[
 \ca{C}_{\ca{K}} \rightarrow \ca{K}
\]
 which sends a comonad in $\ca{K}$ to its Eilenberg-Moore object, which in this case is the category of comodules of the comonad with the usual symmetric monoidal structure. It sends a 1-cell $(f,\bar{f})$ to $\bar{f}$, and a 2-cell to itself. We have now written the pseudofunctor $\Comod$ as the composite
\[
 \xymatrix{\Cat^c(\ca{A}) \ar[r] &
\EM^c\bigl(\Cospan(\ca{A})\bigr) \ar[r]^-{\EM^c(F)} & \EM^c(\ca{K}) \ar[r]^-{\simeq} & \ca{C}_{\ca{K}} \ar[r] & \ca{K}}
\]
 of four pseudofunctors. We need to show that all these pseudofunctors are fully faithful on 2-cells. For the first pseudofunctor, this was shown in Proposition~\ref{prop:cocategories_cospans_ff_2cells}, and for the second pseudofunctor it was shown in Lemma~\ref{lemma:emc_ff_2cells}. The third pseudofunctor is an equivalence by Proposition~\ref{prop:emc_description_if_em_objects_exist}, and the fourth pseudofunctor is fully faithful by definition of 2-cells in $\ca{C}_{\ca{K}}$ (see Proposition~\ref{prop:emc_description_if_em_objects_exist}).
\end{proof}

 This allows us to show that the corestriction of $\Comod \colon \ca{H} \rightarrow \ca{T}$ to its essential image is a bicategorical localization of $\ca{H}$ at the surjective weak equivalences.

\begin{proof}[Proof of Theorem~\ref{thm:comod_localization}]
 We have to check the dual conditions of Proposition~\ref{prop:localization}. That the class of surjective weak equivalences satisfies the dual conditions of (BF3) follows from Proposition~\ref{prop:bf3}, where this was checked for the opposite category of flat affine groupoids. The pseudofunctor $\Comod$ sends surjective weak equivalences to equivalences of categories by Proposition~\ref{prop:comod_weak_equivalences}. It is essentially surjective on objects by assumption, so (EF1) holds. The dual statements of (EF2) and (EF3) are the contents of Corollary~\ref{cor:comod_ef2} and Proposition~\ref{prop:comod_ef3} respectively.
\end{proof}

 From Theorem~\ref{thm:comod_localization} it follows that all symmetric monoidal natural transformations between tame functors are invertible. The following corollary of Proposition~\ref{prop:comod_ef3} shows how this can be deduced directly.

 \begin{cor}\label{cor:tame_natural_transformation_invertible}
 Let $(A,\Gamma)$ and $(B,\Sigma)$ be two flat Hopf algebroids. Any symmetric monoidal natural transformation between two tame functors
\[
 F,G \colon \Comod(A,\Gamma) \rightarrow \Comod(B,\Sigma)
\]
 is invertible.
 \end{cor}

 \begin{proof}
 Such a 2-cell is invertible if and only if its composite with an equivalence is. Applying Corollary~\ref{cor:image_of_gamma_ffl_implies_ef2} twice we reduce to the case where both $F$ and $G$ are induced by morphisms of Hopf algebroids. In this case, the 2-cell is the image of a 2-cell of Hopf algebroids under the pseudofunctor $\Comod$ (see Proposition~\ref{prop:comod_ef3}). Since Hopf algebroids are cogroupoid objects, any natural transformation between two 1-cells of Hopf algebroids is invertible.
 \end{proof}

%% file: stackslocalization.tex
\section{Stacks as a bicategorical localization}\label{section:stackslocalization}

 In this section we will prove Theorem~\ref{thm:stacks_localization}, which states the pseudofunctor which sends a flat affine groupoid to its associated stack exhibits the 2-category of algebraic stacks as a bicategorical localization of the 2-category of flat affine groupoids at the surjective weak equivalences. We start with some generalities about stacks associated to internal groupoids.

\subsection{Stacks associated to internal groupoids}\label{section:groupoids}

 Throughout this section, we write $\Gpd(\ca{C})$ for the 2-category of internal groupoids in a category $\ca{C}$. It is convenient to use Street's construction of the associated stack pseudofunctor (see \cite{STREET_SHEAF, STREET_STACK}). We follow the summary given in \cite{STREET_DESCENT}. Let $\ca{C}$ be a finitely complete category with a Grothendieck pretopology consisting of singleton coverings. In our case, $\ca{C}$ will be the category $\Aff$ of affine schemes, and the covering morphisms are the faithfully flat morphisms $p \colon V \rightarrow U$ of affine schemes. Each covering morphism $p$ as above gives rise to an equivalence relation 
\[
 \Er(p) \defl \xymatrix{V\pb{U} V \ar@<2pt>[r] \ar@<-2pt>[r] & V}
\]
 in $\ca{C}$, where $V\pb{U} V$ is the pullback of $p$ along itself. An equivalence relation can also be thought of as a groupoid with at most one morphism between any two objects. From this point of view, $\Er(p)$ is the groupoid object
\[
 \xymatrix{V\pb{U} V \pb{U} V \ar@<5pt>[r] \ar[r] \ar@<-5pt>[r] & V\pb{U} V \ar@<4pt>[r] \ar@<-4pt>[r] & U \ar[l]}
\]
 in $\ca{C}$. The morphism $p \colon V \rightarrow U$ induces an internal functor $p \colon \Er(p) \rightarrow U$ from $\Er(p)$ to the discrete groupoid $U$.

 Let
\[
\ca{F}=\Ps[\ca{C}^{\op},\Gpd]
\]
 be the 2-category with objects the pseudofunctors $\ca{C}^{\op} \rightarrow \Gpd$, 1-cells the pseudonatural transformations between them and 2-cells the modifications. Recall from \cite[\S4, p.~12]{STREET_DESCENT} that there is a Yoneda-like 2-functor
\[
 \Gpd(\ca{C}) \rightarrow \ca{F}
\]
 which sends an internal groupoid $X$ to the 2-functor $\ca{C}(-,X)$, and that this 2-functor is an equivalence on hom-categories. On discrete groupoids this gives the usual Yoneda embedding of $\ca{C}$ in the category of $\Set$-valued presheaves. We shall not distinguish an internal groupoid from the presheaf of groupoids that it represents notationally.

 \begin{dfn}\label{dfn:stacks}
  A pseudofunctor $F \in \ca{F}$ is called \emph{1-separated} if for all covering morphisms $p \colon V \rightarrow U$, the induced functor
\[
 p^{\ast} \colon \ca{F}(U,F) \rightarrow \ca{F}\bigl(\Er(p),F\bigr)
\]
 is faithful. It is called \emph{2-separated} respectively a \emph{stack} if $p^{\ast}$ is fully faithful respectively an equivalence for all covering morphisms $p \colon V \rightarrow U$.

 Let $\ca{S}$ denote the full sub-2-category of $\ca{F}$ consisting of stacks.
 \end{dfn}

 Street's construction of the associated stack is very similar to the construction of the associated sheaf functor on a Grothendieck site. In \cite{STREET_SHEAF, STREET_STACK}, Street constructs a pseudofunctor
\[
 L \colon \ca{F} \rightarrow \ca{F}
\]
 together with a pseudonatural transformation $\eta_F \colon F \rightarrow LF$ with the properties that $LF$ is always 1-separated, that it is 2-separated whenever $F$ is 1-separated, and that it is a stack when $F$ is 2-separated. In these cases the pseudonatural transformation $\eta_F$ gives a reflection into the sub-category of 1-separated presheaves, 2-separated presheaves, respectively stacks. Thus one obtains the stack associated to a groupoid valued pseudofunctor by applying $F$ three times.

 The value of $LF$ at an object $U \in \ca{C}$ is given by the filtered bicolimit
\[
 \colim \ca{F}\bigr(\Er(p),F\bigl)
\]
 taken over the category of all coverings $p \colon V \rightarrow U$ of $U$. In our case, this category is large. To avoid technical difficulties we therefore restrict attention to a full subcategory of $\Aff$ consisting of objects whose ring of regular functions is contained in some Grothendieck universe. Corollary~\ref{cor:tame_essentially_small} suggests that this is merely a technical convenience, not a necessity. 

 Since $L$ is given by a filtered bicolimit, this construction makes it very clear that $L$ preserves (weighted) finite limits. For this reason, the 2-category of stacks inherits all the nice exactness properties that $\Gpd$ has. Most importantly, there is a (bicategorical) factorization system on the 2-category of stacks, induced by the factorization of a functor into a functor which is essentially surjective on objects (eso) followed by a functor which is fully faithful (ff). In $\Gpd$ (or in $\Cat$), every eso functor exhibits its target as codescent object of its higher kernel (it suffices to check this for the corresponding strict notions of higher kernels and functors which are bijective on objects, which is done in the proof of \cite[Proposition~3]{STREET_DESCENT}). This property extends to $\ca{F}$ because finite bilimits and bicolimits are computed pointwise, and it holds for the 2-category of stacks as well because $L$ is a left exact bireflection. Similarly, a morphism of stacks is an equivalence if and only if it is both eso and ff.

 Note that the $\fpqc$-topology is not itself a topology generated by singleton coverings. But the $\fpqc$-site $\Aff$ is a \emph{superextensive site}, meaning that the category $\Aff$ is extensive, and the Grothendieck topology is generated by the extensive topology (the topology generated by finite coproduct inclusions)
 and the subtopology generated by singleton coverings consisting of a faithfully flat morphism. This notion is due to Bartels and Shulman (see the nLab page on superextensive sites (version 6) \cite{NLAB}). It is also discussed in \cite{ROBERTS}. We defer the proof of the following two propositions to Appendix~\ref{section:superextensive}.

\begin{cit}[Proposition~\ref{prop:extensive_stack}]
 Let $\ca{C}$ be an extensive category. Then a pseudofunctor
\[
 \ca{C}^{\op} \rightarrow \Gpd
\]
 is a stack for the extensive topology if and only if it sends finite coproducts in $\ca{C}$ to biproducts.  
\end{cit}

 In general it is not true that the associated stack on a site generated by two Grothendieck pretopologies can be computed in steps, but it is true in this particular case (the nLab page on superextensive sites (version 6) \cite{NLAB} contains a proof of the corresponding statement for sheaves).

\begin{cit}[Proposition~\ref{prop:superextensive}]
 Let $\ca{C}$ be a superextensive site, and let $F \colon \ca{C} \rightarrow \Gpd$ be a pseudofunctor which is a stack for the extensive topology. Then the associated stack of $F$ for the subtopology generated by singletons is also the associated stack for the original topology.
\end{cit}
 
 All the pseudofunctors we are interested in are represented by groupoids in the category $\Aff$. Thus they send coproducts in $\Aff$ to products in $\Gpd$, and the latter are also bicategorical products. To compute the stack associated to a groupoid we can therefore use the pseudofunctor $L$ defined above. We will use the following two lemmas to check that $L$ is a bicategorical localization.

 \begin{lemma}\label{lemma:groupoids_2_separated}
  Let $X$ be an internal groupoid in a site $\ca{C}$ whose topology is generated by singleton coverings consisting of regular epimorphisms. Then $X$ is 2-separated, and the unit $\eta_X \colon X \rightarrow LX$ is fully faithful.
 \end{lemma}

\begin{proof}
 Let $p \colon V \rightarrow U$ be a covering morphism. Since the Yoneda embedding of internal groupoids is an equivalence on hom-categories it suffices to check that
\[
 p^{\ast} \colon \Gpd(\ca{C})\bigl(U,X \bigr) \rightarrow \Gpd(\ca{C})\bigr(\Er(p),X\bigl)
\]
 is fully faithful. Thus let $f,g \colon U \rightarrow X$ be internal functors, and let $\beta$ be an internal natural transformation between $p^{\ast} f$ and $p^{\ast} g$. Since $U$ is discrete, the internal functors $f$ and $g$ are given by $(f,ef)$ and $(g,eg)$ respectively, where $e \colon X_0 \rightarrow X_1$ represents the function which sends an object to its identity morphism. Writing this in a diagram, we have
\[
 \xymatrix@!=40pt{ V \pb{U} V \ar[r]^-{q} \ar@<-2pt>[d]_{s} \ar@<2pt>[d]^{t} & U \ar@<2pt>[r]^{ef} \ar@<-2pt>[r]_{eg} & X_1 \ar@<-2pt>[d] \ar@<2pt>[d] \\
 V \ar[rru]^{\beta} \ar[r]^-{p} & U \ar@{..>}[ru] \ar@<2pt>[r]^{f} \ar@<-2pt>[r]_{g} & X_0}
\]
 and we would like to show that there exists a unique dotted arrow giving an internal natural transformation $f \Rightarrow g$ and making the diagram commutative. Representably we are dealing with a natural transformation between two functors which send all morphisms to identities. Therefore the component at the source and at the target of any morphism must coincide. To internalize this argument, write $c$ for the composition morphism of $X$. Then naturality of $\beta$ says that the equation
\[
 c \cdot (egq, \beta s)=c \cdot (\beta t, efq)
\]
 holds. The morphism $e$ assigns an object to its identity morphism, so this simplifies to $\beta s=\beta t$. Since $p$ is a regular epimorphism it is the coequalizer of its kernel pair. It follows that the dotted arrow does indeed exist, and the naturality axioms follow immediately from the fact that $p$ is an epimorphism.

 To check that $\eta_X \colon X \rightarrow LX$ is fully faithful we can work pointwise. The $U$-component of the unit can be computed as the bicolimit of the pseudonatural transformation
\[
 p^{\ast} \colon \Gpd(\ca{C})(U,X) \rightarrow \Gpd(\ca{C})\bigl(\Er(p),X)
\]
 from the constant diagram to the defining diagram of $LX(U)$. We have shown that the components $p^{\ast}$ of this pseudonatural transformation are fully faithful. Since the bicolimit in question is filtered, it follows that the induced morphism
\[
 X(U) = \Gpd(\ca{C})(U,X) \rightarrow LX(U)
\]
 is fully faithful as well.
\end{proof}

\begin{lemma}\label{lemma:discrete_groupoids_are_stacks}
 Let $\ca{C}$ be a superextensive site whose singleton coverings are regular epimorphisms. Then every discrete groupoid is a stack.
\end{lemma}

\begin{proof}
 Proposition~\ref{prop:superextensive} reduces the problem to showing that a discrete groupoid is a stack for the subtopology of singleton coverings.

 Let $W \in \ca{C}$, and let $p \colon V \rightarrow U$ be a covering. Since $W$ is discrete, a functor $\Er(p) \rightarrow W$ is uniquely determined by a morphism $V \rightarrow W$ which equalizes $s,t \colon V\pb{U} V \rightarrow V$. The claim follows from the fact that covering morphisms are regular epimorphisms.
\end{proof}

 The following proposition shows that, for a certain class of sites, stacks associated to internal groupoids can be computed by a single application of Street's pseudofunctor $L$.

\begin{prop}\label{prop:associated_stack}
  Let $\ca{C}$ be a superextensive site for which the singleton coverings are regular epimorphisms, and let $L$ be the pseudofunctor defined above using the subtopology generated by singleton coverings. Then the stack associated to $X \in \Gpd(\ca{C})$ (for the full topology) is given by a single application of $L$ to $X$.
\end{prop}

\begin{proof}
 From Proposition~\ref{prop:extensive_stack} and Proposition~\ref{prop:superextensive} we know that the associated stack to $X$ for the superextensive topology is given by the associated stack for the subtopology generated by singleton coverings. The claim follows immediately from the fact that $X$ is 2-separated (see Lemma~\ref{lemma:groupoids_2_separated}).
\end{proof}

\subsection{Proof that stacks form a bicategorical localization}\label{section:stacks_as_localization}

 To prove Theorem~\ref{thm:stacks_localization}, which states that the associated stack pseudofunctor $L$ is a bicategorical localization of the 2-category of flat affine groupoids at the internal weak equivalences, we will check that $L$ satisfies the conditions of Proposition~\ref{prop:localization_pronk}. By definition of algebraic stacks, $L$ is essentially surjective on objects, hence $L$ satisifes (EF1). The following proposition shows that it also satisfies condition (EF3).

 \begin{prop}\label{prop:L_ef3}
 Let $\ca{C}$ be a superextensive site for which the singleton coverings are regular epimorphisms. Then the associated stack pseudofunctor
\[
 L \colon \Gpd(\ca{C}) \rightarrow \ca{S}
\]
 defined in \S \ref{section:groupoids} is fully faithful on 2-cells.
 \end{prop}

\begin{proof}
 That the associated stack pseudofunctor is given by a single application of $L$ is the content of Proposition~\ref{prop:associated_stack}.

 Let $\beta \colon Lf \Rightarrow Lg \colon LX \rightarrow LY$ be a 2-cell. From Lemma~\ref{lemma:groupoids_2_separated} we know that the $\eta_X \colon X \rightarrow LX$ and $\eta_Y \colon Y \rightarrow LY$ are fully faithful. Therefore there exists a unique 2-cell $\alpha \colon f \Rightarrow g$ such that the equation
\[
 \vcenter{\xymatrix@!=30pt{
X \dtwocell_{f}^{g}{^\alpha} \ar[r]^{\eta_X} \ar@{}[d]|{\quad\quad\quad\quad\quad\quad\cong} & LX \ar[d]^{Lg} \\
Y \ar[r]_{\eta_Y} & LY
}}
\quad \; = \;
\vcenter{\xymatrix@!=30pt{
X \ar[d]_{f} \ar[r]^{\eta_X} \ar@{}[d]|{\quad\quad\quad\quad\cong} & LX \dtwocell_{Lf}^{Lg}{^\beta} \\
Y \ar[r]_{\eta_Y} & LY
}}
\]
 of pasting diagrams holds. Applying pseudonaturality of $\eta$ to the left hand side we find that the equation
\[
 \vcenter{\xymatrix@!=30pt{
X \ar[d]_{f} \ar[r]^{\eta_X} \ar@{}[d]|{\quad\quad\quad\quad\cong} & LX \dtwocell_{Lf}^{Lg}{^L\alpha} \\
Y \ar[r]_{\eta_Y} & LY
}} 
\quad \; = \;
 \vcenter{\xymatrix@!=30pt{
X \ar[d]_{f} \ar[r]^{\eta_X} \ar@{}[d]|{\quad\quad\quad\quad\cong} & LX \dtwocell_{Lf}^{Lg}{^\beta} \\
Y \ar[r]_{\eta_Y} & LY
}} 
\]
 holds. The fact that $\eta_X$ is a bireflection implies that $L\alpha=\beta$. The conclusion follows because $\alpha$ was uniquely determined by $\beta$.
\end{proof}

 The following lemma will be used to show that $L$ sends surjective weak equivalences to weak equivalences and that it satisfies condition (EF2).

\begin{lemma}\label{lemma:presentation}
 Let $\ca{C}$ be a superextensive site whose singleton coverings are regular epimorphisms. Let $X \in \Gpd(\ca{C})$, and let $X_0 \rightarrow X$ be the internal functor from the discrete groupoid $X_0$ to $X$ which is the identity on objects. Then the higher kernel of the composite
\[
 \xymatrix{X_0 \ar[r] & X \ar[r]^-{\eta_X} & LX}
\]
 is
\[
 \xymatrix{X_2 \ar@<5pt>[r] \ar[r] \ar@<-5pt>[r] & X_1 \ar@<4pt>[r] \ar@<-4pt>[r] & X_0 \ar[l]} \smash{\rlap{,}} 
\]
 and $X_0 \rightarrow LX$ exhibits $LX$ as codescent object of its higher kernel in the 2-category $\ca{S}$ of stacks. It is in particular eso.
\end{lemma}

\begin{proof}
 We first show that the diagram in question is the higher kernel of $X_0 \rightarrow X$ in $\ca{F}$, and that $X_0 \rightarrow X$ exhibits $X$ as its codescent object. This can be checked pointwise, where it follows from the fact that any small category $C$ is the codescent object of its higher kernel
\[
 \xymatrix{C_2 \ar@<5pt>[r] \ar[r] \ar@<-5pt>[r] & C_1 \ar@<4pt>[r] \ar@<-4pt>[r] & C_0 \ar[l]} \smash{\rlap{,}}
\]
 where $C_2=C_1 \pb{C_0} C_1$ (see the proof of \cite[Proposition~3]{STREET_DESCENT}). By Lemma~\ref{lemma:discrete_groupoids_are_stacks} the higher kernel of $X_0 \rightarrow X$ is in $\ca{S}$. Since $\ca{S} \subseteq \ca{F}$ is bireflective, bicolimits are computed by applying the bireflection to the bicolimit in $\ca{F}$. For $X$ and the $X_i$, the bireflection is given by a single application of $L$ (see Proposition~\ref{prop:L_ef3}) Thus $LX$ is indeed the desired codescent object. The fact that $L$ is left exact implies that
\[
 \xymatrix{X_2 \ar@<5pt>[r] \ar[r] \ar@<-5pt>[r] & X_1 \ar@<4pt>[r] \ar@<-4pt>[r] & X_0 \ar[l]} 
\]
 is the higher kernel of $X_0 \rightarrow LX$.
 
 A morphism in $\ca{S}$ is eso if and only if it exhibits its codomain as a codescent object of its higher kernel, because the same is true in $\Cat$ (cf.\ \cite[Proposition~3]{STREET_DESCENT}).
\end{proof}

 We can generalize the definition of surjective weak equivalences of groupoids to any site. The pullback condition is independent of the topology, and the condition that the object part of the internal functor is faithfully flat is replaced by the condition that it is covering. 

\begin{prop}\label{prop:L_weak_equivalences}
 Let $\ca{C}$ be a superextensive site whose singleton coverings are regular epimorphisms. Then the pseudofunctor
\[
 L \colon \Gpd(\ca{C}) \rightarrow \ca{S}
\]
 sends surjective weak equivalences to equivalences.
\end{prop}

\begin{proof}
 Let $w \colon X \rightarrow Y$ be a surjective weak equivalence, that is, it is internally fully faithful and $w_0 \colon X_0 \rightarrow Y_0$ is covering. The first implies that the represented morphism $w$ in $\ca{F}$ is fully faithful. Equivalently, the higher kernel of $f$ is trivial. Since $L$ commutes with finite bilimits, it follows that $Lw$ is fully faithful. It remains to show that it is also eso.

 From Lemma~\ref{lemma:presentation} we know that the top and the bottom composite of the diagram
\[
 \xymatrix{X_0 \ar@{}[rrd]|{\cong} \ar[d]_{w_0} \ar[r] & X \ar[r]^-{\eta_X} & LX \ar[d]^{Lw} \\
Y_0 \ar[r] & Y \ar[r]^-{\eta_Y} & LY}
\]
 are eso. In Lemma~\ref{lemma:discrete_groupoids_are_stacks} we have seen that $Y_0$ is a stack. From the definition of stacks it follows that $w_0 \colon \Er(w_0) \rightarrow Y_0$ is a bireflection of $\Er(w_0)$ into the 2-category of stacks. Thus the composite
\[
 \xymatrix{ X_0 \ar[r] & \Er(w_0) \ar[r] & Y_0 }
\]
 is eso by Lemma~\ref{lemma:presentation}. But this composite is $w_0$, so we have shown that the left arrow as well as the top and bottom arrows of the above diagram are eso. It follows that the arrow on the right is eso, because the eso-ff factorization system in $\Cat$ has the analogous cancellation property.
\end{proof}

 To show that $L$ has all the properties of a bicategorical localization it remains to check that it satisfies condition (EF2) of Proposition~\ref{prop:localization_pronk}. Note that so far we have not made use of the fact that we are considering the 2-category of \emph{flat} affine groupoids, that is, groupoids whose source and target morphisms are flat. Since both these morphisms have a (common) section, they are automatically faithfully flat. The stack associated to such a groupoid $X$ has two important properties: its diagonal is representable, and the the morphism $X_0 \rightarrow LX$ is faithfully flat. The meaning of these two statements is spelled out in the proposition below.

 We could probably continue to work with a more general superextensive site, as long as the singleton coverings are effective descent morphisms, and consider groupoids whose source and target morphisms are covering. However, there are some technical details that the author has not checked. Therefore we henceforth only consider the site $\Aff$ with the $\fpqc$-topology.

\begin{prop}\label{prop:affine_diagonal_and_ffl_presentation}
 Let $\ca{C}=\Aff$, with the $\fpqc$-topology. Let $X$ be an affine groupoid with flat source and target maps. Then the diagonal of the associated stack $LX$ is representable, that is, in any bipullback square
\[
 \xymatrix{P \ar[r] \ar[d] \ar@{}[rd]|\cong & LX \ar[d]^{\Delta} \\ U \ar[r] & LX \times LX}
\]
 where $U \in \Aff$, the object $P$ is equivalent to some $V \in \Aff$. In particular, for any morphism $U \rightarrow LX$ we can find a bipullback square
\[
 \xymatrix{V \ar[r] \ar[d]_{p} \ar@{}[rd]|\cong & X_0 \ar[d] \\ U \ar[r] & LX} 
\]
 where $V \in \Aff$ and $X_0 \rightarrow LX$ is the morphism from Lemma~\ref{lemma:presentation}. 

 The morphism $X_0 \rightarrow LX$ is faithfully flat, that is, the morphism $p\colon V \rightarrow U$ in the above diagram is faithfully flat for all morphisms $U \rightarrow LX$.
\end{prop}

\begin{proof}
 This is proved in \cite[\S3.3]{NAUMANN} after the definition of the functor $G$.
\end{proof}

 The following proposition shows that $L$ satisfies condition (EF2) of Proposition~\ref{prop:localization_pronk}.

\begin{prop}\label{prop:L_ef2}
 Let $X$, $Y$ be flat affine groupoids, and let $f \colon LX \rightarrow LY$ be a morphism of stacks. Then there exists a flat affine groupoid $Z$, a surjective weak equivalence $w \colon Z \rightarrow X$ and a morphism $g \colon Z \rightarrow Y$ of groupoids such that the diagram
\[
 \xymatrix{ & \ar@{}[d]|{\cong} \ar[ld]_{Lw} LZ \ar[rd]^{Lg} \\ LX \ar[rr]_f && LY}
\]
 commutes up to an invertible 2-cell.
\end{prop}

\begin{proof}
 By Proposition~\ref{prop:affine_diagonal_and_ffl_presentation} we can find $Z_0 \in \Aff$ which fits into a bipullback diagram
\[
  \xymatrix{Z_0 \ar[rr] \ar[d]_{w_0} \ar@{}[rrd]|\cong & & Y_0 \ar[d] \\ X_0 \ar[r]_{p} & LX \ar[r]_{f} & LY}
\]
 where $w_0$ is faithfully flat. Let
\[
 \xymatrix{Z_2 \ar@<5pt>[r] \ar[r] \ar@<-5pt>[r] & Z_1 \ar@<4pt>[r] \ar@<-4pt>[r] & Z_0 \ar[l]}
\]
 be the higher kernel of $pw_0$. By Proposition~\ref{prop:affine_diagonal_and_ffl_presentation} we can arrange for both $Z_1$ and $Z_2$ to lie in $\Aff$. Since there are no nontrivial 2-cells between discrete groupoids we get a truncated \emph{simplicial} object, not just truncated pseudosimplicial object. We claim that this higher kernel is a groupoid object in $\Aff$. To see this we have to check that various squares of the truncated simplicial diagram are pullback squares, and since we are dealing with discrete objects it suffices to check that they are bipullback squares. Since higher kernels and bipullback squares are computed as in $\ca{S}$, it suffices to check the claim in the category of groupoids. There it follows from the fact that for any functor $f \colon A \rightarrow B$ between groupoids, the diagram
\[
 \xymatrix{(f\downarrow f \downarrow f) \ar@<5pt>[r] \ar[r] \ar@<-5pt>[r] & (f\downarrow f) \ar@<4pt>[r] \ar@<-4pt>[r] & A \ar[l]}
\]
 is a higher kernel of $f$. The source and target morphisms of the groupoid $Z$ are faithfully flat by the second part of Proposition~\ref{prop:affine_diagonal_and_ffl_presentation}.

 Since the formation of higher kernels is pseudofunctorial, we also get an extension of $w_0$ to an internal functor $Z \rightarrow X$. Moreover, one way to compute $Z_1$ is as bipullback of the diagonal along the product of two copies of $pw_0$. We can do this by pasting two bipullback squares as in
\[
 \xymatrix{Z_1 \ar@{}[rd]|{\cong} \ar[r]^{w_1} \ar[d] & X_1 \ar@{}[rd]|{\cong} \ar[r] \ar[d] & LX \ar[d]^{\Delta} \\ 
Z_0 \times Z_0 \ar[r]_{w_0 \times w_0} & X_0 \times X_0 \ar[r]_{p\times p} & LX \times LX}
\]
 where the left square can be computed as a strict pullback in $\Aff$ since all the groupoids involved are discrete. This shows that $w$ is a surjective weak equivalence.

 Similarly we get an internal functor $g \colon Z \rightarrow Y$ induced by $f$. The morphism $w_0$ is an eso morphism of stacks since it is the bipullback of an eso morphism (see Lemma~\ref{lemma:presentation}). The same lemma tells us that $p$ is eso. Therefore the composite $pw_0$ is eso, and it follows that $LX$ is a codescent object of the diagram
\[
 \xymatrix{Z_2 \ar@<5pt>[r] \ar[r] \ar@<-5pt>[r] & Z_2 \ar@<4pt>[r] \ar@<-4pt>[r] & Z_0 \ar[l]}
\]
 in the 2-category $\ca{S}$ of stacks. In other words, $LX$ gives \emph{a} bireflection of $Z$ into stacks, and $\id_X$ respectively $f$ give \emph{a} morphism compatible with the internal functor $w \colon Z \rightarrow X$ respectively $g \colon Z \rightarrow Y$. From the universal property of bireflections we find that the diagram
\[
 \xymatrix{ & \ar@{}[d]|{\cong} \ar[ld]_{Lw} LZ \ar[rd]^{Lg} \\ LX \ar[rr]_f && LY}
\]
 commutes up to isomorphism.
\end{proof}

 We are now ready to prove that the pseudofunctor $L$ exhibits the 2-category of algebraic stacks as a bicategorical localization of the category of flat affine groupoids at the surjective weak equivalences.

\begin{proof}[Proof of Theorem~\ref{thm:stacks_localization}]
 We have to check the conditions of Proposition~\ref{prop:localization}. The pseudofunctor $L$ sends surjective weak equivalences to equivalences by Proposition~\ref{prop:L_weak_equivalences}, and the class of surjective weak equivalences satisfies condition (BF3) by Proposition~\ref{prop:bf3}. It remains to check that conditions (EF1)-(EF3) of Proposition~\ref{prop:localization_pronk} are satisfied. The first is true by definition of algebraic stacks. The second and third are the content of Propositions~\ref{prop:L_ef2} and \ref{prop:L_ef3} respectively.
\end{proof}

%% file: adams.tex
\section{Adams Hopf algebroids and Adams stacks}\label{section:adams}
 In this section we will prove that any flat Hopf algebroid for which the dualizable comodules form a generator of the category of all comodules is an Adams Hopf algebroid (Theorem~\ref{thm:adams_iff_resolution_property}), and that all strong symmetric monoidal left adjoint functors from its category of comodules to another category of comodules are tame (Theorem~\ref{thm:adams_implies_tame}). Recall from \cite[Definition~1.4.3]{HOVEY} that a Hopf algebroid $(A,\Gamma)$ is called an \emph{Adams Hopf algebroid} if $\Gamma$, considered as $(A,\Gamma)$-comodule, is a filtered colimit of dualizable comodules $\Gamma_i$.

 \subsection{The Tannakian perspective}
 To understand why every flat Hopf algebroid for which the dualizable comodules form a generator is an Adams Hopf algebroid we recall the following facts about Tannakian reconstruction from \cite{SCHAEPPI}. These are not required for the proof, but they give an explanation for why the proof strategy works.

 Fix a flat Hopf algebroid $(A,\Gamma)$ such that the dualizable comodules form a generator of $\Comod(A,\Gamma)$. Let $\ca{A}^{d} \subseteq \Comod(A,\Gamma)$ be the full subcategory of dualizable comodules, and write
\[
 w \colon \ca{A}^{d} \rightarrow \Mod_A
\]
 for the forgetful functor. Note that $\ca{A}^{d}$ is generally not an abelian category. Adapting \cite[Corollary~7.5.2]{SCHAEPPI} to Hopf algebroids we get an isomorphism
\[
 \Gamma \cong \int^{A\in\ca{A}^{d}} w(A) \otimes w(A)^{\vee}
\]
 of Hopf algebroids over $A$. It is also not hard to see that $w$ is \emph{flat}, that is, it is a filtered colimit of representable functors $\ca{A}^{d}(A_i,-)$. The Yoneda lemma
\[
 \int^{A\in\ca{A}^{d}} \ca{A}^{d}(A_i,A) \otimes w(A)^{\vee} \cong w(A_i)^{\vee}
\]
 and the fact that colimits commute with each other show that $\Gamma$ is a filtered colimit of dualizable $A$-modules. Moreover, these modules are the underlying modules of comodules. This makes it at least plausible that $\Gamma$ is a filtered colimit of comodules. Moreover, going back to the proof that $w$ is a filtered colimit of representables we get a candidate for the indexing diagram.  From this point of view, the content of the proof of Theorem~\ref{thm:adams_iff_resolution_property} is that this candidate diagram works.

\begin{dfn}\label{dfn:strong_resolution_property}
 An algebraic stack $X$ has the \emph{strong resolution property} if the dualizable objects form a generator of $\QCoh(X)$. 
\end{dfn}

\begin{rmk}\label{rmk:strong_resolution}
 From \cite[Proposition~1.4.1]{HOVEY} it follows that an algebraic stack with the strong resolution property also has the resolution property, and that the two notions coincide for coherent algebraic stacks.
\end{rmk}

 We are now ready to prove that a flat Hopf algebroid is an Adams Hopf if and only if the category of dualizable comodules is a generator. Combining this with the above remark we find that an algebraic stack is an Adams stack if and only if it has the strong resolution property.

\begin{proof}[Proof of Theorem~\ref{thm:adams_iff_resolution_property}]
 From \cite[Proposition~1.4.4]{HOVEY} we know that the dualizable comodules of an Adams Hopf algebroid form a generator of the category of comodules, so it remains to show the converse.

 Let $\ca{A}^{d}$ be the category of dualizable comodules of $(A,\Gamma)$, and write
\[
 w \colon \ca{A}^{d} \rightarrow \Mod_A
\]
 for the forgetful functor. Since colimits in the category of comodules are computed as in the category of $A$-modules, we can use \cite[Proposition~1.4.1]{HOVEY} to show that the diagram
\[
 d \colon  \ca{A}^{d} \slash \Gamma \rightarrow \Comod(A,\Gamma)
\]
 which sends $\varphi \colon M \rightarrow \Gamma$ to its domain $M$ has colimit $\Gamma$. In fact, $\ca{A}^{d}$ is a \emph{dense} generator of $\Comod(A,\Gamma)$ (cf.\ \cite[Corollary~7.5.2]{SCHAEPPI}). We claim that the category $\ca{A}^{d} \slash \Gamma$ is filtered.

 We have a natural bijection
\[
\xymatrix{ \Comod(A,\Gamma)\bigl( M,\Gamma \bigr) \ar[r]^-{\cong} & \Mod_A(M, A)}
\]
 given by composition with $\varepsilon \colon \Gamma \rightarrow A$. We can use this to define a contravariant functor
\[
\xymatrix{ \ca{A}^{d} \slash \Gamma \ar[r] & \el(w) }
\]
 which sends $\varphi\colon M \rightarrow \Gamma$ to the object $(M^{\vee}, \varepsilon \varphi)$, and a morphism $f \colon \varphi \rightarrow \varphi^{\prime}$ to $f^{\vee}$. Since every $M \in \ca{A}^{d}$ is its own double dual, the above bijection shows that this functor is essentially surjective. An explicit essential preimage of $(X,x)$ is given by the morphism $X^{\vee} \rightarrow \Gamma$ corresponding to the composite
\[
 \xymatrix{X^{\vee} \ar[r]^-{1 \otimes x} & X^{\vee}\otimes X \ar[r]^-{\varepsilon_X} & A}
\]
 under the above bijection.

 It remains to show that the functor is fully faithful. Fix two objects $\varphi \colon M \rightarrow \Gamma$ and $\psi \colon N \rightarrow \Gamma$ of $\ca{A}^{d} \slash \Gamma$. The assignment which sends an object of $\ca{A}^{d}$ to its dual is an equivalence of categories, so for any morphism $g \colon (N^{\vee},\varepsilon \psi) \rightarrow (M^{\vee}, \varepsilon \varphi)$ there exists a unique morphism $f \colon M \rightarrow N$ of comodules such that $f^{\vee}=g$. The morphism $f^{\vee}=\Hom_A(f,A)$ takes $\varepsilon \psi$ to $\varepsilon \varphi$ if and only if the triangle
\[
 \xymatrix{ M \ar[rr]^-{f} \ar[rd]_{\varepsilon \varphi} & & N \ar[ld]^{\varepsilon \psi}\\ & A}
\]
 commutes. From the natural bijection
\[
\xymatrix{ \Comod(A,\Gamma)\bigl( M,\Gamma \bigr) \ar[r]^-{\cong} & \Mod_A(M, A)}
\]
 it follows that this is the case if and only if $f$ is a morphism $(M,\varphi) \rightarrow (N,\psi)$ in the category $\ca{A}^{d} \slash \Gamma$, which shows that the functor is indeed fully faithful. Therefore our diagram is equivalent to the opposite of the category of elements of $w$. It remains to show that $\el(w)$ is cofiltered.

 It is clearly nonempty, and the existence of direct sums in $\ca{A}^{d}$ shows that the second condition for a category to be cofiltered is also satisfied. To see the last condition, let
\[
\xymatrix{(X,x) \ar@<2pt>[r]^-{\varphi} \ar@<-2pt>[r]_-{\psi} & (Y,y)} 
\]
 be two morphisms in $\el(w)$. Let $\gamma \colon K\rightarrow X$ be the equalizer of $\varphi$ and $\psi$ in $\Comod(A,\Gamma)$, which is computed as in $\Mod_A$ since $(A,\Gamma)$ is flat. From the definition of morphisms in $\el(w)$ it follows that there exists $k \in K$ such that $\gamma(k)=x$. Since the category of dualizable comodules forms a generator, we can find $Z \in \ca{A}$ and a morphism $Z \rightarrow K$ whose image contains $k$ (see \cite[Proposition~1.4.1]{HOVEY}). A choice of an element $z \in Z$ with image $k \in K$ turns the composite
\[
 \xymatrix{ Z \ar[r] & K \ar[r]^-{\gamma} & X}
\]
 into a morphism $(Z,z) \rightarrow (X,x)$ in $\el(w)$ with the desired property.
\end{proof}

\begin{cor}\label{cor:coherent_resolution_property_implies_adams}
 Any coherent algebraic stack $X$ with the resolution property is an Adams stack.
\end{cor}

\begin{proof}
 This follows from Theorem~\ref{thm:adams_iff_resolution_property} and Remark~\ref{rmk:strong_resolution}.
\end{proof}

\subsection{Tame functors and Adams stacks}
 In this section we will prove Theorem~\ref{thm:adams_implies_tame}, which states that all strong symmetric monoidal left adjoints
\[
F \colon \Comod(A,\Gamma) \rightarrow \Comod(B,\Sigma) 
\]
 are tame if $(A,\Gamma)$ is an Adams Hopf algebroid and $(B,\Sigma)$ is flat. In order to do this, we have to use the characterization of tame functors from Corollary~\ref{cor:tame_characterization}.

\begin{proof}[Proof of Theorem~\ref{thm:adams_implies_tame}]

 To show that any such $F$ is tame, it suffices to show that any symmetric monoidal left adjoint
\[
 F\colon \Comod(A,\Gamma) \rightarrow \Mod_B
\]
 sends the algebra $\Gamma \in \Comod(A,\Gamma)$ to a faithfully flat $B$-algebra (see Corollary~\ref{cor:tame_characterization}). Equivalently, we want to show that the sequence
\[
 \xymatrix{0 \ar[r] & F(A) \ar[r] & F(\Gamma) \ar[r] & F(\Gamma \slash A) \ar[r] & 0}
\]
 is an exact sequence of flat $B$-modules (see \cite[Lemma~5.5]{LURIE}). By definition of Adams Hopf algebroids there exists a filtered diagram $\Gamma_i$ of dualizable comodules whose colimit is $\Gamma$. Write $\ca{D}$ for the indexing category of this diagram. Since the comodule $A$ is finitely presentable, there exists an index $j \in \ca{D}$ such that the morphism $A \rightarrow \Gamma$ factors through $\Gamma_j$. Using the fact that $\ca{D}$ is filtered we find that the functor
\[
 j \slash \ca{D} \rightarrow \ca{D}
\]
 is final, hence that $\Gamma$ is also a colimit of the diagram
\[
 D \colon j\slash \ca{D} \rightarrow \Comod(A,\Gamma)
\]
 which sends the object $j \rightarrow i$ to $\Gamma_i$. Moreover, $j \slash \ca{D}$ is clearly filtered. The morphisms
\[
A \rightarrow \Gamma_j \rightarrow \Gamma_i
\]
 define a natural transformation from the constant diagram $C_A$ at $A$ to $D$, which we will denote by $\alpha \colon C_A \Rightarrow D$. By construction, the colimit of $\alpha$ is the unit $A \rightarrow \Gamma$ of the algebra $\Gamma$ in $\Comod(A,\Gamma)$. The morphism $A \rightarrow \Gamma$ is split as a morphism of $A$-modules, with splitting given by $\varepsilon \colon \Gamma \rightarrow A$. It follows that the components $\alpha_{j\rightarrow i}$ of $\alpha$ are also split as morphisms of $A$-modules. 

 By replacing $\ca{D}$ with $j \slash \ca{D}$ if necessary, we can therefore assume that the diagram $\Gamma_{(-)}$ on $\ca{D}$ admits a natural transformation $\alpha_i \colon A \rightarrow \Gamma_i$ with these properties, meaning that the components $\alpha_i$ are split as morphisms of $A$-modules, and the colimit of the $\alpha_i$ is the unit $A \rightarrow \Gamma$ of the algebra $\Gamma$ in $\Comod(A,\Gamma)$.

 The fact that $\alpha_i$ is split as a morphism of $A$-modules implies that the exact sequence
\[
\xymatrix{0\ar[r] & A \ar[r]^-{\alpha_i} & \Gamma_i \ar[r] & \Gamma_i \slash A \ar[r] & 0}
\]
 is split as a sequence of $A$-modules. Thus the underlying $A$-module of $\Gamma_i \slash A$ is finitely generated and projective, so $\Gamma_i \slash A$ is dualizable as a comodule. Since (filtered) colimits commute with colimits it follows that $\Gamma\slash A$ is the filtered colimit of the $\Gamma_i \slash A$. Thus $F(\Gamma)$ and $F(\Gamma \slash A)$ are both filtered colimits of dualizable $B$-modules, so they are flat. It remains to show that $F$ preserves the exact sequence
\[
 \xymatrix{0 \ar[r] & A \ar[r] & \Gamma \ar[r] & \Gamma \slash A \ar[r] & 0} \smash{\rlap{.}}
\]

 Since filtered colimits commute with finite limits, it suffices to check that the sequence
\[
\xymatrix{0\ar[r] & F(A) \ar[r]^-{F(\alpha_i)} & F(\Gamma_i) \ar[r] & F(\Gamma_i \slash A) \ar[r] & 0} 
\]
 is exact for all $i \in \ca{D}$. We have reduced the problem to showing that $F$ preserves exact sequences
\[
 \xymatrix{0\ar[r] & L \ar[r] & M \ar[r] & N \ar[r] & 0}
\]
 of dualizable comodules. Since the underlying $A$-module of a dualizable comodule is projective, such sequences are split as sequences of $A$-modules. The existence of this splitting implies that the dual sequence
\[
 \xymatrix{0\ar[r] & N^{\vee} \ar[r] & M^{\vee} \ar[r] & L^{\vee} \ar[r] & 0} 
\]
 is also an exact sequence in $\Comod(A,\Gamma)$. Since $F$ is left adjoint, both
\[
  \xymatrix{FL \ar[r] & FM \ar[r] & FN \ar[r] & 0}
\]
and
\[
  \xymatrix{F(N^{\vee}) \ar[r] & F(M^{\vee}) \ar[r] & F(L^{\vee}) \ar[r] & 0}
\]
 are exact in $\Mod_B$. From the fact that $\Hom_B(-,B)$ turns colimits into limits it follows that the sequence
\[
  \xymatrix{0 \ar[r] & F(L^{\vee})^{\vee} \ar[r] & F(M^{\vee})^{\vee} \ar[r] & F(N^{\vee})^{\vee}} 
\]
 is exact. As a strong symmetric monoidal functor, $F$ preserves duals, hence there is a natural isomorphism $F(-) \cong F\bigl((-)^{\vee}\bigr)^{\vee}$. This shows that $F$ preserves exact sequences of dualizable comodules.
\end{proof}

 This result allows us to prove Theorem~\ref{thm:adams_embedding} and Theorem~\ref{thm:coherent_adams_weakly_tannakian_equivalence}.

 \begin{dfn}\label{dfn:right_exact_symmetric_monoidal}
 A symmetric monoidal additive category is called \emph{right exact symmetric monoidal} if it has finite colimits, and if for every $A\in\ca{A}$, the functor
\[
 A\otimes - \colon \ca{A} \rightarrow \ca{A}
\]
 preserves finite colimits. We write $\ca{RM}$ for the 2-category with objects the right exact symmetric monoidal additive categories, 1-cells the right exact strong symmetric monoidal functors, and 2-cells the symmetric monoidal natural transformations between them.
\end{dfn}

 Theorem~\ref{thm:adams_embedding} says that the restriction of the pseudofunctor
\[
 \QCoh_{\fp}(-) \colon \ca{AS}^{\op} \rightarrow \ca{RM}
\]
 to Adams stacks is an equivalence on hom-categories.


\begin{proof}[Proof of Theorem~\ref{thm:adams_embedding}]
 From Theorem~\ref{thm:stacks_embedding} we know that the pseudofunctor
\[
 \QCoh(-) \colon \ca{AS}^{\op} \rightarrow \ca{T}
\]
 is an equivalence on hom-categories, and from Theorem~\ref{thm:adams_implies_tame} we know that for two Adams stacks $X$ and $Y$, \emph{any} strong symmetric monoidal left adjoint
\[
F \colon \QCoh(Y) \rightarrow \QCoh(X)
\]
 is tame. Moreover, $\QCoh(Y)$ is finitely presentable and closed (\cite[Theorem~1.3.1]{HOVEY}), so restricting along the inclusion $\QCoh_{\fp}(Y) \subseteq \QCoh(Y)$ gives an equivalence between strong symmetric monoidal left adjoints $F$ as above and right exact strong symmetric monoidal functors
\[
 \QCoh_{\fp}(Y) \rightarrow \QCoh(X) \smash{\rlap{.}}
\]
 From Corollary~\ref{cor:tame_preserves_presentability} we know that this restriction factors through $\QCoh_{\fp}(X)$.
\end{proof}

 Combining this with Theorem~\ref{thm:recognition} we find that the pseudofunctor $\Coh(-)$ gives a biequivalence between the 2-category of coherent algebraic stacks with the resolution property and the 2-cateogory of weakly Tannakian categories.

\begin{proof}[Proof of Theorem~\ref{thm:coherent_adams_weakly_tannakian_equivalence}]
 From Theorem~\ref{thm:adams_embedding} and Corollary~\ref{cor:coherent_resolution_property_implies_adams} we know that the pseudofunctor $\Coh(-)$ is an equivalence on hom-categories.
 Theorem~\ref{thm:recognition} shows that $\Coh(-)$ is essentially surjective on objects.
\end{proof}

%% file: filteredmodules.tex
 \section{A conjecture by Richard Pink}\label{section:filteredmodules}

\subsection{Outline}
 In this section we will prove Theorem~\ref{thm:conjecture_for_adams_stacks}. In \S \ref{section:recollections} we recall the definitions of the categories of filtered modules. The first half of Theorem~\ref{thm:conjecture_for_adams_stacks} is the content of the following proposition.

\begin{prop}\label{prop:first_half_of_conjecture}
  There is an algebraic stack $\mathfrak{P}$ on the $\fpqc$-site $\Aff$ and a symmetric monoidal equivalence $\MF_{W,\fp} \simeq \Coh(\mathfrak{P})$. Let 
\[
\vcenter{ \xymatrix{ \mathfrak{P}_{{\mathbb{Q}}_p} \ar[d] \ar[r] & \mathfrak{P} \ar[d] \\ 
\Spec(\mathbb{Q}_p) \ar[r] & \Spec(\mathbb{Z}_p)} }
\quad\mbox{and}\quad
\vcenter{ \xymatrix{ \mathfrak{P}_{n} \ar[d] \ar[r] & \mathfrak{P} \ar[d] \\ 
\Spec(\mathbb{Z}\slash p^n \mathbb{Z}) \ar[r] & \Spec(\mathbb{Z}_p)}}
\]
 be bipullback squares in the category of \emph{all} stacks on the $\fpqc$-site $\Aff$. Then $\Coh(\mathfrak{P}_{{\mathbb{Q}}_p})$ is equivalent (as a symmetric monoidal $\mathbb{Q}_p$-linear category) to the Tannakian category $\MF^{\Phi,f}_K$ of weakly admissible $\Phi$-modules, and $\Coh(\mathfrak{P}_n)$ is equivalent to the full subcategory of $\MF_{W,\fp}$ of objects whose underlying $W$-modules are annihilated by $p^n$.
\end{prop}

 We prove this proposition in \S \ref{section:first_half}. The existence of $\mathfrak{P}$ is an immediate consequence of the recognition theorem (Theorem~\ref{thm:recognition}). The difficulty lies in identifying the categories of coherent sheaves of $\mathfrak{P}_n$ and $\mathfrak{P}_{\mathbb{Q}_p}$.

 To prove Theorem~\ref{thm:conjecture_for_adams_stacks} from the above proposition, it remains to show that $\mathfrak{P}$ is the bicolimit of the $\mathfrak{P}_n$. We use the embedding theorem (Theorem~\ref{thm:adams_embedding}) to reduce this to a question about bilimits of symmetric monoidal categories, which we compute in \S \S \ref{section:bilimit_in_ab_categories} and \ref{section:bicolimit_adams}.

\subsection{Recollections about filtered modules}\label{section:recollections}
 The following summary is taken from \cite[\S2]{SCHAEPPI}. Fix a perfect field $k$ of characteristic $p>0$, and let $W$ be the ring of Witt vectors with coefficients in $k$. For our purposes it suffices to know that $W$ is a complete discrete valuation ring with residue field $k$ which contains the ring of $p$-adic integers $\mathbb{Z}_p$, and that $p \in \mathbb{Z}_p$ is a uniformizer of $W$. A construction of the ring can be found in \cite[\S II.6]{SERRE}. There is an automorphism $\sigma \colon W \rightarrow W$ of $\mathbb{Z}_p$-algebras which lifts the Frobenius automorphism on the residue field $k$ of $W$ (see \cite[Th\'eor\`eme~II.7 and Proposition~II.10]{SERRE}). This automorphism $\sigma$ is again called the Frobenius automorphism. The induced isomorphism of the field of fractions $K$ of $W$ is again denoted by $\sigma$. For a $W$-module (or a $K$-vector space) $M$, we write $M_\sigma$ for the $W$-module obtained by base change along $\sigma$. In the following definition we use the same notation and terminology that was introduced in \cite{WINTENBERGER}.

\begin{dfn}\label{dfn:MF_W}
 A \emph{filtered $F$-Module} consists of
\begin{itemize}
\item
 a $W$-module $M$ with a decreasing filtration $(\Fil^i M)_{i \in \mathbb{Z}}$ of submodules $\Fil^i M \subseteq M$. The filtration is \emph{exhaustive}, $\bigcup_{i\in \mathbb{Z}} \Fil^i M = M$, and \emph{separated}, $\bigcap_{i\in \mathbb{Z}} \Fil^i M =0$;
\item
 for each $i\in \mathbb{Z}$, a morphism $\varphi^i \colon \Fil^i M \rightarrow M_\sigma$ of $W$-modules such that the restriction of $\varphi^i$ to $\Fil^{i+1} M$ is $p\varphi^{i+1}$.
\end{itemize}
 A morphism of filtered $F$-modules $M \rightarrow M^\prime$ is a morphism $g\colon M\rightarrow M^\prime$ of $W$-modules such that for all $i\in \mathbb{Z}$, $g(\Fil^i M ) \subseteq \Fil^i M^\prime$ and $\varphi^i_{M^{\prime}} \circ g = g \circ \varphi^i_M$. We denote the category of filtered $F$-modules by $\MF$, and we write $\MF_{W,\fp}$ for the full subcategory of objects $M$ which satisfy
\begin{itemize}
 \item
 the $W$-module $M$ is finitely generated;
 \item
 the modules $\Fil^i M$ are direct summands of $M$;
 \item
 the images of the $\varphi^i$ span $M$, that is, $\sum_{i\in \mathbb{Z}} \varphi^i(\Fil^i M) =M$.
\end{itemize}
\end{dfn}

\begin{dfn}\label{dfn:filtered_phi_modules}
 A filtered $\Phi$-module is a vector space $\Delta$ over $K$, endowed with a $\sigma$-semilinear isomorphism $\Phi \colon \Delta \rightarrow \Delta$ and an exhaustive and separated filtration $(\Fil^i \Delta)_{i \in Z}$. A morphism of filtered $\Phi$-modules is a morphism of $K$-vector spaces which is compatible with $\Phi$ and with the filtration. 
 The category of filtered $\Phi$-modules is denoted by $\MF^{\Phi}_K$.

 Fix a finite dimensional filtered $\Phi$-module $\Delta$. A $W$-lattice $M$ in $\Delta$ is called \emph{strongly divisible} if
\[
 \sum_{i\in \mathbb{Z}} p^{-i} \Phi(\Fil^i \Delta \cap M) = M \smash{\rlap{.}}
\]
 A \emph{weakly admissible filtered $\Phi$-module} is a filtered $\Phi$-module which has a strongly divisible lattice. The full subcategory of $\MF^{\Phi}_K$ consisting of weakly admissible filtered $\Phi$-modules is denoted by $\MF^{\Phi,f}_K$.
\end{dfn}

\begin{prop}\label{prop:MF_W_properties}
 The category of filtered $F$-modules has the following properties.
\begin{enumerate}
 \item[i)]
 The category $\MF_{W,\fp}$ is abelian, and the forgetful functor
\[
 w \colon \MF_{W,\fp} \rightarrow \Mod_W 
\]
 is an exact $\mathbb{Z}_p$-linear functor.
 \item[ii)]
 The category $\MF_{W,\fp}$ is endowed with a $\mathbb{Z}_p$-linear tensor product which turns $\MF_{W,\fp}$ into a closed symmetric monoidal $\mathbb{Z}_p$-linear category, and the forgetful functor $w \colon \MF_{W,\fp} \rightarrow \Mod_W$ is strong symmetric monoidal.
 \item[iii)]
 For any object $M$ of $\MF_{W,\fp}$ there exists a dualizable object $M^\prime$ of $\MF_{W,\fp}$ and an epimorphism $g \colon M^{\prime} \rightarrow M$.
\end{enumerate}
\end{prop}

\begin{proof}
 Part~i) is proved in \cite[Proposition~1.4.1]{WINTENBERGER}. The symmetric monoidal structure is constructed in \cite[\S1.7]{WINTENBERGER}. Part~iii) follows from \cite[Proposition~1.6.3]{WINTENBERGER} and the characterization of dualizable objects in \cite[\S1.7]{WINTENBERGER}.
\end{proof}

 In our terminology, the category $\MF_{W,\fp}$ is thus a weakly Tannakian  $\mathbb{Z}_p$-linear category with a fiber functor $w \colon \MF_{W,\fp} \rightarrow \Mod_W$. From Theorem~\ref{thm:recognition} it follows that there exists an algebraic stack $\mathfrak{P}$ over $\mathbb{Z}_p$ and a symmetric monoidal $\mathbb{Z}_p$-linear equivalence $\MF_{W,\fp} \simeq \Coh(\mathfrak{P})$. In order to compute the desired bipullbacks it is convenient to use the corresponding result about flat Hopf algebroids.

 \begin{prop}\label{prop:MF_W_weakly_tannakian}
  There is a flat Hopf algebroid $(W,\Pi)$ in $\Mod_{\mathbb{Z}_p}$ and a symmetric monoidal equivalence $\Comod_{\fp}(W,\Pi) \simeq \MF_{W,\fp}$ such that the diagram
\[
 \xymatrix{\MF_{W,\fp} \ar[r] \ar[rd]_{w} & \Comod_{\fp}(W,\Pi) \ar[d] \\ & \Mod_W}
\]
 is commutative.
 \end{prop}
 
 \begin{proof}
 This follows from Proposition~\ref{prop:MF_W_properties} and Theorem~\ref{thm:hopf_recognition}, applied to the case $R=\mathbb{Z}_p$.
\end{proof}

 \subsection{Computing the fibers}\label{section:first_half}
 The computation of the bipullbacks in Theorem~\ref{thm:conjecture_for_adams_stacks} follows from general facts about flat Hopf algebroids $(A,\Gamma)$ where $A$ is Noetherian. The Noetherian condition is needed in order to identify the categories of finitely presentable comodules of the fibers.

 \begin{prop}\label{prop:quotient_and_localization}
  Let $R$ be a commutative ring, $I \subseteq R$ an ideal, $S \subseteq R$ a multiplicative set. Let $(A,\Gamma)$ be a flat Hopf algebroid in $\Mod_R$. Write $X$ for the algebraic stack associated to $(A,\Gamma)$, and let
\[
\vcenter{ \xymatrix{ S^{-1} X \ar[d] \ar[r] & X \ar[d] \\ 
\Spec(S^{-1} R) \ar[r] & \Spec(R)} }
\quad\mbox{and}\quad
\vcenter{ \xymatrix{ X \slash IX \ar[d] \ar[r] & X \ar[d] \\ 
\Spec(R \slash I) \ar[r] & \Spec(R)}} 
\]
 be bipullback diagrams in the 2-category of stacks on the $\fpqc$-site $\Aff$. Then $X \slash IX$ is equivalent to the stack associated to the Hopf algebroid $(A\slash IA, \Gamma \slash I\Gamma)$ and $S^{-1} X$ is equivalent to the stack associated to $(S^{-1} A, S^{-1}\Gamma)$.
 \end{prop}

 \begin{proof}
 Since the functor which sends a presheaf of groupoids to its associated stack preserves finite bilimits, it suffices to compute the desired bipullbacks in the category of presheaves of groupoids. In both squares the two groupoids in the bottom leg are discrete, so the bipullbacks in question coincide with the strict pullbacks. All the constituents of the two diagrams are obtained by applying the Yoneda embedding to groupoids internal to affine schemes, and the Yoneda embedding preserves all limits. This reduces the problem to a computation of the desired pullbacks in the category of affine groupoid schemes.

 Limits of internal groupoids are computed at the level of underlying affine schemes (because groupoids are a finite limit theory). Under the contravariant equivalence between affine schemes and commutative rings, pullbacks correspond to pushouts. In the case of interest they are given by
\[
(R \slash I \ten{R} A, R \slash I \ten{R} \Gamma ) 
\]
 and
\[
(S^{-1} R \ten{R} A, S^{-1}R \ten{R} \Gamma) 
\]
 respectively. These are isomorphic to the Hopf algebroids in the statement.
 \end{proof}

 It remains to characterize the categories of coherent sheaves of the algebraic stacks $X \slash IX$ and $S^{-1} X$. To do this we have to assume that the commutative ring $A$ is Noetherian.

\begin{lemma}
 Let $(A,\Gamma)$ be a Hopf algebroid. Then an $(A,\Gamma)$-comodule is finitely presentable if and only if its underlying $A$-module is finitely presentable.
\end{lemma}

\begin{proof}
 See \cite[Proposition~1.3.3]{HOVEY}.
\end{proof}

\begin{prop}\label{prop:finitely_generated_subcomodules}
 Let $(A,\Gamma)$ be a flat Hopf algebroid such that $A$ is Noetherian, and let $M$ be an $(A,\Gamma)$-comodule. Then every finitely generated sub-$A$-module of $M$ is contained in a finitely presentable sub-comodule, and the category $\Comod(A,\Gamma)$ is locally finitely presentable.
\end{prop}

\begin{proof}
 The proof is basically identical to the one for representations of flat group schemes over Noetherian rings, see for example \cite[\S1.4]{SERRE_GEBRES} and \cite[Corollary~7.5.3]{SCHAEPPI}.

 Fix a submodule $N \subseteq M$ generated by $m_i$, $1 \leq i \leq n$. The coaction of the elements $m_i$ can be written as $\rho(m_i)= \sum g_j \otimes m_i^j$ for some finite collection $m_i^j$. Let $N^{\prime}$ be the submodule generated by the $m_i^j$, and let
\[
 \newdir{ >}{{}*!/-7pt/@{>}}
 \xymatrix{N^{\prime \prime} \ar@{{ >}->}[d] \ar@{{ >}->}[r] & \Gamma \ten{A} N^{\prime} \ar@{{ >}->}[d] \\ M \ar@{{ >}->}[r]^-{\rho} & \Gamma \ten{A} M }
\]
 be a pullback diagram in $\Mod_A$. Since $(A,\Gamma)$ is flat, this is also a pullback diagram of $(A,\Gamma)$-comodules. Thus $N^{\prime\prime}$ is a comodule and by construction, $N \subseteq N^{\prime\prime}$.

 Commutativity of the diagram
\[
 \newdir{ >}{{}*!/-7pt/@{>}}
 \xymatrix{N^{\prime\prime} \ar@{{ >}->}[d] \ar@{{ >}->}[r] & \Gamma \ten{A} N^{\prime} \ar[r]^-{\varepsilon \ten{A} N^{\prime}} \ar@{{ >}->}[d] & N^{\prime} \ar@{{ >}->}[d] \\ M \ar@{{ >}->}[r]^-{\rho} & \Gamma \ten{A} M \ar[r]^-{\varepsilon \ten{A} M} & M} 
\]
 and the fact that $\varepsilon \ten{A} M \cdot \rho=\id_M$ show that $N^{\prime \prime} \subseteq N^{\prime}$, hence that it is a finitely generated $A$-module.
\end{proof}

 From the above lemma it follows that every comodule is the filtered union of its finitely presented sub-comodules, hence that the category of comodules is locally finitely presentable.

\begin{rmk}
 In \cite{WEDHORN} slightly more is claimed, but there is a subtle mistake in the argument: it is not true that an arbitrary intersection of sub-comodules is again a sub-comodule (at the level of flat coalgebras there are counterexamples due to Gabber and Serre, see \cite[Remarque~VI.11.10.1]{SGA3}). An intersection is a form of limit, and only finite limits of comodules can be computed as in the category of modules. The author was unable to prove the full claim of \cite[Proposition~5.7]{WEDHORN}.
\end{rmk}

 \begin{prop}\label{prop:comodules_of_quotient}
 Let $R$, $I$, $(A,\Gamma)$, $X$ and $X \slash IX$ be as in Proposition~\ref{prop:quotient_and_localization}. Assume further that $A$ is Noetherian. Then $\Coh(X \slash IX)$ is equivalent (as a symmetric monoidal $R$-linear category) to the full subcategory of $\Coh(X)$ consisting of those coherent sheaves which are annihilated by $I$.
 \end{prop}

 \begin{proof}
 From Proposition~\ref{prop:quotient_and_localization} we know that $X$ is the stack associated to $(A,\Gamma)$ and $X \slash IX$ is associated to $(A\slash IA, \Gamma \slash I\Gamma)$. Thus it suffices to check that the category $\Comod_{\fp}(A\slash IA, \Gamma \slash I\Gamma)$ is equivalent to the full subcategory of $\Comod_{\fp}(A,\Gamma)$ consisting of finitely presented comodules whose underlying module is annihilated by $I$. For an $A$-module $M$ which is annihilated by $I$ we have a natural isomorphism
\[
 \Gamma \slash I \Gamma \ten{A \slash IA} M \cong \Gamma \ten{A} M
\]
 from which we deduce that there is a fully faithful functor
\[
 \Comod(A\slash IA, \Gamma \slash I\Gamma) \rightarrow \Comod(A,\Gamma) 
\]
 whose image is the full subcategory of $(A,\Gamma)$-comodules whose underlying module is annihilated by $I$. The assumption that $A$ is Noetherian implies that  an $A$-module which is annihilated by $I$ is finitely presented as an $A$-module if and only if it is finitely presented as an $A\slash IA$-module. Thus the restriction of the above functor to finitely presented comodules gives the desired equivalence.
 \end{proof}
 
  Recall that the base-change $T \ten{R} \ca{C}$ of an $R$-linear category $\ca{C}$ along a ring homomorphism $R \rightarrow T$ has the same objects as $\ca{C}$, and the $T$-module of morphisms between two objects $C$, $C^{\prime}$ in $T\ten{R} \ca{C}$ is given by $T \ten{R} \ca{C}(C,C^{\prime})$. 

 \begin{prop}\label{prop:comodules_of_localization}
 Fix a commutative ring $R$ and a multiplicative set $S \subseteq R$. Let $(A,\Gamma) \in \Mod_R$ be a Hopf algebroid such that $A$ is Noetherian. Then there is an equivalence
\[
 F \colon S^{-1} R\ten{R} \Comod_{\fp}(A,\Gamma) \rightarrow \Comod_{\fp}(S^{-1} A, S^{-1} \Gamma)
\]
 of symmetric monoidal $S^{-1}R$-linear categories.

 Moreover, the restriction of $F$ to the full subcategory of finitely presentable comodules whose underlying modules have no $S$-torsion is also an equivalence. In other words, in the category $S^{-1} R\ten{R} \Comod_{\fp}(A,\Gamma)$, every object is isomorphic to an $S$-torsion free comodule.
 \end{prop}

 \begin{proof}
  First note that the strong symmetric monoidal functor
\[
 S^{-1}A \ten{A} - \colon \Mod_A \rightarrow \Mod_{S^{-1}A}
\]
 induces a strong symmetric monoidal $R$-linear functor
\[
 F_0 \colon \Comod_{\fp}(A,\Gamma) \rightarrow \Comod(S^{-1} A, S^{-1} \Gamma) \smash{\rlap{.}}
\]
 From the universal property of base-change for categories it follows that there exists a unique $S^{-1}R$-linear functor $F$ such that the diagram
\[
 \xymatrix{\Comod_{\fp}(A,\Gamma) \ar[rd]_{F_0} \ar[r] & S^{-1}R \ten{R} \Comod_{\fp}(A,\Gamma) \ar[d]^{F} \\ 
& \Comod_{\fp}(S^{-1}A, S^{-1}\Gamma)}
\]
 is commutative. Using the fact that $S^{-1}A \ten{A} -$ is strong symmetric monoidal, it is not hard to see that $F$ is strong symmetric monoidal as well. We claim that $F$ is fully faithful and essentially surjective on objects.

 Recall that for all finitely presented $A$-modules $M$ and \emph{all} $A$-modules $N$, the function
\[
 S^{-1} \Hom_{A}(M,N) \rightarrow \Hom_{S^{-1} A} (S^{-1} M, S^{-1} N)
\]
 which sends $\varphi \slash s$ to $1\slash s \cdot S^{-1} \varphi$ is bijective. Indeed, since $S^{-1} R$ is flat over $R$ and the function is natural in $M$ we can reduce to the case where $M=A$, where both sides are isomorphic to $S^{-1} N$.

 The hom-$R$-module between two objects $(M,\rho_M)$, $(N,\rho_N)$ of $\Comod(A,\Gamma)$ is given by the equalizer of the diagram
\[
 \xymatrix{\Hom_{A}(M,N) \ar[rd]_{\Gamma \ten{A} -} \ar[rr]^{\Hom_A(M,\rho_N)} && \Hom_{A}(M,\Gamma \ten{A} N) \\
 & \Hom_A(\Gamma \ten{A} M, \Gamma\ten{A} N) \ar[ru]_{\quad\; \Hom_A(\rho_M,N)} }
\]
 and the analogous statement holds for comodules over $(S^{-1}A, S^{-1}\Gamma)$. Commutativity of the diagram
\[
 \xymatrix{S^{-1}R \ten{R} \Comod_{fp}(A,\Gamma)\bigl(M,N\bigr) \ar[r]^-{F_{M,N}} \ar[d] & \Comod(S^{-1}A,S^{-1}\Gamma)\bigl(M,N\bigr) \ar[d] \\
S^{-1} \Hom_A(M,N)  \ar[r]^-{\cong} \ar@<-2pt>[d] \ar@<2pt>[d] & \Hom_{S^{-1}A}(S^{-1}M, S^{-1}N) \ar@<-2pt>[d] \ar@<2pt>[d] \\
 S^{-1} \Hom_A(M,\Gamma\ten{A} N) \ar[r]^-{\cong} & \Hom_{S^{-1}A}(S^{-1}M, S^{-1}\Gamma \ten{S^{-1}A} S^{-1}N)}
\]
 and flatness of $S^{-1}R$ over $R$ imply that $F_{M,N}$ is an isomorphism. Thus $F$ is fully faithful, as claimed.

 It remains to check that $F$, or equivalently $F_0$, is essentially surjective. To do this we will need to use the fact that $A$ is Noetherian. Fix a finitely presentable $(S^{-1}A, S^{-1}\Gamma)$-comodule $M$. If $N$ is an $A$-module on which the elements of $S$ act by isomorphisms, then $\Gamma \ten{A} N$ has the same property. It follows that for such $N$, there is a natural isomorphism
\[
 S^{-1}\Gamma \ten{S^{-1}A} N \cong \Gamma\ten{A} N
\]
 of $S^{-1}A$-modules. This allows us to consider $M$ as an $(A,\Gamma)$-comodule $M_0$. The functor $F_0$ defined above has an evident extension to \emph{all} comodules, which we shall not distinguish notationally. From the construction of $M_0$ it is clear that $F_0(M_0)\cong M$. 

 In general $M_0$ is of course not finitely presentable as an $A$-module. However, from Proposition~\ref{prop:finitely_generated_subcomodules} we know that $M_0$ is the filtered union of its finitely presentable sub-comodules $(M_i)_{i\in I}$. The functor $F_0$ preserves filtered colimits because $S^{-1}A \ten{A} -$ does, and it preserves monomorphisms because $S^{-1}A$ is flat over $A$. Thus $M\cong F_0(M_0)$ is the filtered union of the $F_0(M_i)$. But $M$ is finitely presentable by assumption, so the identity of $M$ factors through one of the monomorphisms $F_0(M_i) \rightarrow M$. We conclude that $M \cong F_0(M_i)$, which shows that $F_0$, and therefore $F$, is essentially surjective.

 The second statement follows from the fact that the underlying $A$-module of $M_i$ has no $S$-torsion (because it is a submodule of an $S$-local $A$-module).
 \end{proof}

 These results about bipullbacks along morphisms $\Spec(R \slash I) \rightarrow \Spec(R)$ and $\Spec(S^{-1} R) \rightarrow \Spec(R)$ allow us to prove the first half of Theorem~\ref{thm:conjecture_for_adams_stacks}.

\begin{proof}[Proof of Proposition~\ref{prop:first_half_of_conjecture}]
 Let $\mathfrak{P}$ be the stack associated to the flat Hopf algebroid $(W,\Pi)$ from Proposition~\ref{prop:MF_W_weakly_tannakian}. In Proposition~\ref{prop:quotient_and_localization} we have shown that $\mathfrak{P}_n$ is associated to the Hopf algebroid $(W_n,\Pi_n)$, and that $\mathfrak{P}_{\mathbb{Q}_p}$ is associated to the Hopf algebroid $(S^{-1}W, S^{-1}\Pi )$ where $S=\{1,p,p^2,\ldots \}$.

 Furthermore, the categories of coherent sheaves are equivalent to the categories of comodules of the respective Hopf algebroids. From Proposition~\ref{prop:comodules_of_quotient} we know that $\Comod_{\fp}(W_n,\Pi_n)$ is equivalent to the full subcategory of $(W,\Pi)$-modules whose underlying module is annihilated by $p^n$. The first claim follows from the fact that $\Comod_{\fp}(W,\Pi)$ is equivalent to $\MF_{W,\fp}$ (see Proposition~\ref{prop:MF_W_weakly_tannakian}).

 Let $\ca{C}$ be the $\mathbb{Z}_p$-linear category whose objects are pairs $(\Delta, M)$ of a weakly admissible filtered $\Phi$-module $\Delta$ and a strongly divisible lattice $M$ of $\Delta$, and with morphisms $(\Delta, M) \rightarrow (\Delta^{\prime}, M^{\prime})$ the morphisms of filtered $\Phi$-modules $\Delta \rightarrow \Delta^{\prime}$ whose underlying morphism of $K$-vector spaces sends $M$ to $M^{\prime}$ (see \cite[\S7.11]{FONTAINE_LAFFAILLE}). If $(\Delta, M)$, $(\Delta^{\prime}, M^{\prime})$ are two objects of $\ca{C}$, then $M\otimes M^{\prime}$ is a strongly divisible lattice in $\Delta \otimes \Delta$ (see \cite[Corollaire~7.9]{FONTAINE_LAFFAILLE}, \cite[Proposition~4.2]{LAFFAILLE}). Thus $\ca{C}$ can be endowed with an evident symmetric monoidal structure.

 Given an arbitrary morphism $f \colon \Delta \rightarrow \Delta^{\prime}$ of filtered $\Phi$-modules, there exists $i \in \mathbb{N}$ such that $p^i f(M) \subseteq M^{\prime}$. This implies that the faithful strong symmetric monoidal $\mathbb{Z}_p$-linear functor
\[
 \ca{C} \rightarrow \MF^{\Phi,f}_{K}
\]
 which sends $(\Delta, M)$ to $\Delta$ induces a symmetric monoidal $\mathbb{Q}_p$-linear equivalence between $\mathbb{Q}_p \ten{\mathbb{Z}_p} \ca{C}$ and $\MF^{\Phi,f}_{K}$ (cf.\ \cite[\S7.11]{FONTAINE_LAFFAILLE}).

 On the other hand, there is also a functor
\[
 \ca{C} \rightarrow \MF_{W,\fp}
\]
 given by $(\Delta, M) \mapsto M$, with filtration $\Fil^i M = \Fil^i \Delta \cap M$, and
\[
 \varphi^i \colon \Fil^i M \rightarrow M_\sigma
\]
 given by the restriction of $p^{-i} \Phi$. In \cite[\S7.12]{FONTAINE_LAFFAILLE} and \cite[\S1.6.2]{WINTENBERGER} it was shown that this functor induces an equivalence between $\ca{C}$ and the full subcategory of $\MF_{W,\fp}$ of objects whose underlying $W$-module has no $p$-torsion. It is evidently strong symmetric monoidal. Combining these two results with Proposition~\ref{prop:comodules_of_localization} we get the desired equivalence
\[
 \Coh(\mathfrak{P}_{\mathbb{Q}_p} ) \simeq \Comod_{\fp}(S^{-1}W,S^{-1}\Pi) \simeq \mathbb{Q}_p \ten{\mathbb{Z}_p} \ca{C} \simeq \MF^{\Phi,f}_{K}
\]
 of $\mathbb{Q}_p$-linear symmetric monoidal categories.
\end{proof}

 \subsection{The category of filtered modules as a bilimit}\label{section:bilimit_in_ab_categories}
 While the computation of finite bilimits in the category of stacks is relatively straightforward (cf.\ Proposition~\ref{prop:quotient_and_localization}), computing bicolimits is more involved. They are obtained by first computing the bicolimit in the category of presheaves of groupoids, and then taking the associated stack. Already the first step is problematic, because bicolimits of groupoids are not as easy to compute.

 In order to prove the second half of Theorem~\ref{thm:conjecture_for_adams_stacks} we make use of the fact that the 2-category of Adams stacks is equivalent to a full sub-2-category of the 2-category of right exact symmetric monoidal additive categories and right exact strong symmetric monoidal functors (see Theorem~\ref{thm:adams_embedding}). The problem is thereby reduced to the computation of a certain bilimit in this 2-category.

 Since $\MF_{W,\fp}$ is $\mathbb{Z}_p$-linear, multiplication by $p$ defines an endomorphism for all of its objects. For each $n\in \mathbb{N}$ we can therefore define a right exact $\mathbb{Z}_p$-linear functor
\[
(-)_n \colon \MF_{W,\fp} \rightarrow \MF_{W,n}
\]
 which sends $M$ to $M_n \defl M \slash p^n M$. It is not hard to see that this functor is strong symmetric monoidal. Our goal is therefore to show that the collection of these functors exhibits $\MF_{W,\fp}$ as bilimit of the sequence $\MF_{W,n}$. We do this in two steps. We first show that $\MF_{W,\fp}$ is a bilimit among all $\Ab$-enriched categories. In a second step we will see that it is also a bilimit in the category of right exact symmetric monoidal additive categories. We then use the above mentioned embedding of the 2-category of Adams stacks in the 2-category of right exact symmetric monoidal additive categories to finish the proof of Theorem~\ref{thm:conjecture_for_adams_stacks}.

 We will use the following proposition to compute the bilimit of the sequence $\MF_{W,n}$ in the category of $\Ab$-enriched categories.

 \begin{prop}\label{prop:bilimit_of_ab_categories}
  Let $\ca{A}$ be an $\Ab$-enriched category, and let $\ca{A}_{n} \subseteq \ca{A}$ be a nested sequence of reflective subcategories, with reflections $L_n \colon \ca{A} \rightarrow \ca{A}_n$. Suppose that the following conditions hold:
\begin{enumerate}
 \item[i)] For every object $A$ of $\ca{A}$, the reflections $A \rightarrow L_n A$ exhibit $A$ as limit of the sequence
\[
 \xymatrix{ \ldots \ar[r] & L_n A \ar[r] & L_{n-1} A \ar[r] & \ldots \ar[r] & L_1 A} \smash{\rlap{,}}
\]
 where the morphisms are the ones induced by the universal property of $L_n$;
 \item[ii)]
 The limit $A=\varprojlim A_n$ of every sequence
\[
 \xymatrix{ \ldots \ar[r] & A_n \ar[r] & A_{n-1} \ar[r] & \ldots \ar[r] & A_1}  
\]
 where $A_n \in \ca{A}_n$ and the morphism $A_n \rightarrow A_{n-1}$ is a reflection of $A_n$ into $\ca{A}_{n-1}$ exists in $\ca{A}$, and the projections $A \rightarrow A_n$ are reflections of $A$ into $\ca{A}_n$.
\end{enumerate}
 Then the functors $L_n$ exhibit $\ca{A}$ as bilimit of the $\ca{A}_n$ in the category of $\Ab$-enriched categories.
 \end{prop}

\begin{proof}
 The bilimit in question can be computed as follows. Let $\ca{B}$ be the category whose objects are given by sequences $(A_n,g_n)_{n \in \mathbb{N}}$ where $A_n \in \ca{A}_n$ and
\[
 g_n \colon A_{n} \rightarrow L_{n}(A_{n+1})
\]
 are isomorphisms. The morphisms of $\ca{B}$ are given by sequences $h_i \colon A_n \rightarrow A^{\prime}_n$ which are compatible with $g_n$ and $g_n^{\prime}$. The evident functors
\[
 \ca{B} \rightarrow \ca{A}_n
\]
 exhibit $\ca{B}$ as the desired bilimit. The pseudo-cone on $\ca{A}$ given by the additive functors $L_n$ defines an additive functor
\[
 F \colon \ca{A} \rightarrow \ca{B}
\]
 which sends $A$ to the system $(L_n(A), g_n)$, where $g_i$ is the canonical isomorphism
\[
 L_{n}(A) \rightarrow L_{n}\bigl( L_{n+1} A\bigr)
\]
 induced by the universal property of the reflections. This functor is fully faithful by i), and it is essentially surjective on objects by ii).
\end{proof}

 As example where the above proposition applies one can take $\ca{A}$ to be the category of finitely generated modules of a complete local Noetherian ring, and $\ca{A}_n$ the full subcategory of the modules annihilated by a power of the maximal ideal. Indeed, it is well-known that i) holds in this case, and ii) is the content of the following proposition.

 \begin{prop}\label{prop:inverse_limit_complete_noetherian_ring}
  Let $(A,\mathfrak{m})$ be a complete Noetherian local ring. Let $M_n$ be an $A$-module annihilated by $ \mathfrak{m}^{n}$, and let $f_n \colon M_{n+1} \rightarrow M_{n}$ be surjective homomorphisms of $A$-modules with kernel $\mathfrak{m}^{n} M_{n+1}$. If $M_1$ is finitely generated, then the limit
\[
 M=\varprojlim M_n 
\]
 of the system $(M_n,f_n)_{n\in \mathbb{N}}$ is a finitely generated $A$-module, and the kernels of the surjective homomorphisms
\[
 M \rightarrow M_n
\]
 are given by $\mathfrak{m}^{n} M$.
 \end{prop}

\begin{proof}
 Note that $f_n$ gives a reflection of $M_{n+1}$ into the full subcategory $\Mod_A$ consisting of modules annihilated by $\mathfrak{m}^n$. The composite $M_{\ell} \rightarrow M_n$ obtained from the $f_k$, $n<k\leq \ell$ therefore gives a reflection of $M_{\ell}$ into the subcategory of $\Mod_A$ consisting of modules annihilated by $\mathfrak{m}^n$. Thus its kernel is given by $\mathfrak{m}^n M_{\ell}$. The claim follows from \cite[Proposition~I.0.7.2.9.]{EGAI}.
\end{proof}

 Our goal is therefore to show that the conditions of Proposition~\ref{prop:bilimit_of_ab_categories} are satisfied for the category $\MF_{W,\fp}$ and the sequence of reflective subcategories $\MF_{W,n}$ consisting of objects whose underlying $W$-module is annihilated by $p^n$.

 \begin{lemma}\label{lemma:bilimit_conditions}
  Let $M_n$ be a sequence of objects of $\MF_{W,\fp}$ such that $p^n M_n=0$, and let $f_n \colon M_{n+1} \rightarrow M_n$ be epimorphisms with kernel $p^n M_{n+1}$. Then the limit $M$ of the system $(M_n,f_n)$ exists in $\MF_{W,\fp}$ and is computed as in $\Mod_W$. Moreover, the projections $M \rightarrow M_n$ are epimorphisms with kernel $p^n M$, that is, they are reflections of $M$ into $\MF_{W,n}$.
 \end{lemma}

 \begin{proof} 
 Let $M =\varprojlim M_n$ in the catgory $\Mod_W$. From Proposition~\ref{prop:inverse_limit_complete_noetherian_ring} we know that it is finitely generated. Similarly, we can define $F_i =\varprojlim \Fil^i M_n$. Since limits commute with each other we know that the $F_i$ define a filtration of $M$. Since $M_1$ is a vector space over the residue field $k$ of $W$, we can find a $i_0$ and $i_1$ such that $\Fil^{i_0} M_1=M_1$ and $\Fil^{i_1} M_1=0$. Nakayama's lemma implies that $F_{i_0}=M$ and $F_{i_1}=0$, so the filtration is exhaustive and separated.

 We claim that $F_i$ is a direct summand of $M$. We will prove this by an argument that is almost identical to the one used by Wintenberger to show that $\MF_{W,\fp}$ is abelian. Recall from \cite[Lemme~1.5.3]{WINTENBERGER} that a $W$-submodule $N \subseteq N^{\prime}$ is a direct summand if and only if $N \cap p^k N^{\prime } =p^k N$ for all $k \in \mathbb{N}$. Since $W$ is a discrete valuation ring, any finitely generated $W$-module that is annihilated by a power of $p$ has finite length. Thus any system consisting of such modules satisfies the Mittag-Leffler condition. From this we deduce that
\[
F_i \cap p^k M = \varprojlim \Fil^i M_n \cap p^k M_n
\]
 and
\[
 p^k F_n =\varprojlim p^k \Fil^i M_n \smash{\rlap{.}}
\]
 The claim follows from the fact that the filtrations of the $M_n$ are direct summands (see \cite[Proposition~1.4.1]{WINTENBERGER}). Thus $\Fil^i M \defl F_i$ is an exhaustive separated filtration by direct summands.

 To give $M$ the structure of an object of $\MF_{W,\fp}$ it remains to construct morphisms $\varphi_k \colon \Fil^k M \rightarrow M_{\sigma}$ satisfying the conditions of Definition~\ref{dfn:MF_W}. Denote the colimit of the diagram
\[
 \xymatrix@H=17pt@!R=40pt@!C=15pt{&  \cdots \ar@{_{(}->}[ld] \ar[rd]^{p\cdot} & & \Fil^{k} M \ar@{_{(}->}[ld] \ar[rd]^{p\cdot}  & & \Fil^{k+1} M \ar@{_{(}->}[ld] \ar[rd]^{p\cdot} & & \cdots \ar@{_{(}->}[ld] \ar[rd]^{p\cdot} \\
\cdots && \Fil^{k-1}M && \Fil^{k} M && \Fil^{k+1} M && \cdots
}
\]
 by $\overline{M}$ (cf.\ \cite[Remarques~1.4]{FONTAINE_LAFFAILLE}). To give $\varphi_k$ with the desired properties is equivalent to giving a surjective homomorphism $\varphi \colon \overline{M} \rightarrow M_{\sigma}$ of $W$-modules. Recall that the filtration of $M$ satisfies $\Fil^{i_0} M=M$ and $\Fil^{i_1} M=0$, the above diagram vanishes to the right of $i_0$, and half the morphisms are identities to the left of $i_0$. Thus we can replace it by a finite diagram which has the same colimit. Moreover, all the filtrations of the $M_n$ also satisfy $\Fil^{i_0} M_n=M_n$ and $\Fil^{i_1} M_n=0$, by the same argument we used to show this for $M$. The observation that $\varprojlim$ is exact on the category of systems $(N_i,f_i)$ where the $f_i$ are surjective and the $N_i$ of finite length implies that the canonical morphism
\[
\overline{M} \rightarrow \varprojlim \overline{ M_n} \smash{\rlap{.}}
\]
 is invertible. On the other hand, base change along $\sigma$ is an isomorphism of categories, so it clearly commutes with all limits. Using exactness of $\varprojlim$ again we deduce the existence of the desired surjective homomorphism $\varphi \colon \overline{M} \rightarrow M_{\sigma}$.

 It remains to check that this construction gives a limit in the category $\MF_{W,\fp}$. It is clear from the construction that $(M,\Fil^i M)$ is the limit of $(M_n, \Fil^i M_n)$ in the category of filtered modules. A homomorphism of filtered $W$-modules between two objects $N$, $M$ of $\MF_{W,\fp}$ is a morphism in $\MF_{W,\fp}$ if and only if the diagram
\[
 \xymatrix{\overline{N} \ar[r]^-{\overline{f}} \ar[d]_{\varphi} & \overline{M} \ar[d]^{\varphi} \\ N_\sigma \ar[r]^-{f_{\sigma}} & M_{\sigma} }
\]
 commutes (see \cite[Remarques~1.4]{FONTAINE_LAFFAILLE}). Since $M_{\sigma}$ is a limit of the $(M_n)_{\sigma}$, the above diagram commutes if and only if the outer rectangle of the diagram
\[
 \xymatrix{\overline{N} \ar[r]^-{\overline{f}} \ar[d]_{\varphi}& \overline{M} \ar[r] \ar[d]^{\varphi} & \overline{M_n} \ar[d]^{\varphi} \\ N_\sigma \ar[r]^-{f_{\sigma}} & M_{\sigma} \ar[r] & (M_n)_{\sigma} } 
\]
 commutes for all $n \in \mathbb{N}$. From this it follows easily that $M$ is the desired limit in $\MF_{W,\fp}$.
 \end{proof}

 \begin{cor}\label{cor:MF_W_bilimit_in_ab_categories}
  The reflections $\MF_{W,\fp} \rightarrow \MF_{W,n}$ exhibit $\MF_{W,\fp}$ as bilimit of the sequence $\MF_{W,n}$ in the 2-category of $\Ab$-enriched categories. 
 \end{cor}

 \begin{proof}
  We have to check that the conditions i) and ii) of Proposition~\ref{prop:bilimit_of_ab_categories} are satisfied. The latter is proved in Lemma~\ref{lemma:bilimit_conditions}. From the same lemma we also know that the limit of the sequence from condition i) is computed as in the category of $W$-modules. Thus the comparison morphism
\[
 M \rightarrow \varprojlim L_n(M)
\]
 is invertible as a morphism of $W$-modules. Since the forgetful functor from $\MF_{W,\fp}$ to $\Mod_W$ is faithful and exact, this morphism must be invertible in $\MF_{W,\fp}$ as well. This shows that condition i) of Proposition~\ref{prop:bilimit_of_ab_categories} holds as well.
 \end{proof}

 \subsection{Computing the bicolimit in the 2-category of Adams stacks}\label{section:bicolimit_adams}

 As already mentioned in \S\ref{section:bilimit_in_ab_categories} we need to show that the result from Corollary~\ref{cor:MF_W_bilimit_in_ab_categories} extends to the 2-category $\ca{RM}$ of right exact symmetric monoidal additive categories (see Definition~\ref{dfn:right_exact_symmetric_monoidal}).

\begin{prop}\label{prop:bilimit_of_rex_monoidal_categories}
 The reflections $\MF_{W,\fp} \rightarrow \MF_{W,n}$ exhibit $\MF_{W,\fp}$ as bilimit of the sequence $\MF_{W,n}$ in the 2-category $\ca{RM}$ of right exact symmetric monoidal additive categories.
\end{prop}

\begin{proof}
 It is well-known that limits in categories of sets endowed with algebraic structure and functions which preserve these are easy to compute. A way to make this statement precise is to say that the forgetful functor from a category of algebras of a monad on $\Set$ creates limits.

 A similar fact is true for categories endowed with algebraic structure. The 2-category $\ca{RM}$ is the category of strict algebras and pseudomorphisms of a 2-monad on $\Ab\mbox{-}\Cat$, so the forgetful functor creates bilimits by 2-dimensional monad theory. It remains to check that all the morphisms involved in the diagram are morphisms in $\ca{RM}$. This follows easily from the observation that they are induced by morphisms $(W,\Pi) \rightarrow (W_n,\Pi_n)$ and $(W_{n+1},\Pi_{n+1}) \rightarrow (W_n,\Pi_n)$ of Hopf algebroids (cf.\ Propositions~\ref{prop:MF_W_weakly_tannakian} and \ref{prop:comodules_of_quotient}).
\end{proof}

 We are now ready to finish the proof of Theorem~\ref{thm:conjecture_for_adams_stacks}.

 \begin{proof}[Proof of Theorem~\ref{thm:conjecture_for_adams_stacks}]
 The first half of Theorem~\ref{thm:conjecture_for_adams_stacks} was already proved in Proposition~\ref{prop:first_half_of_conjecture}. It remains to check that the canonical morphisms $\mathfrak{P}_n \rightarrow \mathfrak{P}$ exhibit $\mathfrak{P}$ as bicolimit of the $\mathfrak{P}_n$ in the 2-category of Adams stacks.

 The categories of comodules of the algebraic stacks $\mathfrak{P}$, $\mathfrak{P}_n$, and $\mathfrak{P}_{\mathbb{Q}_p}$ are weakly Tannakian categories, with fiber functors to modules over $W$, $W \slash p^n W$, and $K$ respectively, hence these stacks are coherent and they have the resolution property (see Theorem~\ref{thm:recognition}). From Corollary~\ref{cor:coherent_resolution_property_implies_adams} it follows that they are Adams stacks. 

 In Proposition~\ref{prop:bilimit_of_rex_monoidal_categories} we have seen that $\MF_{W,\fp}$ is the bilimit of the categories $\MF_{W,n}$ in the bicategory $\ca{RM}$ of right exact symmetric monoidal additive categories. We have just shown that they lie in the image of the contravariant pseudofunctor from Adams stacks to $\ca{RM}$ from Theorem~\ref{thm:adams_embedding}. The conclusion follows because this pseudofunctor is an equivalence on hom-categories.
 \end{proof}

%% file: superextensive.tex
\section{Superextensive sites}\label{section:superextensive}

 In this appendix we show that a the stack associated to a pseudofunctor on a superextensive site can be computed in two steps. In the first step one computes the associated stack in the extensive topology, and in the second step one computes the stack associated to the resulting pseudofunctor in the subtopology generated by singleton coverings. This statement appears on the nLab page on superextensive sites (version 6) \cite{NLAB}, but the author was unable to find a proof in the literature.

\subsection{Extensive categories}

 Recall that coproducts are called \emph{disjoint} if the pullback of two distinct coproduct inclusions is initial, and the coproduct inclusions are monomorphisms. 

\begin{dfn}\label{dfn:extensive}
 A category with finite coproducts is called \emph{extensive} if finite coproducts are disjoint, and finite coproducts are stable under pullback.
\end{dfn}

\begin{rmk}
 A category $\ca{C}$ is extensive if and only if for all finite coproducts $U=\coprod U_i$, the coproduct functor
\[
 \coprod \colon \prod \ca{C}\slash U_i \rightarrow \ca{C} \slash U
\]
 is an equivalence of categories.
\end{rmk}

\begin{example}
 The category $\Aff$ of affine schemes is extensive. It is equivalent to the opposite of the category of commutative ring, and finite products of commutative rings are characterized by complete systems of orthogonal idempotents. Using this characterization it is easy to see that the pushout of a homomorphism
\[
 \varphi \colon  A \times B \rightarrow C
\]
 along the projection $A\times B \rightarrow A$ is given by $\varphi(1,0) C$. The desired properties of binary products of rings (which correspond to binary coproducts of affine schemes) follow easily from this description. The zero ring is a strict terminal object because ring homomorphisms preserve both the zero element and the unit element.
\end{example}

\begin{dfn}\label{dfn:superextensive}
 Let $\ca{C}$ be an extensive category. The \emph{extensive topology} on $\ca{C}$ is the Grothendieck topology generated by the finite families consisting of coproduct inclusions into finite coproducts. A Grothendieck topology on $\ca{C}$ is called \emph{(finitary) superextensive} if it is generated by finite families and it contains the extensive topology.
\end{dfn}

\begin{example}
 The $\fpqc$-topology on $\Aff$ is superextensive.
\end{example}

\subsection{Bifinal functors}

 Recall that a functor $F \colon \ca{C} \rightarrow \ca{D}$ is called \emph{final} if for all diagrams $G \colon \ca{D} \rightarrow \ca{X}$ in a category $\ca{X}$, the comparison morphism
\[
 \colim^{\ca{C}} GF \rightarrow \colim^{\ca{D}} G
\]
 is an isomorphism. The formalism of weighted colimits is useful for characterizing such functors. Recall that an ordinary colimit is a weighted colimit for the constant weight $J \colon \ca{D}^{\op} \rightarrow \Set$ at the terminal object. This says that the colimit of $G \colon \ca{D} \rightarrow \ca{X}$ is a representing object for the functor
\[
 [\ca{D}^{\op},\Set]\bigr(J,\ca{X}(G-,X)\bigl)
\]
 on $\ca{X}$. By definition of left Kan extensions we have an isomorphism
\[
 [\ca{C}^{\op},\Set]\bigr(J,\ca{X}(GF-,X)\bigl) \cong [\ca{D}^{\op},\Set]\bigr(\Lan_{F} J, \ca{X}(G-,X)\bigl) 
\]
 which is natural in $X$. Thus a necessary and sufficient condition for $F$ to be final is that $\Lan_F J(d) \cong \ast$ for all $d \in \ca{D}$. From the standard formula for left Kan extensions it follows that this is the case if and only if the slice category $d \slash F$ is connected.

 In order to use this concept to deal with the associated stacks pseudofunctor, we need a bicategorical version of finality.

\begin{dfn}
 A functor $F \colon \ca{C} \rightarrow \ca{D}$ between 1-categories is \emph{bifinal} if for all bicategories $\ca{X}$ and all pseudofunctors $G \colon \ca{D} \rightarrow \ca{X}$, the canonical comparison morphism
\[
 \colim^{\ca{C}} GF \rightarrow \colim^{\ca{D}} G 
\]
 between bicolimits is an equivalence (and either bicolimit exists if and only if the other does).
\end{dfn}

\begin{prop}
 A functor $F \colon \ca{C} \rightarrow \ca{D}$ between 1-categories is bifinal if and only if for all $d \in \ca{D}$, the slice $d \slash F$ is simply connected, that is, $\pi_0 N(d \slash F)=0$ and $\pi_1 N(d \slash F)=0$, where $N$ denotes the nerve functor.
\end{prop}

\begin{proof}
 We follow the proof for the 1-categorical case outlined above. The bicolimit of a pseudofunctor $G \colon \ca{D} \rightarrow \ca{X}$ to a bicategory $\ca{X}$ represents
\[
 \Ps[\ca{D}^{\op},\Cat]\bigl(J, \ca{X}(G-,X) \bigr)
\]
 up to pseudonatural equivalence, where $J$ is the constant functor sending every object of $\ca{D}^{\op}$ to the terminal category. By definition of the bicategorical Kan extension $\Lan^{\prime}_F J$ of $J \colon \ca{C}^{\op} \rightarrow \Cat$ along $F$ we have an equivalence
\[
 \Ps[\ca{C}^{\op},\Cat]\bigr(J,\ca{X}(GF-,X)\bigl) \simeq \Ps[\ca{D}^{\op},\Cat]\bigr(\Lan^{\prime}_{F} J, \ca{X}(G-,X)\bigl) 
\]
 which is pseudonatural in $X$. Thus $F$ is final if and only if $\Lan^{\prime}_F J(d) \simeq \ast$ for all $d \in \ca{D}$.

 Let $\ca{A}$ be any bicategory, and let $F, G \colon \ca{A} \rightarrow \Cat$ be two pseudofunctors. Recall from \cite{STREET_CORRECTIONS} that the weighted bilimit $\{F,G\} \in \Cat$ is equivalent of to the descent object of the diagram
\[
 \xymatrix@C=10pt{ 
   \ar[r] \displaystyle{\prod_{a}} [Fa,Ga]
 & \displaystyle{\prod_{a,a^{\prime}}} 
   [ \ca{A}(a^{\prime},a)\times Fa,Ga^{\prime}]
   \ar@<4pt>[l] \ar@<-4pt>[l]
 & \displaystyle{\prod_{a,a^{\prime},a^{\prime\prime}}}
   [\ca{A}(a^{\prime},a^{\prime\prime}) \times \ca{A}(a^{\prime},a) \times Fa, Ga^{\prime\prime}]
   \ar[l] \ar@<5pt>[l] \ar@<-5pt>[l] }
\]
 whose simplicial identities hold up to isomorphism. From the natural isomorphism $\Cat(\ast, \ca{B}) \cong \ca{B}$ and the definition of weighted bilimits \cite[Formula~(1.17)]{STREET_FIBRATIONS} it follows that $\{F,G\} \simeq \Ps[\ca{A},\Cat](F,G)$. Using this we can adapt the proof of the standard formula for left Kan extensions in the 1-categorical case to show that $\Lan^{\prime}_F J(d)$ is the codescent object of the diagram
\[
 \xymatrix@C=15pt{\displaystyle{\coprod_{c,c^{\prime},c^{\prime\prime}}} \ca{D}(d,Fc)\times \ca{C}(c,c^{\prime}) \times \ca{C}(c^{\prime},c^{\prime\prime}) \ar[r] \ar@<5pt>[r] \ar@<-5pt>[r] & \displaystyle{\coprod_{c,c^{\prime}}} \ca{D}(d,Fc) \times \ca{C}(c,c^{\prime}) \ar@<4pt>[r] \ar@<-4pt>[r] & \ar[l] \displaystyle{\coprod_{c}} \ca{D}(d,Fc)}
\]
 in $\Cat$ (here we used the fact that $J(c)=\ast$ for all $c\in \ca{C}$). The simplical identites for this particular diagram hold strictly because $F$ is a functor between 1-categories.

 Note that the above codescent diagram is the truncation of the nerve of $d \slash F$, considered as a functor $\Delta^{\op} \rightarrow \Set \subseteq \Cat$. The codescent object of a strict codescent diagram $K$ can be computed as the strict weighted 2-colimit $\Delta_{\mathrm{iso}} \star K$, where $\Delta_{\mathrm{iso}}^n$ is the category consisting of a chain of $n$ isomorphisms. Writing
\[
 \lvert -\rvert_{\mathrm{iso}} \colon \Set^{\Delta^{\op}} \rightarrow \Cat
\]
 for the left adjoint induced by $\Delta_{\mathrm{iso}} \colon \Delta \rightarrow \Cat$, we get the formula
\[
 \Lan^{\prime}_F J (d) \simeq  \lvert (d \slash F) \rvert_{\mathrm{iso}}
\]
 for the value of $\Lan^{\prime}_F J$ at the object $d\in \ca{D}$. Note that this left adjoint factors through $\Gpd$, because groupoids form a coreflective subcategory of $\Cat$. We have reduced the problem to showing that $ \lvert -\rvert_{\mathrm{iso}}$ sends simply connected simplicial sets to contractible groupoids. This follows from the fact that
\[
  \lvert \partial \Delta^n \rvert_{\mathrm{iso}} \rightarrow  \lvert \Delta^n \rvert_{\mathrm{iso}}
\]
 is an isomorphism for $n>2$ (cf.\ \cite[\S 6]{LACK_GRAY}).
\end{proof}

 A similar result about final functors for homotopy colimits in topological spaces is proved in \cite[\S XI.9]{BOUSFIELD_KAN} and for model categories in \cite[\S 19.6]{HIRSCHHORN}. In that context the slice categories have to be contractible.

\begin{cor}\label{cor:projections_bifinal}
 Let $\ca{I}_0$ and $\ca{I}_1$ be two filtered categories. Then the two projection functors
\[
 \pr_{i} \colon \ca{I}_0 \times \ca{I}_1 \rightarrow \ca{I}_i
\]
 are bifinal.
\end{cor}

\begin{proof}
 For each object $d \in \ca{I}_i$, the slice category $d \slash \pr_i$ is filtered. The nerve of a filtered category is contractible.
\end{proof}

\begin{cor}\label{cor:filtered_colimit_biproduct}
 Let $F_i \colon \ca{I}_i \rightarrow \ca{X}$, $i=0,1$ be two pseudofunctors on the filtered categories $\ca{I}_i$. If $\ca{X}$ has filtered bicolimits and biproducts, then the comparison morphism
\[
 \colim^{\ca{I}_0\times \ca{I}_1}F_0\times F_1 \rightarrow \colim^{\ca{I}_0} F_0 \times \colim^{\ca{I}_1} F_1
\]
 is an equivalence.
\end{cor}

\begin{proof}
 Let $G_i=F_i \cdot \pr_i \colon \ca{I}_0 \times \ca{I}_1 \rightarrow \ca{X}$. Since finite bilimits commute with filtered bicolimits, the comparison morphism
\[
 \colim^{\ca{I}_0\times \ca{I}_1}G_0 \times G_1 \rightarrow  \colim^{\ca{I}_0\times \ca{I}_1} G_0 \times  \colim^{\ca{I}_0\times \ca{I}_1} G_1
\]
 is an equivalence, where $G_0 \times G_1$ is the pointwise product. The claim follows from bifinality of the projections $\pr_i$ (see Corollary~\ref{cor:projections_bifinal}).
\end{proof}

\subsection{Stacks on superextensive sites}

\begin{prop}\label{prop:extensive_stack}
 Let $\ca{C}$ be an extensive category. Then a pseudofunctor
\[
 \ca{C}^{\op} \rightarrow \Gpd
\]
 is a stack for the extensive topology if and only if it sends finite coproducts in $\ca{C}$ to biproducts.
\end{prop}

\begin{proof}
 The pseudofunctor $F$ satisfies descent for the empty cover of the initial object if and only if it sends the initial object to a biterminal object. Using this and extensivity we find that the descent diagram for the finite cover
\[
\{\mathrm{in}_i \colon U_i \rightarrow \coprod U_i \} 
\]
 is equivalent to the constant descent diagram with value $\prod F(U_i)$.
\end{proof}

 Recall that the stack associated to a pseudofunctor on a site whose topology is generated by singleton coverings is obtained by applying the pseudofunctor $L$ three times, where $LF(U)$ is given by the bicolimit
\[
 \colim^{\Cov(U)^{\op}} \Ps[\ca{C}^{\op},\Gpd](\Er(p),F) \smash{\rlap{,}}
\]
 taken over the preordered set $\Cov(U)$ of coverings of $U$. Here $\Er(p)$ stands for the kernel pair of $p$, considered as an internal groupoid. The objects of $\Cov(U)$ are the coverings $p \colon V \rightarrow U$ of $U$, with $p \leq p^{\prime}$ if there exists a morphism $V \rightarrow V^{\prime}$ making the triangle
\[
 \xymatrix{V \ar[rr] \ar[rd]_{p} && V^{\prime} \ar[ld]^{p^{\prime}} \\ & U }
\]
 commutative. The preordered set $\Cov(U)^{\op}$ is directed because coverings are stable under pullbacks.

\begin{prop}\label{prop:superextensive}
 Let $\ca{C}$ be a superextensive site, and let $F \colon \ca{C} \rightarrow \Gpd$ be a pseudofunctor which is a stack for the extensive topology. Then the associated stack of $F$ for the subtopology generated by singletons is also the associated stack for the superextensive topology.
\end{prop}

\begin{proof}
 Since a superextensive topology is generated by the extensive topology and the subtopology of singleton coverings, a pseudofunctor is a stack for the superextensive topology if and only if it is a stack for both the extensive topology and the singleton topology. Therefore it suffices to show that $L$ preserves stacks for the extensive topology. Thus we have to show that $LF$ sends finite coproducts to biproducts whenever $F$ does.

 Since the initial object $\varnothing$ of $\ca{C}$ is strict, the identity is the only singleton covering of $\varnothing$, and it follows that $LF(\varnothing) \simeq F(\varnothing) \simeq \ast$. This reduces the problem to showing that $LF(U+U^{\prime}) \simeq LF(U) \times LF(U^{\prime})$.

 Since $\ca{C}$ is extensive, any morphism $q \colon W \rightarrow U+U^{\prime}$ is isomorphic to
\[
p + p^{\prime} \colon V + V^{\prime} \rightarrow U+U^{\prime} 
\]
 where $p$ and $p^{\prime}$ are the pullbacks of $q$ along the respective coproduct inclusions. They are in particular coverings if $q$ is a covering. Conversely, if $p$ and $p^{\prime}$ are covering, so is $q$, because the topology is superextensive. Extensivity of $\ca{C}$ implies that the binary coproduct functor
\[
 \Cov(U) \times \Cov(U^{\prime}) \rightarrow \Cov(U+U^{\prime})
\]
 is an equivalence of categories.

 Moreover, using extensivity again we find that the kernel pair $\Er(p+p^{\prime})$ of $p+p^{\prime}$ is the coproduct of the kernel pair $\Er(p)$ of $p$ and $\Er(p^{\prime})$ of $p^{\prime}$. Using the fact that $\Ps[\ca{C}^{\op},\Gpd]\bigr(\Er(p+q),F\bigl)$ is the descent object of
\[
 \xymatrix{F(V+V^{\prime}) \ar[r]  & \ar@<4pt>[l] \ar@<-4pt>[l] F(V\pb{U} V + V^{\prime}\pb{U^{\prime}} V^{\prime}) &  \ar@<5pt>[l] \ar[l] \ar@<-5pt>[l] \ldots }
\]
 (see \cite[p.~12]{STREET_DESCENT}) we find that the pseudofunctor
\[
 \Ps[\ca{C}^{\op},\Gpd]\bigr(\Er(-),F\bigl) \colon \Cov(U+U^{\prime})^{\op} \rightarrow \Gpd
\]
 is equivalent to the composite
\[
 \xymatrix{\Cov(U)^{\op} \times \Cov(U^{\prime})^{\op} \ar[r]^-{G \times G^{\prime}} & \Gpd\times \Gpd  \ar[r]^-{\times} & \Gpd} \smash{\rlap{,}}
\]
 where $G$ and $G^{\prime}$ stand for the pseudofunctors $\Ps[\ca{C}^{\op},\Gpd]\bigl(\Er(-),F \bigr)$ on $\Cov(U)^{\op}$ and $\Cov(U^{\prime})^{\op}$ respectively. From Corollary~\ref{cor:filtered_colimit_biproduct} it follows that the comparison morphism
\[
 LF(U+U^{\prime}) \rightarrow LF(U) \times LF(U^{\prime})
\]
 is an equivalence.
\end{proof}